\documentclass[11pt]{amsart}
\usepackage{amsmath,amssymb,amsthm, amscd}
\usepackage{pb-diagram}
\usepackage{tikz}

\usepackage[left=2.54cm, right=2.54cm, top=2.54cm, bottom=2.54cm]{geometry}

\usepackage{hyperref}
\usepackage{verbatim}
\usepackage{xcolor}

\newtheorem{lem}{Lemma}[section]
\newtheorem{prop}[lem]{Proposition}

\newtheorem{cor}[lem]{Corollary}
\newtheorem{thm}[lem]{Theorem}
\newtheorem{definition}[lem]{Definition}
\theoremstyle{remark}
\newtheorem{remark}[lem]{Remark}
\newtheorem{remarks}[lem]{Remarks}
\newtheorem{ex}[lem]{Example}

\numberwithin{equation}{section}

\newcommand{\af}{\mathrm{af}}
\newcommand{\al}{\alpha}
\newcommand{\alv}[1]{{\alpha^\vee_{#1}}}
\newcommand{\Aut}{\mathrm{Aut}^s}
\newcommand{\av}{{\alpha^\vee}}
\newcommand{\bv}{{\beta^\vee}}
\newcommand{\casetwo}[4]{\left\{ \begin{array}{ll} #1 &\mbox{if $#2$} \\[2mm] #3 &\mbox{if $#4$}\,. \end{array} \right.}
\newcommand{\casetwoc}[4]{\left\{ \begin{array}{ll} #1 &\mbox{if $#2$} \\[2mm] #3 &\mbox{if $#4$}\,, \end{array} \right.}
\newcommand{\cl}{\mathrm{cl}}
\newcommand{\clproj}{\cl}
\newcommand{\dist}{\mathrm{dist}}

\newcommand{\gaf}{\geh_\af}
\newcommand{\geh}{\mathfrak{g}}
\newcommand{\Hom}{\mathrm{Hom}}
\newcommand{\id}{\mathrm{id}}
\newcommand{\I}{I}
\newcommand{\Iaf}{I_\af}
\newcommand{\IP}{J}
\newcommand{\Is}{I^s}
\newcommand{\Inv}{\mathrm{Inv}}
\newcommand{\la}{\lambda}
\newcommand{\La}{\Lambda}
\newcommand{\lev}{\mathrm{level}}
\newcommand{\mcr}[1]{\lfloor #1 \rfloor}
\newcommand{\om}{\omega}
\newcommand{\OmegaP}{\Omega_J}
\newcommand{\omv}[1]{\omega^\vee_{#1}}
\newcommand{\pair}[2]{\langle #1\,,\,#2\rangle}
\newcommand{\Phiaf}{\Phi^\af}
\newcommand{\Phiafm}{\Phi^{\af-}}
\newcommand{\Phiafp}{\Phi^{\af+}}
\newcommand{\PhiafpP}{\Phi^{\af+}_J}
\newcommand{\PhiafmP}{\Phi^{\af-}_J}
\newcommand{\PhiP}{\Phi_J}
\newcommand{\phiP}{\phi_J}
\newcommand{\piP}[1]{\lfloor #1 \rfloor}
\newcommand{\ppiP}{\pi_J}
\newcommand{\XvP}{X^\vee_J}
\newcommand{\qleft}[1]{\xleftarrow{#1}}
\newcommand{\Q}{Q}
\newcommand{\Qaf}{Q_\af}
\newcommand{\Qafv}{Q^\vee_\af}
\newcommand{\QB}{\mathrm{QB}}
\newcommand{\Qv}{Q^\vee}
\newcommand{\QvP}{Q_J^\vee}
\newcommand{\rhoP}{\rho_J}
\newcommand{\SigmaP}{\Sigma_J}
\newcommand{\tal}{{\widetilde{\alpha}}}
\newcommand{\ti}[1]{\widetilde{#1}}
\newcommand{\tv}{{\theta^\vee}}
\newcommand{\W}{W}
\newcommand{\Waf}{W_\af}
\newcommand{\Wafm}{W_\af^-}
\newcommand{\We}{W_e}
\newcommand{\WP}{W_J}
\newcommand{\WPaf}{(W_J)_\af}
\newcommand{\WUP}{W^J}
\newcommand{\WUPaf}{(W^J)_\af}
\newcommand{\X}{X}
\newcommand{\Xaf}{X_\af}
\newcommand{\Xafv}{X_\af^\vee}
\newcommand{\Xafz}{X_\af^0}
\newcommand{\Xv}{X^\vee}
\newcommand{\Z}{\mathbb{Z}}

\newcommand{\wt}{\mathop{\rm wt}\nolimits}

\newcommand{\ql}{q\ell} 
\newcommand{\qlj}{q\ell_{J}}

\newcommand{\sdp}[2]{\ell(#1 \Rightarrow #2)} 

\newcommand{\bp}{{\bf p}}
\newcommand{\bq}{{\bf q}}

\begin{document}

\title[Lifting the parabolic quantum Bruhat graph]{A uniform model for Kirillov-Reshetikhin crystals I:\\
{\small Lifting the parabolic quantum Bruhat graph}}

\author[C.~Lenart]{Cristian Lenart}
\address{Department of Mathematics and Statistics, State University of New York at Albany, 
Albany, NY 12222, U.S.A.}
\email{clenart@albany.edu}
\urladdr{http://www.albany.edu/\~{}lenart/}

\author[S.~Naito]{Satoshi Naito}
\address{Department of Mathematics, Tokyo Institute of Technology,
2-12-1 Oh-Okayama, Meguro-ku, Tokyo 152-8551, Japan}
\email{naito@math.titech.ac.jp}

\author[D.~Sagaki]{Daisuke Sagaki}
\address{Institute of Mathematics, University of Tsukuba, 
Tsukuba, Ibaraki 305-8571, Japan}
\email{sagaki@math.tsukuba.ac.jp}

\author[A.~Schilling]{Anne Schilling}
\address{Department of Mathematics, University of California, One Shields
Avenue, Davis, CA 95616-8633, U.S.A.}
\email{anne@math.ucdavis.edu}
\urladdr{http://www.math.ucdavis.edu/\~{}anne}

\author[M.~Shimozono]{Mark Shimozono}
\address{Department of Mathematics, Virginia Polytechnic Institute
and State University, Blacksburg, VA 24061-0123, U.S.A.}
\email{mshimo@vt.edu}
 


\begin{abstract}
We lift the parabolic quantum Bruhat graph into the Bruhat order on the affine 
Weyl group and into Littelmann's poset on level-zero weights. We establish a quantum analogue of Deodhar's Bruhat-minimum lift from a 
parabolic quotient of the Weyl group. This result asserts a remarkable 
compatibility of the quantum Bruhat graph on the Weyl group, with the cosets 
for every parabolic subgroup. Also, we generalize Postnikov's lemma from the quantum Bruhat graph to the parabolic one; this lemma
compares paths between two vertices in the former graph.

The results in this paper will be applied in a second paper to establish a uniform
construction of tensor products of one-column Kirillov-Reshetikhin (KR) crystals, and the 
equality, for untwisted affine root systems, between the Macdonald polynomial with
$t$ set to zero and the graded character of tensor products of one-column KR modules.
\end{abstract}

\maketitle

\tableofcontents

\section{Introduction}

Our goal in this series of papers is to obtain a uniform construction of tensor products of one-column
Kirillov-Reshetikhin (KR) crystals. As a consequence we shall prove
the equality $P_\lambda(q)=X_\lambda(q)$,
where $P_\lambda(q)$ is the Macdonald polynomial $P_\lambda(q,t)$ specialized at $t=0$
and $X_\lambda(q)$ is the graded character of a simple Lie algebra
coming from tensor products of one-column KR modules.
Both the Macdonald polynomials and KR modules are of arbitrary untwisted affine type.
The parameter $\lambda$ is a dominant weight for the simple Lie subalgebra obtained
by removing the affine node.
Macdonald polynomials and characters of KR modules have been
studied extensively in connection with various fields such as 
statistical mechanics and integrable systems, 
representation theory of Coxeter groups and Lie algebras (and their quantized analogues
given by Hecke algebras and quantized universal enveloping algebras),
geometry of singularities of Schubert varieties, and combinatorics.

Our point of departure is a theorem of Ion~\cite{Ion}, which asserts 
that the nonsymmetric Macdonald polynomials 
at $t=0$ are characters of Demazure submodules of highest weight
modules over affine algebras. This applies for the Langlands duals of 
untwisted affine root systems (and type $A_{2n}^{(2)}$ in the case of 
nonsymmetric Koornwinder polynomials).
Our results apply to the symmetric Macdonald polynomials for the untwisted affine root systems.
The overlapping cases are the simply-laced affine root systems $A_n^{(1)}$,
$D_n^{(1)}$ and $E_{6,7,8}^{(1)}$.

It is known~\cite{FL1,FL2,FSS,KMOU,KMOTU,Na1,Na2,ST} 
that certain affine Demazure characters (including those for the simply-laced affine root systems)
can be expressed in terms of KR crystals, which motivates the relation
between $P$ and $X$. For types $A_n^{(1)}$ and $C_n^{(1)}$, the above mentioned relation between $P$ and $X$
was achieved in~\cite{Le,LeS} by establishing a combinatorial formula for the Macdonald polynomials
at $t=0$ from the Ram--Yip formula~\cite{RY}, and by using explicit models for the one-column KR crystals~\cite{FOS}.
It should be noted that, in types $A_n^{(1)}$ and $C_n^{(1)}$, the one-column KR modules are irreducible
when restricted to the canonical simple Lie subalgebra, while in general this is not the case.
For the cases considered by Ion~\cite{Ion}, the corresponding KR crystals are perfect. This is not
necessarily true for the untwisted affine root systems considered in this work, 
especially for the untwisted non-simply-laced affine root systems.

In this work we provide a type-free approach to the connection between $P$ and $X$ for untwisted
affine root systems. 
Lenart's specialization~\cite{Le} of the Ram--Yip formula for Macdonald polynomials uses paths in the 
quantum Bruhat graph (QBG), which was defined and studied in \cite{BFP} in relation to the
quantum cohomology of the flag variety. On the other hand, Naito and Sagaki~\cite{NS1,NS2,NS3,NS4}
gave models for tensor products of KR crystals of one-column type in terms of projections of level-zero Lakshmibai--Seshadri (LS) paths
to the classical weight lattice.
Hence we need to bridge the gap between these two approaches by establishing
a bijection between paths in the quantum Bruhat graph and projected level-zero LS paths.
For crystal graphs of integrable highest weight modules over quantized universal
enveloping algebras of Kac-Moody algebras,
Lenart and Postnikov had already established a bijection between the LS path model and the alcove model~\cite{LP1}. 
This bijection was refined and reformulated in \cite{LeSh}
using Littelmann's direct characterization of LS paths~\cite{Littelmann1995} and Deodhar's lifting construction
for Coxeter groups~\cite{D}.

In this first paper we set the stage for the connection between the projected level-zero LS path model~\cite{NS1,NS2,NS3,NS4}
and the quantum alcove model~\cite{LeL}.
We begin by establishing a {\em first lift} from the parabolic quantum Bruhat graph (PQBG) to the Bruhat 
order of the affine Weyl group. This is a parabolic analogue of the fact 
that the quantum Bruhat graph (QBG) can be lifted to the affine Bruhat order~\cite{LS}, which is the
combinatorial structure underlying Peterson's
theorem~\cite{P}; the latter equates the Gromov-Witten invariants of finite-dimensional homogeneous spaces $G/P$ 
with the Pontryagin homology structure constants of Schubert varieties in the affine Grassmannian.
We obtain Diamond Lemmas for the PQBG via projection of the
standard Diamond Lemmas for the affine Weyl group.
We find a {\em second lift} of the PQBG into a poset of Littelmann~\cite{Littelmann1995} for level-zero weights
and characterize its local structure (such as cover relations) in terms of the PQBG.
Littelmann's poset was defined in connection with LS paths for arbitrary (not necessarily dominant) 
weights, but the local structure was not previously known.
Then, we prove the tilted Bruhat theorem, which is a quantum Bruhat graph analogue of the
Deodhar lift~\cite{D} for Coxeter groups. This will turn out to be important in our second
paper~\cite{LNSSS}, where we establish the connection between the LS path model and the 
quantum alcove model. Our proof uses the novel notion of quantum length, which relies on the
fact that the PQBG is strongly connected when using only simple transpositions; 
see~\cite{HST}. The theorem ultimately follows from the application of the Diamond Lemmas for the 
QBG. Finally, we prove the natural generalization from the QBG to the PQBG of Postnikov's 
lemma~\cite[Lemma 1\, (2), (3)]{Po}, which compares the weights of two paths in the former graph (the weight measures the down steps);
note that part (1) of Postnikov's lemma, stating the strong connectivity of the QBG, is
generalized earlier in this paper. Besides our second paper on KR crystals, the generalization of Postnikov's lemma
might find applications to parabolic versions of the results in~\cite{Po} on the quantum cohomology of flag varieties.

The paper is organized as follows. In Section~\ref{section.notation} we set up the notation
for untwisted affine root systems and affine Weyl groups.
In Section~\ref{section.orbits} we give the definitions of stabilizers of orbits of the affine Weyl group
and derive properties of $J$-adjusted elements, where $J$ is the index set of a parabolic
subgroup. The PQBG is introduced in Section~\ref{section.QBG}
and the lift to the Bruhat order of the affine Weyl group is given in Section~\ref{section.QBG and affine BO}
(see Proposition~\ref{subsection.embeddings}). This gives rise to the Diamond Lemmas
in Section~\ref{subsection.diamond}. In Section~\ref{section.level 0 WP} we state and prove our
characterization of Littelmann's level-zero weight poset (see Theorem~\ref{theorem.level 0 weight poset})
and show that the PQBG is strongly connected when using only simple
reflections (see Lemma~\ref{lem:ql2}).
In Section~\ref{section.tilted} we prove the tilted Bruhat Theorem (see Theorem~\ref{thm:tilted}).
Finally, in Section~\ref{section.poslem} we prove the parabolic generalization of Postnikov's lemma (see Proposition~\ref{prop:weight}).

Note: After this paper was submitted, we learned of two previous appearances of the regular weight poset
or nonparabolic QBG; see Remark \ref{R:weightposet}. Lusztig's generic Bruhat order on the
affine Weyl group \cite[\S 1.5]{Lusztig} (or equivalently, on alcoves), is isomorphic to the weight poset
when $\la$ is regular. It is shown in \cite{FFKM}
that the containment relation for closures of strata in the space of semi-infinite flags
(or the space of quasimaps from $\mathbb{P}^1$ to $G/B$) is equivalent to the generic Bruhat order;
moreover, the description in \cite[Prop. 5.5]{FFKM} is ultimately equivalent to 
using the QBG, although far less explicit. We presume that analogous results will hold for the PQBG and strata for the space of
quasimaps from $\mathbb{P}^1$ to $G/P$.

\subsection*{Acknowledgments}

The first two and last two authors would like to thank the Mathematisches Forschungsinstitut Oberwolfach for 
their support during the Research in Pairs program, where some of the main ideas of this paper were conceived. 
We would also like to thank Thomas Lam, for helpful discussions during FPSAC 2012 in Nagoya, Japan;
Daniel Orr, for his discussions about Ion's work~\cite{Ion}; and Martina Lanini, for
informing us about Lusztig's order on alcoves.
We used {\sc Sage}~\cite{sage} and {\sc Sage-combinat}~\cite{sagecombinat} to discover properties about the
level-zero weight poset and to obtain some of the pictures in this paper.

C.L. was partially supported by the NSF grant DMS--1101264.
S.N. was supported by Grant-in-Aid for Scientific Research (C), No. 24540010, Japan.
D.S. was supported by Grant-in-Aid for Young Scientists (B) No.23740003, Japan. 
A.S. was partially supported by the NSF grants DMS--1001256 and OCI--1147247.
M.S. was partially supported by the NSF grant DMS-1200804.

\section{Notation}
\label{section.notation}

\subsection{Untwisted affine root datum}
Let $\Iaf=I\sqcup\{0\}$ (resp. $I$) be the Dynkin node set of an untwisted affine algebra $\gaf$
(resp. its canonical subalgebra $\geh$), $(a_{ij}\mid i,j\in \Iaf)$ the affine Cartan matrix, 
$\Xaf = \Z\delta \oplus\bigoplus_{i\in \Iaf} \Z\La_i$ 
(resp. $\X = \bigoplus_{i\in I} \Z \om_i$) the affine (resp. finite) weight lattice,
$\Xafv=\Hom_{\Z}(\Xaf,\Z)$ the dual lattice, 
and $\pair{\cdot}{\cdot}: \Xafv \times \Xaf\to \Z$ the evaluation pairing. 
Let $\Xafv$ have dual basis $\{d\}\cup \{\alv{i} \mid i\in \Iaf \}$.
The natural projection $\cl:\Xaf\to \X$
has kernel $\Z\La_0\oplus\Z\delta$ and sends $\La_i\mapsto\omega_i$ for $i\in I$.

Let $\{\alpha_i\mid i\in \Iaf \}\subset \Xaf$ be the unique elements such that
\begin{align}
  \pair{\alv{i}}{\alpha_j} &= a_{ij}\qquad\text{for $i,j\in \Iaf$} \\
  \pair{d}{\alpha_j} &= \delta_{j,0} .
\end{align}
The affine (resp. finite) root lattice is defined by $\Qaf=\bigoplus_{i\in \Iaf} \Z \alpha_i$
(resp. $\Q = \bigoplus_{i\in I} \Z\alpha_i$). The set of affine real roots (resp. roots)
of $\geh_\af$ (resp. $\geh$) are defined by $\Phiaf = \Waf \,\{\alpha_i\mid i\in \Iaf\}$
(resp. $\Phi=\W\, \{\alpha_i\mid i\in I\}$). The set of positive affine real (resp. positive) roots are the
set $\Phiafp = \Phiaf \cap \bigoplus_{i\in \Iaf} \Z_{\ge0} \alpha_i$
(resp. $\Phi^+=\Phi\cap \bigoplus_{i\in I} \Z_{\ge0} \alpha_i$).
We have $\Phiaf=\Phiafp \sqcup \Phiafm$ where $\Phiafm=-\Phiafp$
and $\Phi=\Phi^+\sqcup\Phi^-$ where $\Phi^-=-\Phi^+$.

The null root $\delta$ is the unique element such that
$\delta\in \bigoplus_{i\in\Iaf}\Z_{>0} \alpha_i$ which generates the rank 1 sublattice
$\{\la\in \Xaf\mid \pair{\alv{i}}{\la}=0 \text{ for all $i\in \Iaf$} \}$. Define $a_i\in \Z_{>0}$ by 
\begin{align}\label{E:a}
\delta=\sum_{i\in \Iaf} a_i \alpha_i.
\end{align}
We have $\delta = \alpha_0 + \theta$, where $\theta$ is the highest root
for $\geh$, and
\begin{align}
\Phiafp = \Phi^+ \sqcup (\Phi+\Z_{>0}\,\delta).
\end{align}

The canonical central element is the unique
element $c\in \bigoplus_{i\in\Iaf} \Z_{>0} \alv{i}$ which generates the rank 1 sublattice
$\{\mu\in\Xafv\mid \pair{\mu}{\alpha_i}=0 \text{ for all $i\in \Iaf$} \}$.
Define $a^\vee_i\in \Z_{>0}$ by $c=\sum_{i\in \Iaf} a^\vee_i \alv{i}$.
Then $a^\vee_0=1$ \cite{Kac}. The level of a weight $\la\in \Xaf$ is defined by
$\lev(\la) = \pair{c}{\la}$. 

\subsection{Affine Weyl group}
Let $\Waf$ (resp. $\W$) be the affine (resp. finite) Weyl group with simple reflections $r_i$ for $i\in \Iaf$
(resp. $i\in I$). $\Waf$ acts on $\Xaf$ and $\Xafv$ by
\begin{align*}
  r_i \la &= \la - \pair{\alpha_i^\vee}{\la} \alpha_i \\
  r_i \mu &= \mu - \pair{\mu}{\alpha_i} \alv{i}
\end{align*}
for $i\in \Iaf$, $\la\in \Xaf$, and $\mu\in \Xafv$. The pairing is $\Waf$-invariant:
\begin{align*}
  \pair{w \mu}{w\la} &= \pair{\mu}{\la}\qquad\text{for $\la\in \Xaf$ and $\mu\in \Xafv$.}
\end{align*}
Since the action of $\Waf$ on $\Xaf$ is level-preserving,
the sublattice $\Xafz\subset\Xaf$ of level-zero elements is $\Waf$-stable.
There is a section $\X\to \Xafz$ given by $\omega_i\mapsto \La_i - \lev(\La_i)\La_0$
for $i\in I$.

For $\beta\in \Phiaf$ let $w\in \Waf$ and $i\in\Iaf$ be such that $\beta = w\cdot \alpha_i$.
Define the associated reflection $r_\beta\in \Waf$ and associated coroot $\beta^\vee\in \Xafv$ by
\begin{align}
\label{E:reflection}
  r_\beta &= w r_i w^{-1} \\
\label{E:coroot}
  \beta^\vee &= w \alv{i}.
\end{align}
Both are independent of $w$ and $i$. Of course $r_{-\beta}=r_\beta$. We have
\begin{align*}
  r_\beta \la &= \la - \pair{\beta^\vee}{\la} \beta &\qquad&\text{for $\la\in \Xaf$} \\
  r_\beta \mu &= \mu - \pair{\mu}{\beta} \beta^\vee &\qquad&\text{for $\mu\in \Xafv$.}
\end{align*}
There is an isomorphism
\begin{align}\label{E:Wafsemi}
  \Waf \cong \W \ltimes \Qv.
\end{align}
Consider the injective group homomorphism $\Qv:=\bigoplus_{i\in \I} \Z\alv{i}\to \Waf$
from the finite coroot lattice into $\Waf$, denoted by $\mu\mapsto t_\mu$.
Then $w t_\mu w^{-1} = t_{w \mu}$ for $w\in \W$. Under the map \eqref{E:Wafsemi},
for $\alpha\in \Phi$ and $n\in\Z$, we have
\begin{align*}
  r_{\alpha+n\delta} &\mapsto r_\alpha t_{n\av} \\
  r_0 &\mapsto r_\theta t_{-\tv}
\end{align*}
the latter holding since $\alpha_0 = \delta-\theta$.

Let $\We = \W \ltimes \Xv$ be the extended affine Weyl group where $\Xv = \bigoplus_{i\in I} \Z\omv{i}$
is the coweight lattice of $\geh$. 
Let $\Is \subset \Iaf$ be the subset of special or cominuscule nodes,
the set of nodes $i\in I^\af$ which are the image of $0$ under some automorphism of the affine
Dynkin diagram. There is a bijection from $\Is$ to $\Xv/\Qv$ given by
$i\mapsto \om_i^\vee+\Qv$ where $\omv{0}:=0$ and $\Qv=\bigoplus_{i\in I}\Z\alv{i}$ is the
finite coroot lattice. 
For each $i\in \Is$ there is a permutation $\tau_i$ of $\Xv/\Qv$ (and therefore a permutation
of $\Is$) defined by adding $-\om_i+\Qv$. The induced permutation of $\Is$
extends uniquely to an automorphism $\tau_i$ of the affine Dynkin diagram.
The group $\Aut(I^\af)$ of special automorphisms is defined to be the group of $\tau_i$ for $i\in \Is$.
It acts on $\Xaf$, $\Xafv$, $\Qaf$, $\Qafv=\bigoplus_{i\in \Iaf}\Z\alv{i}$, and $\Waf$
by permuting $\Iaf$ on basis elements and for $\Waf$, indices of simple reflections.

Define $v_i\in\W$ by the length-additive product
\begin{align}\label{E:v}
  w_0 = v_i w_0^J \qquad\text{for $i\in \Is$}
\end{align}
where $w_0\in \W$ and $w_0^J \in \W_J$ are the longest elements in $\W$ and the subgroup 
$\W_J$ of $\W$ generated by $r_j$ for $j\in J=I\setminus\{i\}$ respectively.
In particular $v_0=\id$. Then there is an injective group homomorphism 
\begin{align*}
  \Aut(I^\af)&\to \We \\
  \tau_i &= v_i t_{-\om_i^\vee}\qquad\text{for $i\in \Is$.}
\end{align*}
$\Aut(I^\af)$ acts on $\Waf$ by conjugation. This action may be defined by relabeling
indices of simple reflections: $\tau r_i \tau^{-1} = r_{\tau(i)}$ for all $\tau\in \Aut(I^\af)$
and $i\in I^\af$. As such we have $\We \cong \Aut(I^\af) \ltimes \Waf$.
There is an injective group homomorphism
\begin{equation}\label{E:specialautomap}
\begin{split}
  \Aut(I^\af) &\to \W \\
  \tau_i &\mapsto v_i .
\end{split}
\end{equation}

\begin{lem}\label{L:cominone} For every $i\in \Is$, $a_i=1$
and $\alpha_i$ occurs in $\theta$
with coefficient $1$ for $i\in \Is\setminus\{0\}$.
\end{lem}
\begin{proof} For untwisted affine algebras $a_0=1$ \cite{Kac}.
The lemma follows since $\Aut(I^\af)$ acts transitively on $\Is$ and fixes $\delta$.
\end{proof}

\begin{lem}\label{L:comin} For every $i\in \Is$
\begin{align}\label{E:autolength}
  \ell(v_i) = \pair{\omega_i^\vee}{2\rho}.
\end{align}
\end{lem}
\begin{proof} Fix $i\in \Is$. Since $\theta$ is the highest root,
it follows from Lemma \ref{L:cominone} that if $\alpha_i$ occurs in a positive root then its 
coefficient is $1$. Consequently the right hand side of \eqref{E:autolength} equals 
the number of positive roots that contain $\alpha_i$. This is the complement of the
number of positive roots in the parabolic subsystem for $J=I\setminus\{i\}$.
But this is equal to $\ell(w_0)-\ell(w_0^J)=\ell(v_i)$.
\end{proof}

\section{Orbits of level-zero weights}
\label{section.orbits}

\subsection{$\Waf$-orbit and $\W$-orbit}
The action of $\Waf$ on $\Xafz$ is given by
\begin{align}\label{E:onlevelzero}
  w t_\mu \la = w \la - \pair{\mu}{\la} \delta
\end{align}
for $w\in \W$, $\mu\in \Qv$, and $\la\in \Xafz$.

\begin{lem} \label{lemma.W mod delta}
For a dominant weight $\la\in \X\cong \Xafz/\Z\delta$ we have
$\Waf \la = \W \la$ in $\Xafz/\Z\delta$.
\end{lem}
\begin{proof}
This follows immediately from \eqref{E:onlevelzero}.
\end{proof}

\subsection{Stabilizers}\label{s:stab}
Let $\la\in \X$ be a dominant weight, which will be used several times in this paper, so the notation below applies throughout.
Let $\WP$ be the stabilizer of $\la$ in $\W$.
It is a parabolic subgroup, being generated by $r_i$ for $i\in \IP$ where
\begin{align}
\label{E:pDynkin}
  \IP &= \{i\in \I\mid \pair{\alv{i}}{\la}=0\}.
\end{align}
Let $\QvP = \bigoplus_{i\in \IP} \Z\alv{i}$ be the associated
coroot lattice, 
$\WUP$ the set of minimum-length coset representatives in $\W/\WP$,
$\PhiP= \PhiP^+\sqcup\PhiP^-$ the set of roots and positive/negative roots respectively,
and $\rhoP = (1/2)\sum_{\alpha\in\PhiP^+} \alpha$.

\begin{lem}\label{L:affstab} 
The stabilizer of $\la$ in $\Waf$ under its level-zero 
action on $\X\cong \Xafz/\Z\delta$, is given by the subgroup of
elements of the form $w t_\mu$ where $w\in \WP$ and $\mu\in \Qv$
satisfies $\pair{\mu}{\lambda}=0$.
\end{lem}
\begin{proof}
This follows immediately from the definitions and Lemma~\ref{lemma.W mod delta}.
\end{proof}

\subsection{Affinization of stabilizer}
Let $\IP= \bigsqcup_{m=1}^k I_m$ have connected components with vertex sets $I_1,I_2,\dotsc,I_k$.
The coweight lattice $\XvP$ is the direct sum $\bigoplus_{m=1}^k \Xv_{I_m}$ where $\Xv_{I_m}$
is the coweight lattice for the root system defined by the component $I_m$.
Define $\IP^\af = \bigsqcup_m I_m^\af$, where $I_m^\af = I_m \sqcup \{0_m\}$
and $0_m$ is a separate additional affine node attached to $I_m$. Define
$\WPaf = \prod_{m=1}^k (W_{I_m})_\af$, where $W_{I_m}$ and $(W_{I_m})_\af$
are the finite and affine Weyl groups for the root subsystem with Dynkin node set $I_m$. Under this
isomorphism $r_{0_m} = r_{\theta_m} t_{-\theta_m^\vee}$ where $\theta_m$ is the highest root for $I_m$.

Define
\begin{align}
\label{E:Paffinepos}
  \PhiafpP &= \{ \beta\in \Phiafp \mid \cl(\beta) \in \PhiP \}= \PhiP^+ \cup (\Z_{>0}\delta + \PhiP)\,,\;\;\;\;  \PhiafmP=-\PhiafpP\,,\\
\label{E:PaffineWpos}
   \WUPaf &= \{x\in \Waf\mid  x \beta > 0 \text{ for all $\beta\in \PhiafpP$}\, \}.
\end{align}

\begin{lem} \cite[Lemma 10.1]{LS} 
$wt_\mu\in \WUPaf$ if and only if, for all $\alpha\in\PhiP^+$, 
$w\alpha > 0$ implies that $\pair{\mu}{\alpha}=0$ and
$w\alpha< 0$ implies that $\pair{\mu}{\alpha} = -1$.
\end{lem}

\begin{prop} \label{P:factor} \cite[Lemma 10.5]{LS} \cite{P} 
Given $w\in \Waf$ there exist unique 
$w_1\in \WUPaf$ and $w_2\in \WPaf$ such that $w = w_1 w_2$. If $w\in W$, then $w_1\in \WUP$ 
is the minimum-length representative of the coset $w \WP$.
\end{prop}

Define $\ppiP:\Waf\to \WUPaf$ by
\begin{align}\label{E:piP}
  w\mapsto w_1\,,
\end{align}
with $w_1$ as in Proposition \ref{P:factor}. Note that for $x\in \Waf$,
$x\in \WUPaf$ if and only if $\ppiP(x)=x$.

Let $\Wafm$ be the set of minimum-length coset representatives in $\Waf/\W$.

\begin{prop} \label{P:piP} \cite[Proposition 10.8]{LS} \cite{P} 
Let $x\in \Waf$ and $\mu\in \Qv$. Then
\begin{enumerate}
\item $\ppiP(x v)=\ppiP(x)$ if $v\in \WPaf$.
\item $\ppiP(\W)\subset \WUP \subset \WUPaf$.
\item $\ppiP(\Wafm)\subset \Wafm$.
\item $\ppiP(x t_\mu) = \ppiP(x) \ppiP(t_\mu)$.
\end{enumerate}
\end{prop}

We shall employ the explicit description of $\ppiP$ in \cite[Lemma 10.7]{LS}.
The element $\mu\in \Qv$ can be written uniquely in the form
\begin{align}\label{E:corootlatticeform}
  \mu = \sum_{i\in \I\setminus\IP} c_i \omv{i} -\phiP(\mu) - \sum_{m=1}^k \omv{j_m}\,,
\end{align}
where $\phiP(\mu) \in \QvP$ and each $j_m\in I_m$ is a cominuscule node.
The element $\mu$ is first separated into the part in $\XvP$ and the part not in it,
and then one considers the projection of the part in $\XvP$
to $\XvP/\QvP$, takes a canonical lift (the last sum).
Then $\phiP(\mu)\in \QvP$ is the correction term.
We write $z_\mu = \prod_{m=1}^k v_{j_m}^{I_m}$
where $v_{j_m}\in \W_{I_m}\subset\WP$ is defined in \eqref{E:v}.
Then for $w\in \W$ and $\mu\in\Qv$ we have
\begin{align} \label{E:piformula}
  \ppiP(w t_\mu) = \ppiP(w) \ppiP(t_\mu) = \ppiP(w) z_\mu t_{\mu+\phiP(\mu)}.
\end{align}

\begin{remark}\label{R:zhom}
By Proposition \ref{P:piP} the map
\begin{equation}\label{E:zmap}
\begin{split}
  \Qv &\to \Aut(\IP^\af) \subset \WP \\
  \mu&\mapsto z_\mu 
\end{split}
\end{equation}
is a group homomorphism. 
\end{remark}

Denote by $\SigmaP\subset \Aut(\IP^\af)\subset \WP$ 
the image of the homomorphism \eqref{E:zmap}:
\begin{align}\label{E:Sigma}
\SigmaP = \{ z\in \WP \mid z=z_\mu\text{ for some $\mu\in\Qv$} \}.
\end{align}

\subsection{$J$-adjusted elements}
We say that $\mu\in \Qv$ is {\em $J$-adjusted} if $\phiP(\mu)=0$ or equivalently
\begin{align}\label{E:adjusted}
  \ppiP(t_\mu) = z_\mu t_\mu.
\end{align}

This notion gives a nice parametrization of the set $\WUPaf$.

\begin{lem} \label{L:WPaf} Let $w\in \WUP$, $z\in \WP$, and $\mu\in \Qv$.
Then $w z t_\mu\in \WUPaf$ if and only if $\mu$ is $J$-adjusted and $z=z_\mu$.
In particular every element of $\WUPaf$ can be uniquely written
as $w \ppiP(t_\mu)=w z_\mu t_\mu$ where $w\in \WUP$ and $\mu\in \Qv$ is $J$-adjusted.
\end{lem}
\begin{proof} 
$w z t_\mu\in \WUPaf$ if and only if $wz t_\mu = \ppiP(wzt_\mu)=\ppiP(wz)\ppiP(t_\mu)=w\ppiP(t_\mu)$
from which the result follows.
\end{proof}

\begin{lem} \label{L:adjusted} 
Let $\mu\in \Qv$ and consider \eqref{E:corootlatticeform}.
The following are equivalent:
\begin{enumerate}
\item[(1)] $\mu$ is $J$-adjusted.
\item[(2)] For every component $I_m$ of $\IP$, either 
\begin{enumerate}
\item[(a)] $\pair{\mu}{\alpha_i}=0$ for all $i\in I_m$ (that is, $j_m=0_m\in I_m^\af$), or 
\item[(b)] there is a unique $j\in I_m$ such that
$\pair{\mu}{\alpha_{j}}\ne0$, and in this case $j=j_m$ and
$\pair{\mu}{\alpha_{j_m}}=-1$. 
\end{enumerate}
\item[(3)] $\pair{\mu}{\alpha} \in \{0,-1\}$ for all $\alpha\in \PhiP^+$.
\end{enumerate}
\end{lem}
\begin{proof} Given \cite[Lemma 10.7]{LS}, (1) and (2) are equivalent. Suppose (2) holds.
Let $\alpha\in \PhiP^+$. Then $\alpha$ is a positive root in the
subrootsystem $\Phi_m^+$ of $\Phi$ for some component $I_m$ of $\IP$. 
Let $\alpha = \sum_{i\in I_m} b_i \alpha_i$. Since $j_m\in I_m$
is cominuscule, $\pair{\omv{j_m}}{\theta_m}=1$
where $\theta_m\in \Phi_m^+$ is the highest root.
Therefore $b_{j_m} \in\{0,1\}$. Since $b_i=0$ for $i\in I_m\setminus \{j_m\}$, (3) follows.

Conversely, suppose (3) holds. Let $I_m$ be a component of $\IP$.
Applying (3) to $\theta_m$ and to each of the $\alpha_i$ for $i\in I_m$, we see that (2) must hold.
\end{proof}

\begin{lem} \label{L:Pinvariant} For $\mu\in \Qv$,
$\mu$ is $\WP$-invariant if and only if $\mu$ is $J$-adjusted and 
$z_\mu=\id$.
\end{lem}
\begin{proof} The first condition holds if and only if no fundamental coweight
$\omv{i}$ occurs in $\mu$ for $i\in \IP$, which for the expression
\eqref{E:corootlatticeform} means that $\phiP(\mu)=0$ and $j_m=0_m$ for all $m$.
But this holds if and only if $\ppiP(t_\mu)=t_\mu$
by \eqref{E:piformula}.
\end{proof}

\begin{lem}\label{L:adjustedroot} Let $\mu\in \Qv$ be $J$-adjusted. Then
\begin{align}\label{E:lengthadjusted}
  \ell(z_\mu) = - \pair{\mu}{2\rhoP}.
\end{align}
\end{lem}
\begin{proof} The proof reduces to considering each component $I_m$ of $\IP$.
Note that $-\mu$ pairs with roots of $I_m$ like a fundamental cominuscule coweight
by Lemma~\ref{L:adjusted} and the result follows by Lemma~\ref{L:comin}.
\end{proof}

\begin{lem} \label{L:zWP} For every $\mu\in \Qv$ and $v\in \WP$,
$z_{\mu}=z_{v\mu}$.
\end{lem}
\begin{proof} 
Using~\eqref{E:piformula} with $w = 1$, we have
$z_{\mu} t_{\mu + \phi_{J}(\mu)} = \ppiP(t_{\mu}) = \ppiP(v t_\mu v^{-1}) = \ppiP(t_{v\mu})$,
which implies the result.
\end{proof}

\begin{lem}\label{L:lalpha} Given $\alpha\in \Phi^+$ and $x=wt_\mu\in \Waf$
with $w\in \W$ and $\mu\in \Qv$, let $\ell_\alpha(x)$ be the number of
roots $\pm \alpha + n\delta\in \Phiafp$ with $n\in\Z$,
which $x$ sends to $\Phiafm$. Then
\begin{align} \label{E:lalpha}
  \ell_\alpha(x) = | \chi(w\alpha\in \Phi^-) + \pair{\mu}{\al}|.
\end{align}
Here $\chi(S)=1$ if $S$ is true and $\chi(S)=0$ if $S$ is false.
\end{lem}
\begin{proof} 
This follows from $x(\pm\alpha+n\delta)=\pm w\alpha + (n-\pair{\mu}{\pm\alpha})\delta$.
\end{proof}

\begin{lem} \label{L:ellP}
Let $w\in \WUP$, $z\in \WP$, and $\mu\in \Qv$ be such that
$\pair{\mu}{\al}<0$ for all $\al\in\Phi^+\setminus\PhiP^+$ 
and $x=wz t_\mu\in \WUPaf$. Then $\mu$ is $J$-adjusted,
$z=z_\mu$, and
\begin{align}
  \ell(x) = - \pair{\mu}{2\rho-2\rhoP} - \ell(w).
\end{align}
\end{lem}
\begin{proof} By Lemma \ref{L:WPaf} we need only prove the length condition.
We have $\ell(x)=\sum_{\alpha\in\Phi^+} \ell_\alpha(x)$.
Fix $\alpha\in \Phi^+$. Since $x\in \WUPaf$, if $\alpha\in \PhiP^+$
then $\ell_\alpha(x)=0$. Let $\alpha\in\Phi^+\setminus \PhiP^+$.
By Lemma \ref{L:lalpha}
we have $\ell_\al(wzt_\mu) = -\chi(wz\al\in\Phi^-) - \pair{\mu}{\al}$.
Summing this over $\al\in \Phi^+\setminus \PhiP^+$ we have
\begin{align*}
  \ell(x) = -\pair{\mu}{2\rho-2\rhoP} + \sum_{\alpha\in \Phi^+\setminus \PhiP^+} -\chi(wz\al\in\Phi^-).
\end{align*}
But $z\in \WP$ so it permutes the set $\Phi^+\setminus \PhiP^+$. Moreover $w\in \WUP$
so $w\PhiP^+\subset \Phi^+$. The lemma follows.
\end{proof}

Let $\mu\in \Qv$. We say that $\mu$ is antidominant if 
\begin{align}
\label{E:antidom}
  \pair{\mu}{\alpha} &\le 0 &\qquad&\text{for all $\alpha\in \Phi^+$.}
\end{align}
Say that $\mu$ is {\em strictly $J$-antidominant} if it is antidominant and
\begin{align}\label{E:strictanti}
  \pair{\mu}{\alpha} &< 0 &\qquad &\text{for $\alpha\in\Phi^+\setminus\PhiP^+$.}
\end{align}
Say that $\mu$ is {\em $J$-superantidominant} if $\mu$ is antidominant and
\begin{align}\label{E:superanti}
  \pair{\mu}{\alpha} &\ll 0 &\qquad &\text{for $\alpha\in\Phi^+\setminus\PhiP^+$.}
\end{align}
In the notation of \eqref{E:corootlatticeform}, the condition
\eqref{E:superanti} means that $c_i\ll0$ for all $i\in \I\setminus \IP$. 

\begin{remark}\label{R:superanti}
If $J=\emptyset$, then 
the $J$-superantidominant property becomes the superantidominant one in \cite{LS}.
If $\mu$ is superantidominant, then \eqref{E:corootlatticeform} and \eqref{E:piformula} show that, in the projection
$\pi_J(t_\mu)=z_\mu t_{\nu}$, the element $\nu$ is $J$-superantidominant.
\end{remark}

\begin{lem} \label{L:zfromadj} Let $z\in \SigmaP$ (see \eqref{E:Sigma}).
Then there is a $J$-superantidominant,
$J$-adjusted element $\mu\in \Qv$ such that $z=z_\mu$.
\end{lem}
\begin{proof} 
By assumption there is a $\nu\in \Qv$ such that $\ppiP(t_\nu) = z t_{\nu+\phiP(\nu)}$.
Since $\gamma=\phiP(\nu)\in\QvP$, by \eqref{E:piformula} we have $\ppiP(t_\gamma)=\id$.
We have $\ppiP(t_{\nu+\gamma})=\ppiP(t_\nu)\ppiP(t_\gamma)=z t_{\nu+\gamma}$
so that $\nu+\gamma$ is a $J$-adjusted element of $\Qv$ with $z_{\nu+\gamma}=z$.
Let $\eta\in \Qv$ be $J$-superantidominant and $\WP$-invariant, so that
$z_\eta = \id$. Then $\nu+\gamma+\eta$ is the required element.
\end{proof}

\begin{lem} \label{L:W-af} Let $w\in \WUP$ and let $\mu\in \Qv$
be $J$-adjusted and strictly $J$-antidominant. Then $wz_\mu t_\mu\in \Wafm$.
\end{lem}
\begin{proof} By \cite[Lemma 3.3]{LS} $wt_\mu\in \Wafm$.
We have $\ppiP(wt_\mu)=\ppiP(w)\ppiP(t_\mu)=w z_\mu t_\mu \in \Wafm$
by Proposition \ref{P:piP}.
\end{proof}

\section{Quantum Bruhat graph}
\label{section.QBG}
The quantum Bruhat graph was first introduced in a paper by Brenti, Fomin and Postnikov~\cite{BFP}
and later appeared in connection with the quantum cohomology of flag varieties in a paper by Fulton 
and Woodward~\cite{FW}. In this section we define the QBG and its parabolic analogue,
and prove some properties we need.

\subsection{Quantum roots}
Say that $\alpha \in \Phi^+$ is a \emph{quantum root}
if $\ell(r_\alpha)=\pair{\av}{2\rho}-1$. 

\begin{lem} \cite[Lemma 4.3]{BFP} \cite[Lemma 3.2]{Mar}
\label{L:reflectionlength}
For any positive root $\alpha \in \Phi^+$, we have
$\ell(r_\alpha) \le -1+ \pair{\av}{2\rho}$. 
In simply-laced type all roots are quantum roots.
\end{lem}

\begin{lem} \label{L:quantumroot} \cite{BMO} 
$\alpha\in \Phi^+$ is a quantum root if and only if
\begin{enumerate}
\item $\alpha$ is a long root, or
\item $\alpha$ is a short root, and writing $\alpha = \sum_i c_i \alpha_i^\vee$, we have
$c_i=0$ for all $i$ such that $\alpha_i$ is long.
\end{enumerate}
Here for simply-laced root systems we consider all roots to be long.
\end{lem}

\subsection{Regular case}
The \emph{quantum Bruhat graph} $\QB(\W)$ is a directed graph structure on
$W$ that contains two kinds of directed edges.
For $w\in W$ there is a directed edge $w\overset{\al}{\longrightarrow} wr_\alpha$ if 
$\alpha\in \Phi^+$ and one of the following holds.
\begin{enumerate}
\item (Bruhat edge) $w\lessdot wr_\al$ is a covering relation in Bruhat order, that is, $\ell(wr_\al)=\ell(w)+1$.
\item (Quantum edge) $\ell(wr_\al)=\ell(w)-\ell(r_\alpha)$ and $\alpha$ is a quantum root.
\end{enumerate}
Condition (2) is equivalent to
\begin{enumerate}
\item[($2^{\prime}$)] $\ell(wr_\al)=\ell(w)+1-\pair{\av}{2\rho}$.
\end{enumerate}
An example is given in Figure~\ref{fig:qB}, where the quantum edges are drawn in red
and $\alpha_{ij} = \alpha_i+\alpha_{i+1} + \cdots + \alpha_{j-1}$.

\begin{figure}
\begin{center}
\includegraphics[scale=0.45]{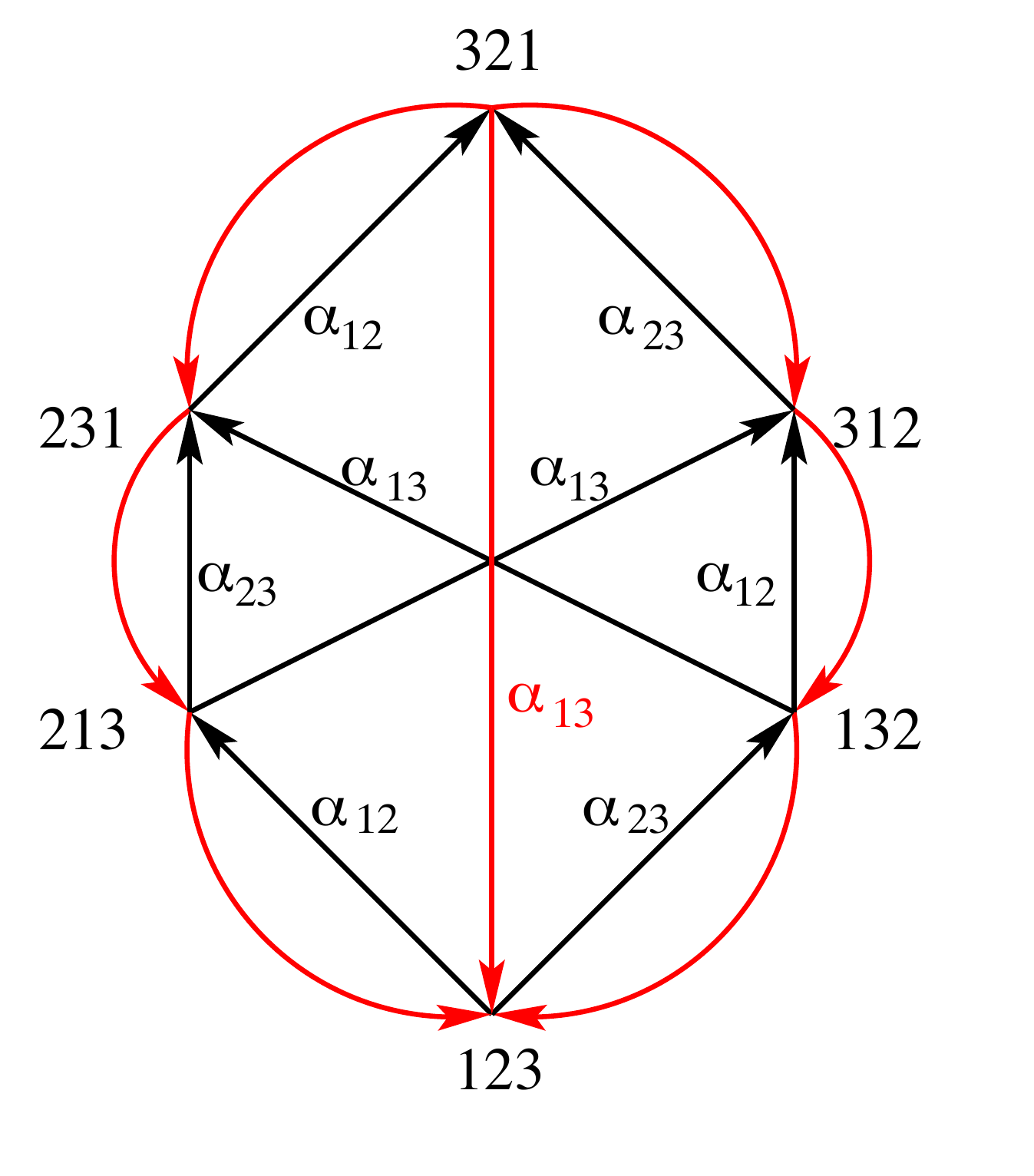}
\end{center}
\caption{Quantum Bruhat graph for $S_3$ \label{fig:qB}}
\end{figure}

\subsection{Parabolic case}
\label{ss:parabolic}
Let $\QB(\WUP)$ be the \emph{parabolic quantum Bruhat graph}. Its vertex set is $\WUP$.
There are two kinds of directed edges.
Both are labeled by some $\alpha\in \Phi^+\setminus \PhiP^+$.
We use the notation $\piP{w}$ to indicate the minimum-length coset representative
in the coset $w \WP$.
\begin{enumerate}
\item (Bruhat edge) $w\overset{\al}\longrightarrow \piP{wr_\alpha}$ where $w\lessdot wr_\al$.
(One may deduce that $wr_\al\in \WUP$.)
\item (Quantum edge) 
\begin{align}\label{E:projqarrow}
\ell(\piP{wr_\al})=\ell(w)+1-\pair{\av}{2\rho-2\rhoP}.
\end{align}
\end{enumerate}
Condition (2) is equivalent to 
\begin{enumerate}
\item[($2^{\prime}$)] $wr_\al \qleft{\al} w$ is a quantum edge in $\QB(\W)$
and $wr_\al t_\av \in \WUPaf$.
\end{enumerate}
This equivalence may be deduced from \cite[Lemma 10.14]{LS} and the proof of
\cite[Theorem 10.18]{LS}. The arguments there rely on geometry, namely, the
quantum Chevalley rule and the Peterson-Woodward comparison theorem.
An example of a PQBG is given in Figure~\ref{fig:pqB}.

We define the \emph{weight} of an edge $w\overset{\al}\longrightarrow \piP{wr_\alpha}$ in the PQBG to be
either $\al^\vee$ or $0$, depending on whether it is a quantum edge or not, respectively. Then the weight of a 
directed path $\bp$, denoted by $\wt(\bp)$, is defined as the sum of the weights of its edges. 

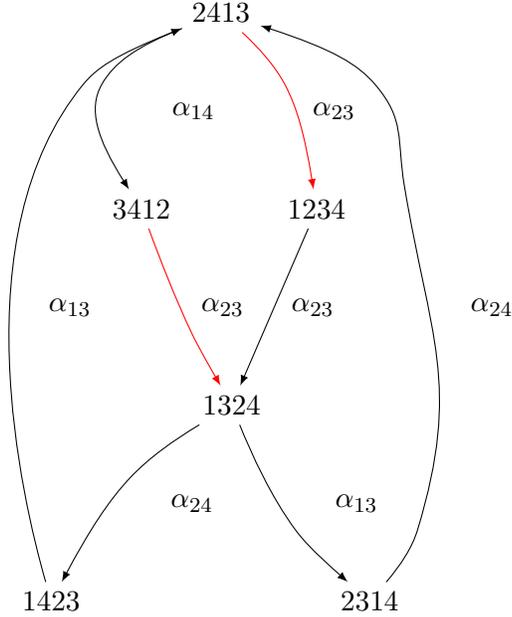
\begin{figure}
\begin{tikzpicture}[>=latex,line join=bevel]
\node (1+2+3+4) at (124bp,157bp) [draw,draw=none] {$1234$};
  \node (1+3+2+4) at (92bp,83bp) [draw,draw=none] {$1324$};
  \node (1+3+4+2) at (24bp,9bp) [draw,draw=none] {$1423$};
  \node (3+1+2+4) at (144bp,9bp) [draw,draw=none] {$2314$};
  \node (3+1+4+2) at (88bp,231bp) [draw,draw=none] {$2413$};
  \node (3+4+1+2) at (58bp,157bp) [draw,draw=none] {$3412$};
  \draw [black,->] (3+1+4+2) ..controls (55.971bp,218.25bp) and (47.668bp,212.25bp)  .. (43bp,204bp) .. controls (37.801bp,194.82bp) and (41.08bp,183.59bp)  .. (3+4+1+2);
  \definecolor{strokecol}{rgb}{0.0,0.0,0.0};
  \pgfsetstrokecolor{strokecol}
  \draw (77.5bp,194bp) node {$\alpha_{14}$};
  \draw [black,->] (1+2+3+4) ..controls (115.2bp,136.65bp) and (106.27bp,116.01bp)  .. (1+3+2+4);
  \draw (122.5bp,120bp) node {$\alpha_{23}$};
  \draw [black,->] (1+3+2+4) ..controls (69.755bp,69.62bp) and (60.806bp,63.153bp)  .. (54bp,56bp) .. controls (45.68bp,47.256bp) and (38.347bp,35.88bp)  .. (1+3+4+2);
  \draw (77bp,46bp) node {$\alpha_{24}$};
  \draw [red,->] (3+4+1+2) ..controls (64.731bp,138.65bp) and (70.778bp,123.02bp)  .. (77bp,110bp) .. controls (78.433bp,107bp) and (80.052bp,103.87bp)  .. (1+3+2+4);
  \draw (88.5bp,120bp) node {$\alpha_{23}$};
  \draw [black,->] (1+3+2+4) ..controls (99.541bp,64.353bp) and (106.94bp,48.219bp)  .. (116bp,36bp) .. controls (118.93bp,32.048bp) and (122.46bp,28.159bp)  .. (3+1+2+4);
  \draw (139bp,46bp) node {$\alpha_{13}$};
  \draw [black,->] (1+3+4+2) ..controls (12.94bp,47.473bp) and (-10.455bp,147.35bp)  .. (36bp,204bp) .. controls (40.922bp,210bp) and (47.531bp,214.83bp)  .. (3+1+4+2);
  \draw (31bp,120bp) node {$\alpha_{13}$};
  \draw [black,->] (3+1+2+4) ..controls (155.69bp,22.77bp) and (159.98bp,29.314bp)  .. (162bp,36bp) .. controls (178.76bp,91.34bp) and (166.42bp,108.95bp)  .. (157bp,166bp) .. controls (154.14bp,183.35bp) and (157.44bp,190.65bp)  .. (146bp,204bp) .. controls (139.51bp,211.57bp) and (130.44bp,217.16bp)  .. (3+1+4+2);
  \draw (190bp,120bp) node {$\alpha_{24}$};
  \draw [red,->] (3+1+4+2) ..controls (102.61bp,217.42bp) and (108.35bp,210.87bp)  .. (112bp,204bp) .. controls (116.61bp,195.32bp) and (119.51bp,184.78bp)  .. (1+2+3+4);
  \draw (130.5bp,194bp) node {$\alpha_{23}$};
\end{tikzpicture}
\caption{Parabolic quantum Bruhat graph for $S_4$ with $J=\{1,3\}$ \label{fig:pqB}}
\end{figure}

\subsection{Duality antiautomorphism of $\QB(\WUP)$}
Let $w_0\in W$ be the longest element. There is an involution on $W$
defined by $w\mapsto w_0 w$. It reverses length in that
$\ell(w_0w)=\ell(w_0)-\ell(w)$. It also reverses Bruhat order in $W$:
$v\lessdot w$ if and only if $w_0 v \gtrdot w_0 w$.
The map $w\mapsto ww_0$ also has the same properties.
In particular $w\mapsto w^*=w_0 w w_0$ is a group automorphism of $W$
which preserves length. Define the involution $*$ on the Dynkin diagram $I$
by $w_0 r_i w_0 = r_{i^*}$ or equivalently $w_0 \alpha_i = -\alpha_{i^*}$.
Then $*$ is an automorphism of $I$. The map $w\mapsto w^*$ can be computed
on reduced words by replacing each $r_i$ by $r_{i^*}$.

Define the map $w\mapsto w^\circ$ on $\WUP$ by 
$w^\circ = \piP{w_0 w}$. Let $w_0^J\in\WP$ be the longest element.

\begin{prop} \label{P:circ} The map $w\mapsto w^\circ$ is an involution on $\WUP$
such that
\begin{enumerate}
\item $w^\circ = w_0 w w_0^J$.
\item $\ell(w^\circ)=\ell(w_0)-\ell(w_0^J)-\ell(w)=|\Phi^+\setminus\PhiP^+|-\ell(w)$.
\item $v \qleft{\beta} w$ is an edge in $\QB(\WUP)$ 
if and only if $w^\circ \qleft{w_0^J\beta} v^\circ$ is an edge in $\QB(\WUP)$.
Moreover both edges are Bruhat or both are quantum.
\end{enumerate}
In particular this involution reverses arrows in $\QB(\WUP)$ and preserves whether an arrow is
quantum or not.
\end{prop}
\begin{proof} For $\alpha\in\PhiP^+$ we have $w_0^J\al\in\PhiP^-$.
Since $w\in \WUP$, $w w_0^J \al \in \Phi^-$. Then $w_0 w w_0^J \al \in \Phi^+$.
Therefore $w_0 w w_0^J \in \WUP$ and $w^\circ = w_0 w w_0^J$. This implies (1).

Since elements of $\WUP$ permute $\PhiP^+$,
$\Inv(w)$ and $\Inv(w^\circ)$ are subsets of $\Phi^+\setminus \PhiP^+$.
Note that for $\al\in\Phi^+\setminus\PhiP^+$,
$\al\in\Inv(w)$ if and only if $w\al\in\Phi^-$, if and only if
$w_0 w w_0^J w_0^J \al \in\Phi^+$, if and only if $w_0^J\al\in \Phi^+\setminus\PhiP^+\setminus \Inv(w^\circ)$.
Therefore the map $\al\mapsto w_0^J\al$ defines a bijection
from $\Inv(w)$ to $(\Phi^+\setminus\PhiP^+)\setminus \Inv(w^\circ)$.
This implies (2).

Let $v=wr_\beta$. Then $v^\circ = w_0 v w_0^J = w_0 w w_0^J w_0^J r_\beta w_0^J = w^\circ r_{w_0^J\beta}$,
that is, $w^\circ = v^\circ r_{w_0^J\beta}$. Since $\beta\in\Phi^+\setminus\PhiP^+$, 
$w_0^J \beta\in\Phi^+\setminus\PhiP^+$. Let $\chi$ be $0$ or $1$ according
as the edge $v\qleft{\beta} w$ is Bruhat or quantum. By (2) we have
\begin{align*}
  \ell(w^\circ) &= |\Phi^+\setminus \PhiP^+| - \ell(w) \\
  &= |\Phi^+\setminus \PhiP^+| - (\ell(v)-1+\chi\pair{\bv}{2\rho-2\rho_J}) \\
  &= \ell(v^\circ) +1 - \chi \pair{\bv}{2\rho-2\rho_J} \\
  &= \ell(v^\circ) +1 - \chi \pair{w_0^J\bv}{2\rho-2\rho_J}
\end{align*}
where the last equality holds by Lemma~\ref{L:Pperm}.
This proves the existence of the required arrow in $\QB(\WUP)$.
\end{proof}

\begin{lem}\label{L:Pperm} 
For any $z\in\WP$,
\begin{align}\label{E:Pperm}
  z (2\rho-2\rho_J) &= 2\rho-2\rho_J.
\end{align}
\end{lem}
\begin{proof} 
$z\in\WP$ permutes the set $\Phi^+\setminus\PhiP^+$, whose sum is $2\rho-2\rho_J$.
\end{proof}

\section{Quantum Bruhat graph and the affine Bruhat order}
\label{section.QBG and affine BO}

In this section we consider the lift of the PQBG to the Bruhat order of the
affine Weyl group (see Theorem~\ref{T:Plift}). This is used in Section~\ref{subsection.diamond}
to establish the Diamond Lemmas for the PQBG.

\subsection{Regular case}
The following result is \cite[Proposition 4.4]{LS}.

\begin{prop} \label{P:lift}
Let $\mu\in \Qv$ be superantidominant and let $x= w t_{v\mu}$ with $w,v\in \W$.
Then $y=x r_{v\alpha+n\delta} \lessdot x$ if and only if one of the following hold.
\begin{enumerate}
\item $\ell(wv)=\ell(wvr_\alpha)-1$ and $n=\pair{\mu}{\alpha}$, giving $y=wr_{v\alpha} t_{v\mu}$.
\item $\ell(wv)=\ell(wvr_{\alpha}) -1 + \pair{\av}{2\rho}$ and $n=1+\pair{\mu}{\alpha}$, giving
$y=wr_{v\alpha} t_{v(\mu+\av)}$.
\item $\ell(v)=\ell(vr_\alpha)+1$ and $n=0$, giving $y=wr_{v\alpha} t_{vr_\alpha\mu}$.
\item $\ell(v)=\ell(vr_\alpha)+1-\pair{\av}{2\rho}$ and $n=-1$ 
giving $y= wr_{v\alpha} t_{vr_\alpha(\mu+\av)}$.
\end{enumerate}
\end{prop}

Note that if we impose the condition that both $x$ and $y$ are in $\Wafm$ then
$v=\id$ and only Cases (1) and (2) apply.

\subsection{Embeddings $\QB(\WUP)\hookrightarrow \Waf$}
\label{subsection.embeddings}
We shall give a parabolic analogue (Theorem~\ref{T:Plift} below) 
of Proposition~\ref{P:lift} for $\Waf^-$.
Theorem~\ref{T:Plift} is proved in the same manner as Proposition~\ref{P:lift}
but the latter cannot be directly invoked to prove the former, since $J$-superantidominance
does not imply superantidominance.

Let $\OmegaP\subset\Waf$ be the subset of 
elements of the form $w\ppiP(t_\mu)$
with $w\in \WUP$ and $\mu\in \Qv$ strictly $J$-antidominant 
(see \eqref{E:strictanti}) and $J$-adjusted. 
Define $\OmegaP^\infty$ similarly but with strict $J$-antidominance replaced by $J$-superantidominance.
We have $\OmegaP^\infty \subset \WUPaf\cap \Wafm$.
Impose the Bruhat covers in $\OmegaP^\infty$ whenever the connecting root 
has classical part in $\Phi\setminus\PhiP$. Then $\OmegaP^\infty$
is a subposet of the Bruhat poset $\Waf$.

\begin{thm} \label{T:Plift} Every edge in $\QB(\WUP)$
lifts to a downward Bruhat cover in $\OmegaP^\infty$, and every
cover in $\OmegaP^\infty$ projects to an edge in $\QB(\WUP)$. More precisely:
\begin{enumerate}
\item 
For any edge $\piP{wr_\al}\qleft{\al} w$ in $\QB(\WUP)$,
$z\in \SigmaP$ (see \eqref{E:Sigma}), and $\mu\in \Qv$ that is $J$-superantidominant
and $J$-adjusted with $z=z_\mu$ (which exists by Lemma {\rm \ref{L:zfromadj}}),
there is a covering relation $y\lessdot x$ in $\OmegaP^\infty$
where 
\[x=wzt_\mu\,,\;\;\;\; y=xr_\tal=wr_\al t_{\chi\av} z t_\mu\,,\;\;\;\;\tal = z^{-1}\al + (\chi + \pair{\mu}{z^{-1}\al})\delta\in \Phiafm\,,\]
and $\chi$ is $0$ or $1$ according as the
arrow in $\QB(\WUP)$ is of Bruhat or quantum type respectively.
\item 
Suppose $y\lessdot x$ is an arbitrary covering relation in $\OmegaP^\infty$.
Then we can write $x=wzt_\mu$ with $w\in \WUP$, $z=z_\mu\in \WP$, and $\mu\in \Qv$ $J$-superantidominant and 
$J$-adjusted, as well as $y=xr_\gamma$ with $\gamma=z^{-1}\al+n\delta\in\Phiaf$, $\al\in\Phi^+\setminus\PhiP^+$, and $n\in\Z$.
With the notation $\chi:=n-\pair{\mu}{z^{-1}\al}$, we have 
\[\chi\in\{0,1\}\,,\;\;\;\;\gamma = z^{-1}\al + (\chi+\pair{\mu}{z^{-1}\al})\delta\in \Phiafm\,;\]
furthermore, there is an edge $wr_\al z\qleft{z^{-1}\al} wz$ in $\QB(\W)$ 
and an edge $\piP{wr_\al} \qleft{\al} w$ in $\QB(\WUP)$, where both edges are of Bruhat
type if $\chi=0$ and of quantum type if $\chi=1$.
\end{enumerate}
\end{thm}

\begin{remark}\label{R:affineroots} The affine Bruhat covering relation considered in part (2)
is completely general, subject to both elements being in $\OmegaP^\infty$
and the transition root having classical part in $\Phi\setminus\PhiP$.
\end{remark}

\begin{proof} (1) Since $\alpha\in \Phi^+\setminus \PhiP^+$
we have $z^{-1}\alpha\in \Phi^+\setminus \PhiP^+$. By
Lemmas \ref{L:WPaf} and \ref{L:W-af}, $x\in \WUPaf\cap \Wafm$.
We have 
\begin{align*}
  y &= xr_\tal = wzt_\mu r_{z^{-1}\al} t_{(\chi+\pair{\mu}{z^{-1}\al})\av} \\
  &= w r_\al z t_{\chi z^{-1}\av} t_\mu \\
  &= w r_\al t_{\chi\av} z t_\mu \\
\ppiP(y) &= \ppiP(w r_\alpha t_{\chi\av} z)\ppiP(t_\mu) \\
&= \ppiP(w r_\alpha t_{\chi\av}) z t_\mu \\
&= w r_\alpha t_{\chi\av} zt_\mu \\
&= y
\end{align*}
using Proposition \ref{P:piP}, the assumption on $\mu$,
and ($2^{\prime}$) of the definition of $\QB(\WUP)$ in the case $\chi=1$.
We conclude that $y\in \WUPaf$. 
Let $i\in I$. We have
$y\al_i = w z r_{z^{-1}_\al} (\al_i - \pair{\mu+\chi z^{-1}\av}{\al_i}\delta)$.
If $i\notin \IP$ then the $J$-superantidominance of $\mu$
implies that  $y\al_i\in\Phiafp$.
Suppose $i\in \IP$. Then $\al_i\in \PhiP^+$ and $y\al_i\in\Phiafp$ by the
definition of $y\in \WUPaf$. We have shown that $y\in \Wafm$.
To prove $x\gtrdot y$ we need only show that $\ell(x)-\ell(y)=1$.
Suppose $\chi=0$. 
Since $y$ and $x$ are in $\Wafm$, by \cite[Lemma 3.3]{LS} we have
\begin{align*}
  \ell(x) - \ell(y) &= \ell(t_\mu) - \ell(wz) - \ell(t_\mu) + \ell(wr_\al z) \\
  &=  -\ell(w)-\ell(z) +\ell(wr_\alpha)+\ell(z) \\
  &= 1.
\end{align*}
Suppose $\chi=1$. We have $x=w\ppiP(t_\mu)$ and 
$y= \piP{wr_\alpha} \ppiP(t_{\mu+z^{-1}\av})$. 
By Lemma \ref{L:ellP} we have
\begin{align*}
\ell(x)-\ell(y) &= -\ell(w) -\pair{\mu}{2\rho-2\rhoP} + \ell(\piP{wr_\alpha}) + \pair{\mu+z^{-1}\al}{2\rho-2\rhoP} \\
&= 1-\pair{\av}{2\rho-2\rhoP}+\pair{\av}{z(2\rho-2\rhoP)}
\end{align*}
by condition (2) of the case $\chi=1$ of the arrow in $\QB(\WUP)$. 
By Lemma \ref{L:Pperm} it follows that $\ell(x)-\ell(y)=1$ as required.

(2) Let $n=\chi+\pair{\mu}{z^{-1}\al}$ where $\chi\in\Z$.

We have $y=w r_\alpha z t_{\mu+\chi z^{-1}\av}$.
Since $y\in \Wafm$, $\mu+\chi z^{-1}\av$
is antidominant by \cite[Lemma 3.3]{LS}. By \cite[Lemma 3.2]{LS}
we have
\begin{align*}
  1 &= \ell(x) - \ell(y) \\
    &= \left(-\pair{\mu}{2\rho}-\ell(wz)\right) - \left(-\pair{\mu+\chi z^{-1}\av}{2\rho} - \ell(wz r_{z^{-1}\al})\right) \\
  &=\ell(w z r_{z^{-1}\al}) -\ell(wz) + \chi \pair{z^{-1}\av}{2\rho} 
\end{align*}
By Lemma \ref{L:reflectionlength}  we deduce that $\chi\in \{0,1\}$.

Suppose $\chi=0$. Then $y=wr_\al z t_\mu$ and
$\ell(w z r_{z^{-1}\al}) -\ell(wz)=1$,
that is, $w z \lessdot w z r_{z^{-1}\al} = w r_\al z$. This gives the
required Bruhat cover in $\QB(\W)$. 
Since $y\in \WUPaf$ we have $\ppiP(y)=y$ and
$wr_\al z t_\mu = \piP{wr_\al z} \ppiP(t_\mu)=\piP{wr_\al} z t_\mu$
using Proposition~\ref{P:piP}. We deduce that $wr_\al\in \WUP$. By length-additivity
it follows that $wr_\al\qleft{\al} w$ is a Bruhat arrow in $\QB(\WUP)$.

Otherwise we have $\chi=1$. Then $y=wr_\al t_\av z t_\mu=wr_\al z t_{\mu+z^{-1}\av}$
and $\ell(w z r_{z^{-1}\al})=\ell(wz)+1-\pair{z^{-1}\av}{2\rho}$,
which yields the required quantum arrow in $\QB(\W)$. 

Since $y\in \WUPaf$ we have
\begin{align*}
  wr_\al t_\av z t_\mu &= y = \pi(y) = \ppiP(wr_\al t_{\av} z)\ppiP(t_\mu) \\
  &= \ppiP(wr_\al t_{\av}) z t_\mu
\end{align*}
from which we deduce that $wr_\al t_\av\in \WUPaf$ and that $\av$ is $J$-adjusted.

By Remark \ref{R:zhom} and Lemma \ref{L:zWP} we have
$z_{\mu+z^{-1}\av} = z_\mu z_{z^{-1}\av} = z_\mu z_{\av} = z z_\av$.
Since $\av$ is $J$-adjusted we have
$wr_\al z = \pi_J(wr_\al) z_\av z=\pi_J(wr_\al) z_{\mu+z^{-1}\av}$
and the last product is length-additive. Therefore
\begin{align*}
  \ell(\pi_J(wr_\al)) &= \ell(wr_\al z) - \ell(z_{\mu+z^{-1}\av}) \\
  &= \ell(wz)+1-\pair{z^{-1}\av}{2\rho} - \ell(z_{\mu+z^{-1}\av}) \\
  &= \ell(w) +1+ \ell(z_\mu)  - \ell(z_{\mu+z^{-1}\av})-\pair{z^{-1}\av}{2\rho} \\
  &= \ell(w) +1+\pair{z^{-1}\av}{2\rho_J} -\pair{z^{-1}\av}{2\rho} \\
  &= \ell(w) + 1 - \pair{z^{-1}\av}{2\rho-2\rho_J} \\
  &= \ell(w) + 1 - \pair{\av}{2\rho-2\rho_J}
\end{align*}
using Lemma \ref{L:adjustedroot}, that $\mu$ and $\mu+z^{-1}\av$ are $J$-adjusted,
and Lemma \ref{L:Pperm}. This proves the existence of the required edge in $\QB(\WUP)$.
\end{proof}

\begin{ex}\label{X:A2} Let $\geh$ be of type $A_2$ and $\IP=\{1\}$.
Then $\QB(\WUP)$ is given by
\begin{align*}
\begin{diagram}
\node[2]{r_1r_2} \arrow[2]{s,b,..}{\al_2} \\
\node{r_2} \arrow{ne,t}{\al_1+\al_2} \\
\node[2]{\id} \arrow{nw,b}{\al_2}
\end{diagram}
\end{align*}
where the quantum arrow is dotted. In $\OmegaP\subset\Waf$,
let $\mu=-6\omv{2}$ and $\nu=-3\omv{2}-\tv$. We have
\begin{align*}
\begin{diagram}
\node[2]{t_\mu=t_{-6\omv{2}}} \arrow{sw,t}{6\delta-\al_2} \\
\node {r_2t_\mu} \arrow{se,b}{6\delta-\al_1-\al_2} \\
\node[2]{r_1r_2t_\mu} \arrow{s,t,..}{5\delta-\al_2} \\
\node[2]{x=r_0r_1r_2t_\mu}
\end{diagram}
\qquad
\begin{diagram}
\node[2]{x=r_1 t_\nu} \arrow{sw,t}{5\delta-r_1(\al_2)} \\
\node {r_2 (r_1 t_\nu)} \arrow{se,b}{4\delta-r_1(\al_1+\al_2)} \\
\node[2]{r_1r_2(r_1 t_\nu)} \arrow{s,t,..}{4\delta-r_1(\al_2)} \\
\node[2]{r_0r_1r_2(r_1t_\nu)=t_{-3\omv{2}}}
\end{diagram}
\end{align*}
We have a single chain running from $t_{-6\omv{2}}$
down to $t_{-3\omv{2}}$. The diagram is broken at
$t_{-\nu}$, which appears at the bottom on the left and the top
on the right. If the bottom element is removed from each side
then one obtains an upside-down copy of $\QB(\WUP)$. In this case
the quantum arrows transition to a different copy of $\QB(\WUP)$.
The left hand copy has $z=\id$
and the right hand copy has $z=r_1$ where in this situation $\SigmaP$
is generated by $r_1$. The poset $\OmegaP^\infty$ is an infinite chain
that wraps down onto the $3$-cycle given by $\QB(\WUP)$ with two
flavors of lifts, one for $z=\id$ and the other for $z=r_1$.

{\bf Warning:} generally not every affine cover is produced by left
multiplication by a simple reflection, 
nor is a general quantum cover always induced by
left multiplication by $r_0$ (although we shall see that
left multiplication by $r_0$ always induces a quantum arrow).
\end{ex}

We say that a walk in the directed graph $\QB(\WUP)$ is
locally-shortest if any segment of the walk not containing a repeated
vertex is a shortest path.

\begin{cor}\label{C:PQBG2Waf} Downward saturated chains in $\OmegaP^\infty$
project to locally-shortest walks in $\QB(\WUP)$. Conversely, shortest paths
in $\QB(\WUP)$ are projections of downward saturated chains in $\OmegaP^\infty$. 
\end{cor}
\begin{proof} Say $x_0\gtrdot x_1\gtrdot \dotsm\gtrdot x_N$
is a saturated Bruhat chain in $\OmegaP^\infty$. Let $\ppiP(x_i)=w_i z_i t_{\mu_i}$
where $w_i\in \WUP$, $z_i\in\WP$, and $\mu_i\in \Qv$. Then Theorem~\ref{T:Plift}
asserts that $w_0\to w_1\to\dotsm\to w_N$ is a locally-shortest walk in $\QB(\WUP)$.

Now let $u=u_0\to u_1\to\dotsm\to u_N=u'$ be a shortest path in $\QB(\WUP)$.
We apply Theorem~\ref{T:Plift} to the edge $u_0\to u_1$
with $\mu=\mu_0$ $J$-superantidominant and $\WP$-invariant. 
The element $x_0 = u_0 t_\mu$ lifts $u_0$ since $\ppiP(t_\mu)=t_\mu$.
Then the Proposition produces a cocover $x_1 = u_1 z_1 t_{\mu_1}$
of $x_0$ with $z_1\in\SigmaP$. In general we have a descending Bruhat chain
$x_0\gtrdot x_1\gtrdot\dotsm\gtrdot x_{i-1} = u_{i-1} z_{i-1} t_{\mu_{i-1}}$
with $z_{i-1}\in \SigmaP$ and we apply the Proposition to obtain a cocover
$x_i=u_i z_i t_{\mu_i}$ of $x_{i-1}$ with $z_i\in \SigmaP$ and by induction the
required affine chain is produced.
\end{proof}

\begin{cor} \label{C:PQBGintoBQBG} For each $z\in\SigmaP$ there
is a copy of $\QB(\WUP)$ inside $\QB(\W)$,
embedded by $w\mapsto wz$ such that the edge label $\al$
is sent to the root $z^{-1}\al$, and 
Bruhat and quantum edges are sent to the same kind of edge.
\end{cor}
\begin{proof} For every $z\in\SigmaP$, we take an edge $\piP{wr_\al}\qleft{\al} w$
in $\QB(\WUP)$, lift it to $w z t_\mu \gtrdot w r_\al z t_\nu$ for some $\nu$,
and project to an edge $wr_\al z\qleft{z^{-1}\al} wz$ in $\QB(\W)$; the lift is based on 
Theorem~\ref{T:Plift}~(1), and the projection on Theorem~\ref{T:Plift}~(2). 
\end{proof}

\begin{remark}\label{R:lifting} Lifting quantum edges causes a ``phase shift"
by an element $z\in\SigmaP$. Theorem~\ref{T:Plift} is just general enough
to lift in the presence of such a shift. If one tries to twist by a $z\in\WP$ that is not
in $\SigmaP$ then the affine element of the form $x=w z t_\mu$
no longer lies in the set $\WUPaf$ and lifting the edge of $\QB(\WUP)$
starting from $x$ is not possible in general.
\end{remark}

\subsection{Trichotomy of cosets}

\begin{lem}\label{L:tricoset} \cite{D} 
Let $\W$ be a Weyl group, $\WP\subset\W$ a parabolic subgroup,
$v\in \WUP$ and $r\in\W$ a simple reflection. Then one of the following holds.
\begin{enumerate}
\item If $r v < v$ then $r v\in \WUP$ and $rv\WP<v\WP$.
\item If $r v > v$ and $v^{-1} r v\in\WP$ then $rv\WP=v\WP$.
\item If $r v > v$ and $v^{-1} r v\not\in\WP$ then $rv\in\WUP$ and $rv\WP>v\WP$.
\end{enumerate}
\end{lem}

\begin{lem}\label{L:tri} 
Let $v\in \W$ and $\alpha\in \Phi^+$. Let $\la\in \X$ be a dominant weight (cf. Section {\rm \ref{s:stab}} and the notation thereof, e.g., $W_J$ is the stabilizer of $\lambda$).
\begin{enumerate}
\item[(1)] Let $\pair{\av}{v\lambda}<0$. Then $v^{-1}\alpha\in\Phi^-\setminus\PhiP^-$
and $r_\alpha v \WP < v \WP$.
\item[(2)] Let $\pair{\av}{v\lambda}=0$. Then $v^{-1}\alpha\in \PhiP$ and $r_\alpha v \WP =v\WP$.
\item[(3)] Let $\pair{\av}{v\lambda}>0$. Then $v^{-1}\alpha \in \Phi^+\setminus \PhiP^+$
and $r_\alpha v \WP > v \WP$.
\end{enumerate}
\end{lem}
The proof of the above lemma is easy using standard techniques for Weyl groups
(see for example~\cite[Proposition 2.5.1]{BB}).

\subsection{Quantum edges induced by left multiplication by suitable reflections}

\begin{prop}\label{P:leftrPQBG} \cite{D}
Let $w\in \WUP$ and $j\in I$. Then exactly one of the following holds.
\begin{enumerate}
\item[(1)] $w^{-1}\al_j\in\Phi^-\setminus \PhiP^-$. In this case $r_jw\in \WUP$ and
there is a Bruhat edge $w \qleft{-w^{-1}\al} r_j w$ in $\QB(\WUP)$.
\item[(2)] $w^{-1}\al_j\in \PhiP$. In this case $w^{-1}\al_j\in\PhiP^+$ and $\piP{r_j w}=w$.
\item[(3)]
$w^{-1}\al_j\in \Phi^+\setminus\PhiP^+$. In this case $r_jw\in \WUP$ and
there is a Bruhat edge $r_jw \qleft{w^{-1}\al_j} w$ in $\QB(\WUP)$.
\end{enumerate}
\end{prop}

\begin{prop}\label{P:thetaPQBG} 
Let $w\in \WUP$. Then exactly one of the following holds.
\begin{enumerate}
\item[(1)] $w^{-1}\theta\in\Phi^-\setminus \PhiP^-$. In this case
there is an edge $\piP{r_\theta w} \qleft{-w^{-1}\theta} w$ of quantum type in $\QB(\WUP)$.
\item[(2)] $w^{-1}\theta\in \PhiP$. In this case $w^{-1}\theta\in\PhiP^+$ and $\piP{r_\theta w}=w$.
\item[(3)]
$w^{-1}\theta\in \Phi^+\setminus\PhiP^+$. In this case
there is an edge $w \qleft{z w^{-1}\theta} \piP{r_\theta w}$ of quantum type in 
$\QB(\WUP)$, where $z\in \WP$ is defined by $r_{\theta} w = \lfloor r_{\theta} w \rfloor z$,
so that $z=(z_{w^{-1}\tv})^{-1}$.
\end{enumerate}
\end{prop}
\begin{proof} The three cases correspond to those in Lemma \ref{L:tri} with $v=w$ and $\alpha=\theta$.
The conclusion in Case (2) is immediate from the definition of $w\in \WUP$. 
By exchanging the roles of $w$ and $\piP{r_\theta w}$,
it suffices to prove the existence of the edge in (1). Let $w^{-1}\theta\in \Phi^-\setminus\PhiP^-$.

Choose any $\mu\in \Qv$ that is $J$-superantidominant and $\WP$-invariant.
We have $\ppiP(t_\mu)=t_\mu$ and $x:=w t_\mu \in \Wafm \cap \WUPaf$. We have 
\begin{align}\label{E:xinvalpha0}
x^{-1}\alpha_0 &= -w^{-1}\theta +(1+\pair{\mu}{-w^{-1}\theta})\delta \in \Phiafm\,,
\end{align}
since $w^{-1}\theta\in\Phi^-$. We conclude that
$x > y:= r_0 x = r_\theta w t_{\mu-w^{-1}\tv}$. 
Since $x\in \Wafm$ it follows that $y\in \Wafm$ as well.
Let $\beta\in \Phiafp_J$. Suppose $y\beta \in \Phiafm$.
Since $x\in \WUPaf$ we have $x\beta=r_0 y\beta \in \Phiafp$.
Since $r_0$ has the sole inversion $\alpha_0$, it follows that
$x\beta=\alpha_0$ or $x^{-1}\alpha_0 = \beta\in \Phiafp_J$.
But this contradicts \eqref{E:xinvalpha0}.
Therefore $y\in \WUPaf$. 

By Theorem~\ref{T:Plift}, the required quantum edge exists in $\QB(\WUP)$.
\end{proof}

\begin{cor} \label{C:thetaadj} 
For every $\gamma = w^{-1} \theta \in W \theta \cap (\Phi^{+} \setminus \Phi^{+}_{J})$,
$z \gamma^{\vee}$ is $J$-adjusted, where $z \in W_{J}$ is defined, as in part (3)
of Proposition~\ref{P:thetaPQBG}, by $r_{\theta} w = \lfloor r_{\theta} w \rfloor z$.
In addition, if $w \in W^{J}$, then $z = (z_{w^{-1} \theta^{\vee}})^{-1}$.
Also, for every $\gamma = w^{-1} \theta \in W \theta \cap (\Phi^{-} \setminus \Phi^{-}_{J})$,
$- z^{\prime} \gamma^{\vee}$ is $J$-adjusted, where $z^{\prime}$ is the element of
 $W_{J}$ defined by $w = \lfloor w \rfloor  z^{\prime}$.
(Obviously, if $w \in W^{J}$, then $z^{\prime} = 1$.)
\end{cor}
\begin{proof} 
Follows from the existence of the edges in $\QB(\WUP)$.
\end{proof}

\subsection{Diamond Lemmas for $\QB(\WUP)$}
\label{subsection.diamond}
We recall the Diamond Lemma for Coxeter groups and the Bruhat order.

\begin{lem}\label{L:diamond} \cite{H} Let $\W$ be any Coxeter group,
$v,w\in \W$, and $r$ a simple reflection.
\begin{enumerate}
\item Suppose $v\lessdot w$, $rw<w$ and $v\ne rw$. Then
$rv<v$ and $rv\lessdot rw$.
\item Suppose $v \gtrdot w$, $rw>w$ and $v\ne rw$. Then $rv>v$
and $rv\gtrdot rw$.
\end{enumerate}
\end{lem}

In the following diagrams, a dotted (resp. plain) edge represents a quantum (resp. Bruhat) edge
in $\QB(\WUP)$.
We always refer to the PQBG on $W^J$. Given $w\in W$ and 
$\gamma\in\Phi^+$, define $z,z'\in W_J$ by
\begin{equation}\label{zz}
	r_\theta \lfloor w\rfloor = \lfloor r_\theta w \rfloor z\,,\qquad r_\theta \lfloor w r_\gamma\rfloor 
	= \lfloor r_\theta \lfloor w r_\gamma \rfloor \rfloor z'
	= \lfloor r_\theta w r_\gamma \rfloor z'\,.
\end{equation}

We are now ready to state the Diamond Lemmas for the PQBG.

\begin{lem} \label{lemma.diamond.PQBG}
Let $\alpha$ be a simple root in $\Phi$ , $\gamma \in \Phi^+\setminus\Phi_J^+$, 
and $w\in W^J$. Then we have the following cases, in each of which the bottom two edges imply the top two edges in the left diagram, and the top two edges imply the bottom two edges in the right diagram. Moreover, in each diagram the weights of the two directed paths 
(on the left side and the right side) are congruent modulo $Q_{J}^{\vee}$.
\begin{enumerate}
\item \label{case.1}
In both cases we assume $\gamma\ne w^{-1}\alpha$ and have
$r_\alpha \lfloor w r_\gamma\rfloor=r_\alpha w r_\gamma = \lfloor r_\alpha  w r_\gamma\rfloor$.
\begin{equation}\label{dq1}
	\begin{diagram}\node[2]{r_\alpha \lfloor w r_\gamma\rfloor}\\
	\node{r_\alpha w} \arrow{ne,t}{\gamma} 
	\node[2]{\lfloor w r_\gamma\rfloor} \arrow{nw,t}{\lfloor w r_\gamma\rfloor^{-1}\alpha} \\\node[2]{w} \arrow{nw,b}{w^{-1}\alpha} \arrow{ne,b}{\gamma}\end{diagram}\qquad \qquad
	\begin{diagram}\node[2]{\lfloor w r_\gamma\rfloor}\\
	\node{w} \arrow{ne,t}{\gamma} 
	\node[2]{r_\alpha \lfloor w r_\gamma\rfloor} \arrow{nw,t}{-\lfloor w r_\gamma\rfloor^{-1}\alpha} \\\node[2]{r_\alpha w} \arrow{nw,b}{-w^{-1}\alpha} \arrow{ne,b}{\gamma}\end{diagram}
\end{equation}
\item \label{case.2}
Here we have $r_\alpha \lfloor w r_\gamma\rfloor=\lfloor r_\alpha  w r_\gamma\rfloor$ in both cases. 
\begin{equation}\label{dq2}
	\begin{diagram}\node[2]{r_\alpha \lfloor w r_\gamma \rfloor}\\
	\node{r_\alpha w} \arrow{ne,t,..}{\gamma}
	\node[2]{\lfloor w r_\gamma\rfloor} \arrow{nw,t}{\lfloor w r_\gamma \rfloor^{-1}\alpha} \\\node[2]{w} \arrow{nw,b}{w^{-1}\alpha} \arrow{ne,b,..}{\gamma}\end{diagram}
\qquad  \qquad
	\begin{diagram}\node[2]{\lfloor w r_\gamma \rfloor}\\
	\node{w} \arrow{ne,t,..}{\gamma}
	\node[2]{r_\alpha \lfloor w r_\gamma\rfloor} \arrow{nw,t}{-\lfloor w r_\gamma \rfloor^{-1}\alpha} \\\node[2]{r_\alpha w} \arrow{nw,b}{-w^{-1}\alpha} \arrow{ne,b,..}	{\gamma}\end{diagram}
\end{equation}
\item \label{case.3} 
Here $z,z'$ are defined as in \eqref{zz}. In subcase \eqref{dq3} (resp. \eqref{dq31}) we
assume that $\langle \gamma^\vee,\,w^{-1}\theta\rangle$ is nonzero (resp. zero).
In all cases, we have $w r_{\gamma}=\lfloor w r_{\gamma} \rfloor$.
\begin{equation}\label{dq3}
	\begin{diagram}\node[2]{\lfloor r_\theta w r_\gamma \rfloor}\\
	\node{\lfloor r_\theta w \rfloor} \arrow{ne,t,..}{z\gamma}
	\node[2]{\lfloor w r_\gamma \rfloor} \arrow{nw,t,..}{- \lfloor w r_\gamma \rfloor^{-1}\theta} \\ \node[2]{w} \arrow{nw,b,..}{-w^{-1}\theta} \arrow{ne,b}{\gamma}\end{diagram}
\qquad \qquad
	\begin{diagram}\node[2]{\lfloor w r_\gamma \rfloor}\\
	\node{w } \arrow{ne,t}{\gamma}
	\node[2]{\lfloor r_\theta w r_\gamma\rfloor} \arrow{nw,t,..}{z' \lfloor w r_\gamma\rfloor^{-1}\theta} \\ \node[2]{\lfloor r_\theta w\rfloor} \arrow{nw,b,..}{zw^{-1}\theta} \arrow{ne,b,..}
	{z\gamma}\end{diagram}
\end{equation}
\begin{equation}\label{dq31}
	\begin{diagram}\node[2]{\lfloor r_\theta w r_\gamma \rfloor}\\
	\node{\lfloor r_\theta w \rfloor} \arrow{ne,t}{z\gamma}
	\node[2]{\lfloor w r_\gamma \rfloor} \arrow{nw,t,..}{- \lfloor w r_\gamma \rfloor^{-1}\theta} \\ \node[2]{w} \arrow{nw,b,..}{-w^{-1}\theta} \arrow{ne,b}{\gamma}\end{diagram}
\qquad \qquad
	\begin{diagram}\node[2]{\lfloor w r_\gamma \rfloor}\\
	\node{w } \arrow{ne,t}{\gamma}
	\node[2]{\lfloor r_\theta w r_\gamma\rfloor} \arrow{nw,t,..}{z' \lfloor w r_\gamma\rfloor^{-1}\theta} \\ \node[2]{\lfloor r_\theta w\rfloor} \arrow{nw,b,..}{zw^{-1}\theta} \arrow{ne,b}
	{z\gamma}\end{diagram}
\end{equation}
\item \label{case.4} 
Here we assume $\gamma\ne -w^{-1}\theta$ in all cases, and $z,z'$ are  defined as in \eqref{zz}. 
In subcase \eqref{dq4} (resp. \eqref{dq41}) we assume that 
$\langle \gamma^\vee,\,w^{-1}\theta \rangle$ is nonzero (resp. zero).
\begin{equation}\label{dq4}
	\begin{diagram}\node[2]{\lfloor r_\theta w r_\gamma \rfloor}\\
	\node{\lfloor r_\theta w \rfloor} \arrow{ne,t}{z\gamma}
	\node[2]{\lfloor w r_\gamma\rfloor} \arrow{nw,t,..}{-\lfloor w r_\gamma \rfloor^{-1}\theta} \\\node[2]{w} \arrow{nw,b,..}{-w^{-1}\theta} \arrow{ne,b,..}{\gamma}
	\end{diagram}
\qquad \qquad
	\begin{diagram}\node[2]{\lfloor w r_\gamma \rfloor}\\
	\node{w } \arrow{ne,t,..}{\gamma}
	\node[2]{\lfloor r_\theta w r_\gamma\rfloor} \arrow{nw,t,..}{z' \lfloor w r_\gamma \rfloor^{-1}\theta} \\\node[2]{\lfloor r_\theta w\rfloor} \arrow{nw,b,..}{zw^{-1}\theta} 
	\arrow{ne,b}{z\gamma}\end{diagram}
\end{equation}
\begin{equation}\label{dq41}
	\begin{diagram}\node[2]{\lfloor r_\theta w r_\gamma \rfloor}\\
	\node{\lfloor r_\theta w \rfloor} \arrow{ne,t,..}{z\gamma}
	\node[2]{\lfloor w r_\gamma\rfloor} \arrow{nw,t,..}{-\lfloor w r_\gamma \rfloor^{-1}\theta} \\\node[2]{w} \arrow{nw,b,..}{-w^{-1}\theta} \arrow{ne,b,..}{\gamma}
	\end{diagram}
\qquad \qquad
	\begin{diagram}\node[2]{\lfloor w r_\gamma \rfloor}\\
	\node{w } \arrow{ne,t,..}{\gamma}
	\node[2]{\lfloor r_\theta w r_\gamma\rfloor} \arrow{nw,t,..}{z' \lfloor w r_\gamma \rfloor^{-1}\theta} \\\node[2]{\lfloor r_\theta w\rfloor} \arrow{nw,b,..}{zw^{-1}\theta} 
	\arrow{ne,b,..}{z\gamma}\end{diagram}
\end{equation}
\end{enumerate}
\end{lem}

\begin{remark}\label{lr-relabel} 
(1) The left diagram of~\eqref{dq1} of Lemma~\ref{lemma.diamond.PQBG} is the classical Diamond 
Lemma~\ref{L:diamond}.

(2) The right diagrams in \eqref{dq1}, \eqref{dq2}, \eqref{dq3}, \eqref{dq31}, \eqref{dq4}, and~\eqref{dq41} 
are relabelings of the left diagrams in \eqref{dq1}, \eqref{dq2}, \eqref{dq4}, \eqref{dq31}, \eqref{dq3}, 
and~\eqref{dq41}, respectively, and vice versa. For instance, we can obtain a right diagram by labeling 
$w$ the leftmost vertex in the corresponding left diagram, and by recalculating all the other vertex and 
edge labels.
\end{remark}

\begin{proof}[Proof of Lemma~{\rm \ref{lemma.diamond.PQBG}}]
By Proposition \ref{P:circ} only the left diagrams need to be established.
In all cases, the bottom half of a diamond in $\QB(W^J)$ is lifted to the affine Bruhat order
using Theorem~\ref{T:Plift}~(1). There the diamond is completed using the usual 
Diamond Lemma~\ref{L:diamond} for the affine Weyl group. 
The affine diamond is pushed down to $\QB(W^J)$ using Theorem~\ref{T:Plift}~(2).

Consider the left diagram in \eqref{dq2}. By Theorem~\ref{T:Plift}~(1), the quantum edge $\lfloor w r_\gamma \rfloor \qleft{\gamma} w$
lifts to an affine Bruhat cover $y\lessdot x$ in $\OmegaP^\infty$ where
$x=wt_\mu$, $\mu$ is $J$-superantidominant with $z_\mu=\id$, $y=wr_\gamma t_{\gamma^\vee+\mu} = xr_{\ti{\gamma}}$,
and $\ti{\gamma}= \gamma+(1+\pair{\mu}{\gamma})\delta\in\Phiafm$.
Since $r_\alpha w \gtrdot w$ and $r_\alpha w\in W^J$, it follows that $r_\alpha x \lessdot x$.
Moreover this covering relation is the affine lift
into $\OmegaP^\infty$ of the Bruhat edge $r_\alpha w \qleft{w^{-1}\alpha} w$. 
The elements $r_\alpha x$ and $y$ are distinct since they have different translation components.
By the Diamond Lemma \ref{L:diamond} for the affine Weyl group, we have
$r_\alpha x \gtrdot r_\alpha y$ and $y \gtrdot r_\alpha y$. The latter cover implies that
$r_\alpha y \in \OmegaP^\infty$. Theorem~\ref{T:Plift}~(2) yields the top half of the
left diagram in \eqref{dq2}.

Consider the bottom half of the left diagram in \eqref{dq3} (which is also the
same half diagram in \eqref{dq31}). The quantum edge $\mcr{r_\theta w} \qleft{-w^{-1}\theta} w$
lifts to the affine cover in $\OmegaP^\infty$ given by $r_0 x \lessdot x$ where $x=wt_\mu$
and $r_0x=r_\theta w t_{\mu-w^{-1}\tv}$. The Bruhat edge $wr_\gamma \qleft{\gamma} w$
lifts to the affine cover in $\OmegaP^\infty$ given by $w r_\gamma t_\mu=x r_{\ti{\gamma}} \lessdot wt_\mu$
where $\ti{\gamma} = \gamma + \pair{\mu}{\gamma} \delta$. One may verify that
$r_0 x \ne x r_{\ti{\gamma}}$. By the Diamond Lemma \ref{L:diamond} for the affine Weyl group,
we have $r_0 x \gtrdot r_0 x r_{\ti{\gamma}}$ and $x r_{\ti{\gamma}} \gtrdot r_0 x r_{\ti{\gamma}}$.
Arguing as in the proof of Proposition \ref{P:thetaPQBG} and using that $x r_{\ti{\gamma}}\in \OmegaP^\infty$,
one may show that $r_0 x r_{\ti{\gamma}} \in \OmegaP^\infty$. By Theorem~\ref{T:Plift}~(2)
we obtain edges in $\QB(W^J)$ which complete the diamond, with the only remaining issue being the type
of the edge $\mcr{r_\theta w} \to \mcr{r_\theta w r_\gamma}$. It is quantum or Bruhat
depending on whether the translation elements in the affine lift $r_0 x \gtrdot r_0 x r_{\ti{\gamma}}$
are different or the same. Since 
$r_0 x r_{\ti{\gamma}} = r_\theta w r_\gamma t_{\mu-r_\gamma w^{-1}(\tv)}$ we see that
the translation element changes in passing from $r_0x$ to $r_0xr_{\ti{\gamma}}$
if and only if $\pair{w^{-1}\tv}{\gamma}\ne0$, as required.

The cases for the diagrams \eqref{dq4} and \eqref{dq41}
are similar to those for \eqref{dq3} and \eqref{dq31}.

Let us now prove the congruence of the weights of the two paths, by focusing, as an example, on the left diagram in \eqref{dq3};
the proofs for the other diagrams are similar. 
The weight of the directed path on the left side of 
the mentioned diagram is equal to
$-w^{-1}\theta^{\vee}+z\gamma^{\vee}$, and hence 
congruent to 
\begin{equation*}
-w^{-1}\theta^{\vee}+\gamma^{\vee} \mod Q_{J}^{\vee}.
\end{equation*}
The weight of the directed path on the right side of the diagram
is equal to $-\mcr{wr_{\gamma}}^{-1}\theta^{\vee}$, and 
hence is congruent to 
\begin{equation*}
-r_{\gamma}w^{-1}\theta^{\vee}=
-w^{-1}\theta^{\vee}+\pair{w^{-1}\theta^{\vee}}{\gamma}\gamma^{\vee} \mod Q_{J}^{\vee}.
\end{equation*}
Because $w \stackrel{\gamma}{\longrightarrow} \mcr{wr_{\gamma}}$ 
is a Bruhat edge, we see that $w\gamma$ is a positive root. Notice that 
$w\gamma \ne \theta$ since $w^{-1}\theta$ is a negative root. 
Also, since $\pair{\gamma^{\vee}}{w^{-1}\theta} \ne 0$ by 
the assumption of part (3), we see that 
$\pair{w^{-1}\theta^{\vee}}{\gamma} \ne 0$. 
Recalling the well-known fact that 
$\pair{\theta^{\vee}}{\beta}$ can only be $0$ or $1$ 
for each $\beta \in \Phi^{+}$ with $\beta \ne \theta$, 
it follows that $\pair{w^{-1}\theta^{\vee}}{\gamma}=
 \pair{\theta^{\vee}}{w\gamma}=1$. 
Therefore we obtain
\begin{equation*}
-w^{-1}\theta^{\vee}+\pair{w^{-1}\theta^{\vee}}{\gamma}\gamma^{\vee}=
-w^{-1}\theta^{\vee}+\gamma^{\vee}, 
\end{equation*}
as desired.
\end{proof}

\section{Quantum Bruhat graph and the level-zero weight poset}
\label{section.level 0 WP}

In~\cite{Littelmann1995}, Littelmann introduced a poset related to Lakshmibai--Seshadri (LS) paths
for arbitrary (not necessarily dominant) integral weights. We consider this poset for level-zero weights.
Littelmann did not give a precise local description of it. Our main result in this section
is a characterization of its cover relations in terms of the PQBG.

\subsection{The level-zero weight poset}
Fix a dominant weight $\lambda$ in the finite weight lattice $X$ (cf. Section \ref{s:stab} and the notation thereof, e.g., $W_J$ is the stabilizer of $\lambda$).  We view $X$ as a sublattice of $\Xafz$. Let $\Xafz(\lambda)$ be the orbit of $\lambda$ under the action of the affine Weyl group $\Waf$. 

\begin{definition} {\rm (Level-zero weight poset~\cite{Littelmann1995})}
\label{definition.level-0 weight poset}
A poset structure is defined on $\Xafz(\lambda)$ as the transitive closure of the relation
\begin{equation}\label{lev0anti}
 \mu < r_\beta\mu \quad \Leftrightarrow \quad \langle \beta^\vee,\,\mu \rangle >0\,,
\end{equation}
where $\beta\in\Phiafp$. This poset is called the {\em level-zero weight poset for} $\lambda$.
\end{definition}

\begin{remarks} \label{R:weightposet}\hfill
\begin{enumerate}
\item Assume that $W_J$ is trivial, and we set $\mu = w\lambda$ for $w \in \Waf$.
Then, for $\beta \in \Phiafp$, we have $\mu < r_{\beta}\mu$ in the level-zero weight poset
if and only if $w^{-1} r_{\beta} \prec w^{-1}$ in the generic Bruhat order $\prec$ on $\Waf$
introduced by Lusztig~\cite{Lusztig}. Indeed, this equivalence follows from the definitions of these partial orders
by using~\cite[Claim 4.14, page 96]{Soergel}.
The generic Bruhat order also recently appeared in~\cite{Lanini}.
\item We can define the poset $\Xafz(-\lambda)$ on the orbit of the antidominant weight $-\lambda$ in the same way, using \eqref{lev0anti}. The posets $\Xafz(\lambda)$ and $\Xafz(-\lambda)$ are dual isomorphic, in the sense that, for $\mu,\nu\in \Xafz(\lambda)$, we have
\[\mu<\nu\;\;\Leftrightarrow\;\;-\mu>-\nu\,.\]
Therefore, all the statements in this section can be easily rephrased for $\Xafz(-\lambda)$. 
\end{enumerate}
\end{remarks}

An example of $\Xafz(\lambda)$ is given in Figure \ref{exposet}. As we can see from this example, $\Xafz(\lambda)$ is not a graded poset in general.

Littelmann~\cite{Littelmann1995} introduced a distance function on the level-zero weight poset. Namely, if $\mu \le \nu$
in $\Xafz(\lambda)$, then $\dist(\mu,\nu)$\footnote{The notation in \cite{Littelmann1995} is $\dist(\nu,\mu)$.} is the maximum length 
of a chain from $\mu$ to $\nu$. Clearly, covers correspond to elements at distance 1. 

\begin{lem} \cite[Lemma 4.1]{Littelmann1995}
\label{lemma.littelmann}
Let $\mu,\nu \in \Xafz(\lambda)$. 
\begin{enumerate}
\item If $\mu \le \nu$ and $\alpha$ is a simple root in $\Phiaf$ such that 
$\langle \alpha^\vee, \mu \rangle\ge 0$ but $\langle \alpha^\vee, \nu \rangle<0$, then $\mu\le r_\alpha\nu$ and $\dist(\mu,r_\alpha\nu) < 
\dist(\mu,\nu)$.
\item If $\mu \le \nu$ and $\alpha$ is a simple root in $\Phiaf$ such that 
$\langle \alpha^\vee, \mu \rangle > 0$ but $\langle \alpha^\vee, \nu \rangle \le 0$, then $r_\alpha\mu\le \nu$ and $\dist(r_\alpha\mu,\nu) < 
\dist(\mu,\nu)$.
\item If $\mu \le \nu$ and $\alpha$ is a simple root in $\Phiaf$ such that $\langle \alpha^\vee, \mu \rangle,\, 
\langle \alpha^\vee, \nu \rangle>0$ (respectively
\linebreak 
$\langle \alpha^\vee, \mu \rangle, \, \langle \alpha^\vee, \nu \rangle < 0$),
then $\dist(\mu,\nu) = \dist(r_\alpha \mu, r_\alpha \nu)$.
\end{enumerate}
\end{lem}

We label a cover $\mu\lessdot\nu=r_\beta\mu$ of $\Xafz(\lambda)$ by the corresponding positive real root 
$\beta$. Preliminary results about the covers of $\Xafz(\lambda)$ were obtained by Naito and Sagaki.

\begin{lem}  \label{lemma.cover}  \mbox{}
\begin{enumerate}
\item \cite[Remark 2.10 and Lemma 2.11]{NS4}.  For untwisted types,  
a necessary condition for $\mu< \nu$ to be a cover in $\Xafz(\lambda)$ is that $\nu = r_\beta\mu$
with $\beta \in \Phi^+$ or $\beta \in \delta - \Phi^+$.
\item \cite[Remark 2.10 (2)]{NS4} Let $\mu,\nu \in \Xafz(\lambda)$ be such that $\nu = r_{\alpha}\mu$ 
for a simple root $\alpha$ in $\Phiaf$ such that $\langle \alpha^\vee, \mu \rangle>0$. Then $\dist(\mu,\nu)=1$.
\end{enumerate}
\end{lem}

We consider the standard projection map $\clproj$ from $\Xafz(\lambda)$ to the orbit of $\lambda$ under the 
finite Weyl group (by factoring out the $\delta$ part). 
We identify $W\lambda\simeq W/W_J\simeq W^J$, and consider on $W^J$ the PQBG
structure. Note that, by contrast with $\Xafz(\lambda)$, the edges of the latter 
are labeled by positive roots $\gamma\in\Phi^+$ (of the finite root system) corresponding to right multiplication 
by $r_\gamma$. We use solid arrows to denote covers in the Bruhat order, whereas dotted arrows denote quantum
edges in the PQBG on $W^J$. 

Our main result is that the level-zero weight poset is an affine lift of the corresponding parabolic quantum 
Bruhat graph. This is illustrated in Figure~\ref{exposet}, where the edges of the (parabolic) Bruhat graph 
(i.e., the slice of the level-zero weight poset with no $\delta$, onto which we project) are shown in red.
Projecting all vertices onto the red part, one obtains the QBG of Figure~\ref{fig:qB}.

\begin{thm} \label{theorem.level 0 weight poset}
Let $\mu\in \Xafz(\lambda)$ and $w:=\clproj(\mu)\in W^J$.
If $\mu \lessdot \nu$ is a cover in $\Xafz(\lambda)$ labeled by $\beta \in \Phi^{{\rm af}+}$, 
then $w \to \clproj(\nu)$ is a Bruhat 
(respectively quantum) edge in the PQBG on $W^J$ labeled by 
$w^{-1}\beta\in\Phi^+\setminus\Phi_J^+$ (respectively $w^{-1}(\beta-\delta)$), depending on 
$\beta\in\Phi^+$ (respectively $\beta\in\delta-\Phi^+$). Conversely, if $\begin{diagram}\dgARROWLENGTH=1.5em\node{w} \arrow{e,t}{\gamma}\node{wr_\gamma=w'}\end{diagram}$ 
(respectively $\begin{diagram}\dgARROWLENGTH=1.5em\node{w} \arrow{e,t,..}{\gamma}
\node{\lfloor wr_\gamma\rfloor=w'}\end{diagram}$) in the PQBG for 
$\gamma \in \Phi^{+} \setminus \Phi_{J}^{+}$, then there exists a cover $\mu \lessdot \nu$ in 
$\Xafz(\lambda)$ labeled by $w\gamma$ (respectively $\delta+w\gamma$) with $\clproj(\nu)=w'$. 
\end{thm}

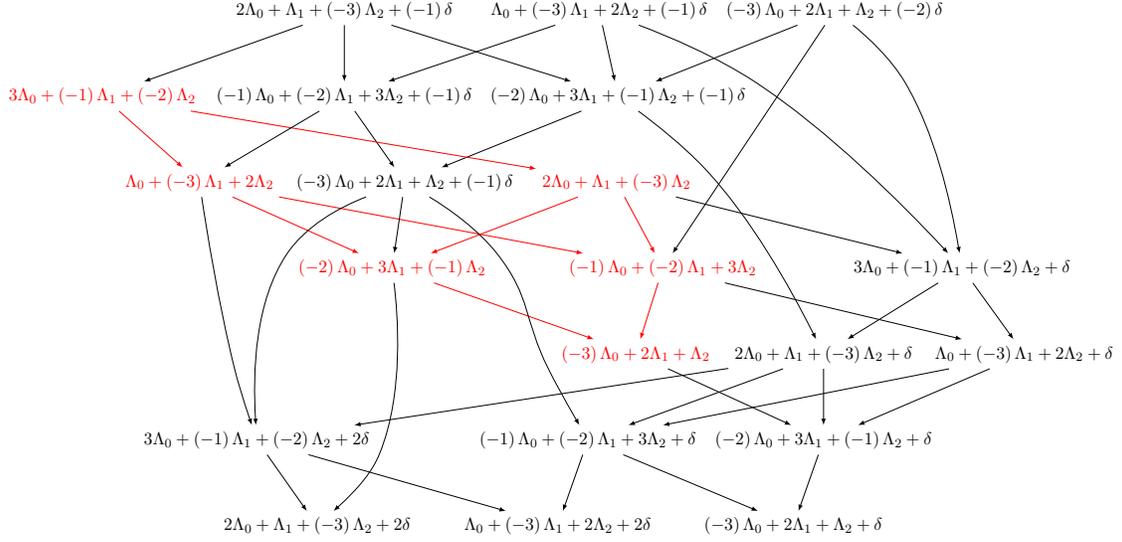
\begin{figure}\label{exposet}
\scalebox{.6}{
\begin{tikzpicture}[>=latex,line join=bevel]
\node (2*Lambda0+Lambda1-3*Lambda2) at (380bp,225bp) [red, draw,draw=none] {$2\Lambda_{0} + \Lambda_{1} + \left(-3\right)\Lambda_{2}$};
  \node (3*Lambda0-Lambda1-2*Lambda2) at (57bp,279bp) [red, draw,draw=none] {$3\Lambda_{0} + \left(-1\right)\Lambda_{1} + \left(-2\right)\Lambda_{2}$};
  \node (-Lambda0-2*Lambda1+3*Lambda2-delta) at (209bp,279bp) [draw,draw=none] {$\left(-1\right)\Lambda_{0} + \left(-2\right)\Lambda_{1} + 3\Lambda_{2} + \left(-1\right)\delta$};
  \node (-2*Lambda0+3*Lambda1-Lambda2) at (239bp,171bp) [red, draw,draw=none] {$\left(-2\right)\Lambda_{0} + 3\Lambda_{1} + \left(-1\right)\Lambda_{2}$};
  \node (3*Lambda0-Lambda1-2*Lambda2+2*delta) at (154bp,63bp) [draw,draw=none] {$3\Lambda_{0} + \left(-1\right)\Lambda_{1} + \left(-2\right)\Lambda_{2} + 2\delta$};
  \node (-3*Lambda0+2*Lambda1+Lambda2+delta) at (491bp,9bp) [draw,draw=none] {$\left(-3\right)\Lambda_{0} + 2\Lambda_{1} + \Lambda_{2} + \delta$};
  \node (2*Lambda0+Lambda1-3*Lambda2-delta) at (209bp,333bp) [draw,draw=none] {$2\Lambda_{0} + \Lambda_{1} + \left(-3\right)\Lambda_{2} + \left(-1\right)\delta$};
  \node (2*Lambda0+Lambda1-3*Lambda2+delta) at (510bp,117bp) [draw,draw=none] {$2\Lambda_{0} + \Lambda_{1} + \left(-3\right)\Lambda_{2} + \delta$};
  \node (-3*Lambda0+2*Lambda1+Lambda2-2*delta) at (517bp,333bp) [draw,draw=none] {$\left(-3\right)\Lambda_{0} + 2\Lambda_{1} + \Lambda_{2} + \left(-2\right)\delta$};
  \node (Lambda0-3*Lambda1+2*Lambda2+delta) at (636bp,117bp) [draw,draw=none] {$\Lambda_{0} + \left(-3\right)\Lambda_{1} + 2\Lambda_{2} + \delta$};
  \node (-3*Lambda0+2*Lambda1+Lambda2) at (392bp,117bp) [red, draw,draw=none] {$\left(-3\right)\Lambda_{0} + 2\Lambda_{1} + \Lambda_{2}$};
  \node (Lambda0-3*Lambda1+2*Lambda2) at (118bp,225bp) [red, draw,draw=none] {$\Lambda_{0} + \left(-3\right)\Lambda_{1} + 2\Lambda_{2}$};
  \node (Lambda0-3*Lambda1+2*Lambda2-delta) at (369bp,333bp) [draw,draw=none] {$\Lambda_{0} + \left(-3\right)\Lambda_{1} + 2\Lambda_{2} + \left(-1\right)\delta$};
  \node (3*Lambda0-Lambda1-2*Lambda2+delta) at (597bp,171bp) [draw,draw=none] {$3\Lambda_{0} + \left(-1\right)\Lambda_{1} + \left(-2\right)\Lambda_{2} + \delta$};
  \node (-Lambda0-2*Lambda1+3*Lambda2+delta) at (362bp,63bp) [draw,draw=none] {$\left(-1\right)\Lambda_{0} + \left(-2\right)\Lambda_{1} + 3\Lambda_{2} + \delta$};
  \node (2*Lambda0+Lambda1-3*Lambda2+2*delta) at (192bp,9bp) [draw,draw=none] {$2\Lambda_{0} + \Lambda_{1} + \left(-3\right)\Lambda_{2} + 2\delta$};
  \node (-Lambda0-2*Lambda1+3*Lambda2) at (409bp,171bp) [red, draw,draw=none] {$\left(-1\right)\Lambda_{0} + \left(-2\right)\Lambda_{1} + 3\Lambda_{2}$};
  \node (-3*Lambda0+2*Lambda1+Lambda2-delta) at (247bp,225bp) [draw,draw=none] {$\left(-3\right)\Lambda_{0} + 2\Lambda_{1} + \Lambda_{2} + \left(-1\right)\delta$};
  \node (-2*Lambda0+3*Lambda1-Lambda2-delta) at (381bp,279bp) [draw,draw=none] {$\left(-2\right)\Lambda_{0} + 3\Lambda_{1} + \left(-1\right)\Lambda_{2} + \left(-1\right)\delta$};
  \node (-2*Lambda0+3*Lambda1-Lambda2+delta) at (510bp,63bp) [draw,draw=none] {$\left(-2\right)\Lambda_{0} + 3\Lambda_{1} + \left(-1\right)\Lambda_{2} + \delta$};
  \node (Lambda0-3*Lambda1+2*Lambda2+2*delta) at (343bp,9bp) [draw,draw=none] {$\Lambda_{0} + \left(-3\right)\Lambda_{1} + 2\Lambda_{2} + 2\delta$};
  \draw [->] (2*Lambda0+Lambda1-3*Lambda2-delta) ..controls (209bp,317.01bp) and (209bp,306.86bp)  .. (-Lambda0-2*Lambda1+3*Lambda2-delta);
  \draw [->] (Lambda0-3*Lambda1+2*Lambda2-delta) ..controls (372.57bp,316.93bp) and (374.86bp,306.64bp)  .. (-2*Lambda0+3*Lambda1-Lambda2-delta);
  \draw [red,->] (Lambda0-3*Lambda1+2*Lambda2) ..controls (156.91bp,207.64bp) and (187.89bp,193.81bp)  .. (-2*Lambda0+3*Lambda1-Lambda2);
  \draw [red,->] (-2*Lambda0+3*Lambda1-Lambda2) ..controls (288.87bp,153.4bp) and (329.49bp,139.06bp)  .. (-3*Lambda0+2*Lambda1+Lambda2);
  \draw [->] (3*Lambda0-Lambda1-2*Lambda2+2*delta) ..controls (216.16bp,45.241bp) and (267.53bp,30.562bp)  .. (Lambda0-3*Lambda1+2*Lambda2+2*delta);
  \draw [->] (Lambda0-3*Lambda1+2*Lambda2+delta) ..controls (544.48bp,98.964bp) and (466.91bp,83.675bp)  .. (-Lambda0-2*Lambda1+3*Lambda2+delta);
  \draw [->] (-Lambda0-2*Lambda1+3*Lambda2) ..controls (484.32bp,153.08bp) and (547.49bp,138.06bp)  .. (Lambda0-3*Lambda1+2*Lambda2+delta);
  \draw [->] (2*Lambda0+Lambda1-3*Lambda2+delta) ..controls (510bp,101.01bp) and (510bp,90.861bp)  .. (-2*Lambda0+3*Lambda1-Lambda2+delta);
  \draw [->] (-2*Lambda0+3*Lambda1-Lambda2) ..controls (242.64bp,142.38bp) and (246.6bp,91.832bp)  .. (231bp,54bp) .. controls (226.45bp,42.961bp) and (218.06bp,32.723bp)  .. (2*Lambda0+Lambda1-3*Lambda2+2*delta);
  \draw [->] (3*Lambda0-Lambda1-2*Lambda2+delta) ..controls (569.45bp,153.9bp) and (548.62bp,140.97bp)  .. (2*Lambda0+Lambda1-3*Lambda2+delta);
  \draw [->] (-2*Lambda0+3*Lambda1-Lambda2-delta) ..controls (337.62bp,261.52bp) and (302.68bp,247.44bp)  .. (-3*Lambda0+2*Lambda1+Lambda2-delta);
  \draw [->] (2*Lambda0+Lambda1-3*Lambda2+delta) ..controls (392.75bp,99.215bp) and (288.21bp,83.358bp)  .. (3*Lambda0-Lambda1-2*Lambda2+2*delta);
  \draw [->] (3*Lambda0-Lambda1-2*Lambda2+2*delta) ..controls (165.64bp,46.455bp) and (173.52bp,35.264bp)  .. (2*Lambda0+Lambda1-3*Lambda2+2*delta);
  \draw [->] (-3*Lambda0+2*Lambda1+Lambda2-2*delta) ..controls (537.86bp,316.16bp) and (552.29bp,302.65bp)  .. (561bp,288bp) .. controls (579.84bp,256.3bp) and (589.68bp,213.82bp)  .. (3*Lambda0-Lambda1-2*Lambda2+delta);
  \draw [red,->] (Lambda0-3*Lambda1+2*Lambda2) ..controls (215.41bp,206.92bp) and (298.27bp,191.55bp)  .. (-Lambda0-2*Lambda1+3*Lambda2);
  \draw [->] (-2*Lambda0+3*Lambda1-Lambda2-delta) ..controls (404.29bp,261.87bp) and (421.76bp,248.02bp)  .. (435bp,234bp) .. controls (464.28bp,203bp) and (488.88bp,159.03bp)  .. (2*Lambda0+Lambda1-3*Lambda2+delta);
  \draw [->] (2*Lambda0+Lambda1-3*Lambda2) ..controls (451.92bp,207.1bp) and (512.13bp,192.12bp)  .. (3*Lambda0-Lambda1-2*Lambda2+delta);
  \draw [->] (-Lambda0-2*Lambda1+3*Lambda2+delta) ..controls (356.32bp,46.849bp) and (352.64bp,36.409bp)  .. (Lambda0-3*Lambda1+2*Lambda2+2*delta);
  \draw [->] (Lambda0-3*Lambda1+2*Lambda2-delta) ..controls (316.73bp,315.36bp) and (274bp,300.94bp)  .. (-Lambda0-2*Lambda1+3*Lambda2-delta);
  \draw [->] (-3*Lambda0+2*Lambda1+Lambda2) ..controls (429.94bp,99.636bp) and (460.15bp,85.812bp)  .. (-2*Lambda0+3*Lambda1-Lambda2+delta);
  \draw [->] (2*Lambda0+Lambda1-3*Lambda2-delta) ..controls (265.44bp,315.28bp) and (311.92bp,300.69bp)  .. (-2*Lambda0+3*Lambda1-Lambda2-delta);
  \draw [->] (2*Lambda0+Lambda1-3*Lambda2-delta) ..controls (159.46bp,315.4bp) and (119.11bp,301.06bp)  .. (3*Lambda0-Lambda1-2*Lambda2);
  \draw [->] (-Lambda0-2*Lambda1+3*Lambda2-delta) ..controls (220.64bp,262.46bp) and (228.52bp,251.26bp)  .. (-3*Lambda0+2*Lambda1+Lambda2-delta);
  \draw [->] (Lambda0-3*Lambda1+2*Lambda2) ..controls (122.21bp,196.96bp) and (130.02bp,148.12bp)  .. (140bp,108bp) .. controls (142.22bp,99.056bp) and (145.2bp,89.271bp)  .. (3*Lambda0-Lambda1-2*Lambda2+2*delta);
  \draw [->] (-3*Lambda0+2*Lambda1+Lambda2-2*delta) ..controls (472.97bp,315.52bp) and (437.51bp,301.44bp)  .. (-2*Lambda0+3*Lambda1-Lambda2-delta);
  \draw [->] (Lambda0-3*Lambda1+2*Lambda2-delta) ..controls (413.13bp,316.57bp) and (443.22bp,303.63bp)  .. (467bp,288bp) .. controls (513.07bp,257.71bp) and (559.19bp,211.52bp)  .. (3*Lambda0-Lambda1-2*Lambda2+delta);
  \draw [->] (-Lambda0-2*Lambda1+3*Lambda2+delta) ..controls (403.67bp,45.557bp) and (437.1bp,31.563bp)  .. (-3*Lambda0+2*Lambda1+Lambda2+delta);
  \draw [red,->] (2*Lambda0+Lambda1-3*Lambda2) ..controls (334.25bp,207.48bp) and (297.26bp,193.31bp)  .. (-2*Lambda0+3*Lambda1-Lambda2);
  \draw [->] (-2*Lambda0+3*Lambda1-Lambda2+delta) ..controls (504.32bp,46.849bp) and (500.64bp,36.409bp)  .. (-3*Lambda0+2*Lambda1+Lambda2+delta);
  \draw [red,->] (3*Lambda0-Lambda1-2*Lambda2) ..controls (75.818bp,262.34bp) and (89.406bp,250.31bp)  .. (Lambda0-3*Lambda1+2*Lambda2);
  \draw [red,->] (-Lambda0-2*Lambda1+3*Lambda2) ..controls (403.92bp,154.85bp) and (400.63bp,144.41bp)  .. (-3*Lambda0+2*Lambda1+Lambda2);
  \draw [->] (-3*Lambda0+2*Lambda1+Lambda2-delta) ..controls (244.62bp,208.93bp) and (243.09bp,198.64bp)  .. (-2*Lambda0+3*Lambda1-Lambda2);
  \draw [->] (-3*Lambda0+2*Lambda1+Lambda2-delta) ..controls (205.79bp,209.28bp) and (184.76bp,197.46bp)  .. (173bp,180bp) .. controls (152.85bp,150.09bp) and (151.39bp,106.3bp)  .. (3*Lambda0-Lambda1-2*Lambda2+2*delta);
  \draw [red,->] (2*Lambda0+Lambda1-3*Lambda2) ..controls (388.76bp,208.69bp) and (394.53bp,197.95bp)  .. (-Lambda0-2*Lambda1+3*Lambda2);
  \draw [->] (-3*Lambda0+2*Lambda1+Lambda2-delta) ..controls (274.85bp,208.25bp) and (293.05bp,195.25bp)  .. (305bp,180bp) .. controls (326.6bp,152.43bp) and (321.37bp,139.34bp)  .. (337bp,108bp) .. controls (341.64bp,98.688bp) and (347.24bp,88.537bp)  .. (-Lambda0-2*Lambda1+3*Lambda2+delta);
  \draw [->] (Lambda0-3*Lambda1+2*Lambda2+delta) ..controls (595.39bp,99.596bp) and (562.94bp,85.687bp)  .. (-2*Lambda0+3*Lambda1-Lambda2+delta);
  \draw [red,->] (3*Lambda0-Lambda1-2*Lambda2) ..controls (166.87bp,260.63bp) and (262.78bp,244.6bp)  .. (2*Lambda0+Lambda1-3*Lambda2);
  \draw [->] (3*Lambda0-Lambda1-2*Lambda2+delta) ..controls (608.95bp,154.46bp) and (617.03bp,143.26bp)  .. (Lambda0-3*Lambda1+2*Lambda2+delta);
  \draw [->] (-Lambda0-2*Lambda1+3*Lambda2-delta) ..controls (180.05bp,261.82bp) and (157.99bp,248.73bp)  .. (Lambda0-3*Lambda1+2*Lambda2);
  \draw [->] (-3*Lambda0+2*Lambda1+Lambda2-2*delta) ..controls (494.05bp,298.58bp) and (443.44bp,222.66bp)  .. (-Lambda0-2*Lambda1+3*Lambda2);
  \draw [->] (2*Lambda0+Lambda1-3*Lambda2+delta) ..controls (461.87bp,99.438bp) and (422.81bp,85.189bp)  .. (-Lambda0-2*Lambda1+3*Lambda2+delta);
\end{tikzpicture}
}
\caption{Slice of the level-zero weight poset for type $A_2^{(1)}$ and weight $2\Lambda_1 + \Lambda_2 - 3 \Lambda_0$.\label{figure.poset}}
\end{figure}

The proof of Theorem~\ref{theorem.level 0 weight poset} is given in the remainder of this section.

\subsection{Outline of the proof}
Let us begin by giving a brief outline of the proof. To relate the cover relations in the level-zero weight poset $\Xafz(\lambda)$
and the edges in the PQBG, we use the so-called Diamond Lemma on $\Xafz(\lambda)$ 
to successively move a cover $\mu \lessdot r_\beta \mu$ ``closer'' to a cover $\mu \lessdot r_{\alpha} \mu$ for a simple
root $\alpha$ in $\Phiaf$. For simple roots, the statement of Theorem~\ref{theorem.level 0 weight poset} is proved in
Section~\ref{section.simple roots}. The Diamond Lemma in the level-zero weight poset 
is the subject of Section~\ref{section.diamond.poset}. Recall that the Diamond Lemmas for the
PQBG were treated in Section~\ref{subsection.diamond}. In Section~\ref{section.diamond.PQBG} we 
prove some further statements related to the Diamond Lemmas for the PQBG that we need
for our arguments. We conclude in Section~\ref{section.main argument} with the main argument, based on matching 
the diamond reductions in the level-zero weight poset with those in the PQBG.

\subsection{Results for simple roots}
\label{section.simple roots}

In this section, we characterize a cover relation $\mu\lessdot \nu$ in $\Xafz(\lambda)$ when $\mu$ and $\nu$ are related 
by an affine simple reflection. 

We start with a simple lemma. Since some versions of it will be needed beyond this section, we collect all of them here. 

\begin{lem}\label{lem.proj} Let $\alpha$ be a simple root, $\beta$ a positive root (both in $\Phiaf$), and 
$\mu=wt_\tau\lambda$ with $w\in W$ and $\tau\in Q^\vee$. Let $\gamma\in\Phi^+$ be given by 
$\beta=\pm w\gamma+k\delta$. 

{\rm (1)} We have
\[\clproj(\mu)=\lfloor w\rfloor\,,\qquad\clproj(r_\beta\mu)=\lfloor w r_\gamma\rfloor\,.\]

{\rm (2)} If $\alpha\ne\alpha_0$, assume that $r_\alpha \lfloor w\rfloor\in W^J$, i.e., $\clproj(\mu)\ne\clproj(r_\alpha\mu)$. Then we have
\[
	\clproj(r_\alpha\mu)=\casetwo{r_\alpha \lfloor w\rfloor
	=\lfloor r_\alpha w\rfloor}{\alpha\ne\alpha_0}{\lfloor r_\theta w\rfloor}{\alpha=\alpha_0}
\]

{\rm (3)} If $\alpha\ne\alpha_0$, assume that $r_\alpha \lfloor w r_\gamma\rfloor\in W^J$, i.e., 
$\clproj(r_\beta\mu)\ne\clproj(r_\alpha r_\beta\mu)$. Then we have
\[
	\clproj(r_\alpha r_\beta\mu)
	=\casetwo{r_\alpha \lfloor w r_\gamma\rfloor
	=\lfloor r_\alpha \lfloor w r_\gamma\rfloor \rfloor}{\alpha\ne\alpha_0}{\lfloor r_\theta w r_\gamma\rfloor}{\alpha=\alpha_0}
\]
\end{lem}

\begin{proof}
We have
\begin{equation}\label{translw} 
	\mu=wt_\tau\lambda=w\lambda-\langle \tau,\,\lambda \rangle\delta\,.
\end{equation}
So $\clproj(\mu)=\lfloor w\rfloor$. Similarly, we have
\[\
	\clproj(r_{{\beta}}\mu)=\clproj(r_{w\gamma}t_{\pm kw\gamma^\vee}w t_\tau(\lambda))
	=\lfloor w r_\gamma\rfloor\,,\quad\mbox{and}\quad \clproj(r_{{\alpha}}\mu)
	=\casetwo{\lfloor r_\alpha w \rfloor}{\alpha\ne\alpha_0}{\lfloor r_\theta w \rfloor}{\alpha=\alpha_0}
\]
In addition, if $\alpha\ne\alpha_0$, then $\lfloor r_\alpha w\rfloor$ can only be $\lfloor w\rfloor$ or $r_\alpha \lfloor w\rfloor$, by
Lemma~\ref{L:tricoset}; but the first case cannot happen by the assumptions of the lemma. The calculation 
of $\clproj(r_\alpha r_\beta\mu)$ is similar, by also noting that, if ${\alpha}\ne\alpha_0$, then 
$\lfloor r_\alpha w r_\gamma\rfloor=\lfloor r_\alpha \lfloor w r_\gamma\rfloor\rfloor$. 
\end{proof}

\begin{lem}\label{lemma.simple.root}
Let $\mu,\nu \in \Xafz(\lambda)$ be such that $\nu = r_{\alpha}\mu$ for a simple root $\alpha$ in $\Phiaf$.
Then $\mu \lessdot \nu$ is a cover in $\Xafz(\lambda)$ if and only if  $\clproj(\mu)\to \clproj(\nu)$ is a Bruhat (respectively quantum) 
edge in the PQBG on $W^J$ labeled by $w^{-1}\alpha$ (respectively $-w^{-1}\theta$), 
where $w=\clproj(\mu)\in W^J$, depending on $\alpha\ne\alpha_0$ (respectively $\alpha=\alpha_0$). 
\end{lem}

\begin{proof}
Since $\alpha$ is a simple root, we have by Lemma~\ref{lemma.cover}~(2) that $\dist(\mu,\nu)=1$ if $\mu<\nu$. 
So in this case $\mu \lessdot \nu$ is equivalent to $\mu<\nu$. Letting $\mu=wt_\tau(\lambda)$ with $w\in W^J$ 
and $\tau \in Q^\vee$, we have $\clproj(\mu)=w$, by Lemma \ref{lem.proj}~(1). 
Let us first assume that $\alpha\ne\alpha_0$. Then 
\begin{equation*}
	\langle \alpha^\vee, \mu \rangle  = \langle \alpha^\vee, w\lambda \rangle 
	= \langle w^{-1}\alpha^\vee, \lambda \rangle ,
\end{equation*}
where for the first equality we used~\eqref{translw}. Hence
\begin{equation*}
	\mu < r_{\alpha}\mu \quad \Leftrightarrow \quad  \langle \alpha^\vee, \mu \rangle>0 \quad 
	\Leftrightarrow \quad  w^{-1}\alpha \in \Phi^+ \setminus \Phi^+_J 
	\quad \Leftrightarrow \quad  w \lessdot r_{\alpha} w \:\mbox{ in $W^J$}\,,
\end{equation*}
where the last equivalence is based on Lemma~\ref{L:tri}. The last condition is equivalent to $\clproj(\mu)\to \clproj(\nu)$ 
being a Bruhat edge in the PQBG, by Lemma~\ref{lem.proj}~(2). 
This proves the claim for $\alpha\ne\alpha_0$.

Now assume $\alpha=\alpha_0$.  Similarly to before
\begin{equation*}
	\langle \alpha^\vee, \mu \rangle  = \langle \alpha^\vee, w\lambda \rangle   
	= \langle -\theta^\vee, w\lambda \rangle 
	= -\langle w^{-1}\theta^\vee, \lambda \rangle ,
\end{equation*}
where we used $\alpha_0=-\theta + \delta$, or $\alpha_0^\vee+\theta^\vee=c$. Hence
\begin{equation*}
	\mu > r_{\alpha}\mu \quad \Leftrightarrow \quad \langle \alpha^\vee, \mu \rangle>0 
	\quad \Leftrightarrow \quad w^{-1}\theta \in \Phi^- \setminus \Phi^-_J .
\end{equation*}
By Proposition~\ref{P:thetaPQBG}, the last condition is equivalent to the fact that 
$\begin{diagram}\dgARROWLENGTH=3em\node{w} \arrow{e,t,..}{-w^{-1}\theta}\node{\lfloor r_\theta w \rfloor}\end{diagram}$ 
is a quantum edge in the PQBG. Also note that 
$\clproj(r_\alpha\mu)=\lfloor r_\theta w \rfloor $, by Lemma \ref{lem.proj}~(2). 
This proves the claim for $\alpha=\alpha_0$.
\end{proof}

\subsection{The Diamond Lemma in the level-zero weight poset}
\label{section.diamond.poset}

In this section we investigate the Diamond Lemma in the level-zero weight poset
$\Xafz(\lambda)$.

\begin{lem} \label{lem.covers}
Let $\mu\in \Xafz(\lambda)$ and $\mu < r_\beta\mu$ in $\Xafz(\lambda)$, where $\beta\in\Phiafp$.
Then there exists a simple root $\alpha$ in $\Phiaf$ (in fact, $\alpha\ne\alpha_0$ if $\beta\in\Phi$) 
such that $\langle\alpha^\vee,\beta\rangle>0$, and either
\[
 	\text{\rm{(1)}} \quad \mu \lessdot r_{\alpha}\mu \qquad \text{or} \qquad 
 	\text{\rm{(2)}} \quad r_{\alpha} r_\beta\mu \lessdot r_\beta\mu
\]
is a cover in $\Xafz(\lambda)$.
\end{lem}

\begin{proof}
We pick a simple root $\alpha$ in the decomposition of $\beta$ 
such that $\langle \alpha^\vee, \beta \rangle > 0$. This clearly exists, and in fact $\alpha\ne\alpha_0$ if $\beta\in\Phi$. 

By Definition~\ref{definition.level-0 weight poset}, we have 
$\langle \beta^\vee, \mu \rangle > 0$. We claim that either
\begin{equation} \label{equation.choices}
	\langle \alpha^\vee, \mu \rangle > 0 \qquad \text{or} \qquad \langle -r_\beta\alpha^\vee, \mu\rangle
	=-\langle \alpha^\vee, r_\beta\mu \rangle >0\,.
\end{equation}
Indeed, the reflection formula
\[
	r_\beta\alpha^\vee = \alpha^\vee - \langle \alpha^\vee, \beta \rangle \beta^\vee
\]
implies that $\alpha^\vee - r_\beta\alpha^\vee$ is a positive multiple of $\beta^\vee$. 
Now (\ref{equation.choices}) follows since $\langle \beta^\vee, \mu \rangle > 0$.
We conclude the proof by combining \eqref{equation.choices} with Lemma~\ref{lemma.cover}~(2). 
\end{proof}

Next we state the Diamond Lemma for the level-zero
weight poset.

\begin{lem} \label{lemma.diamond.weight}
Let $\alpha$ be a simple root, $\beta\ne\alpha$ a positive root (both in $\Phiaf$), and $\mu\in \Xafz(\lambda)$. In the left diagram, 
the bottom two covers imply the top two covers, while the top two covers imply the bottom two covers in the right diagram. 
\begin{equation}\label{dw}\begin{diagram}\node[2]{r_\alpha r_\beta\mu}\\
\node{r_\alpha\mu} \arrow{ne,t}{r_\alpha\beta} 
\node[2]{r_\beta\mu} \arrow{nw,t}{\alpha} \\\node[2]{\mu} \arrow{nw,b}{\alpha} \arrow{ne,b}{\beta}\end{diagram}
\qquad \qquad
\begin{diagram}\node[2]{r_\beta\mu}\\
\node{\mu} \arrow{ne,t}{\beta} 
\node[2]{r_\alpha r_\beta\mu} \arrow{nw,t}{\alpha} \\ \node[2]{r_\alpha\mu} \arrow{nw,b}{\alpha} \arrow{ne,b}{r_\alpha\beta}\end{diagram}
\end{equation}
\end{lem}

\begin{proof}
We start by assuming that the bottom two arrows are covers in the left diagram. Set $\nu:=r_\beta\mu$. By definition, we have 
$\langle \alpha^\vee, \mu \rangle >0$ and $\langle \beta^\vee, \mu \rangle >0$. We 
first show that $\langle \alpha^\vee, \nu \rangle >0$, which implies that we have the cover $\nu\lessdot r_\alpha\nu$, 
by Lemma~\ref{lemma.cover}~(2). Indeed, if  $\langle \alpha^\vee, \nu \rangle \le 0$, then Lemma~\ref{lemma.littelmann}~(2) would 
imply $\dist(r_\alpha\mu,\nu) < \dist(\mu,\nu)$; since $\dist(\mu,\nu)=1$, it would follow 
that $\nu=r_\alpha\mu$, which is impossible, since $\alpha\ne\beta$.

Turning to the remaining edge of the diamond, we clearly have $r_\alpha\mu< r_\alpha\nu$, as 
\[
	r_\alpha\nu=r_{r_\alpha\beta}(r_\alpha\mu)\;\;\;\;\mbox{and}\;\;\;\;\langle r_\alpha\beta^\vee, r_\alpha\mu \rangle
	=\langle \beta^\vee, \mu\rangle>0\,;
\]
note that $r_\alpha\beta$ is a positive root, as $\alpha\ne\beta$. The hypotheses of Lemma~\ref{lemma.littelmann}~(3) apply, so we have 
$1=\dist(\mu,\nu) = $ $\dist(r_\alpha\mu,r_\alpha\nu)$. We conclude that we have the cover $r_\alpha\mu \lessdot 
r_\alpha\nu$.

The proof for the right diagram is similar, where we now assume that the top two arrows are covers. More precisely, in order to prove that the bottom 
arrows are covers, we use Lemma~\ref{lemma.littelmann}~(1) for the left one, and then Lemma~\ref{lemma.littelmann}~(3) for the right one. 
\end{proof}

\subsection{More on the Diamond Lemmas for the PQBG}
\label{section.diamond.PQBG}

Recall the Diamond Lemmas for the PQBG on $W^J$ from Section~\ref{subsection.diamond}.
Recall that, given $w\in W$ and $\gamma\in\Phi^+$, in \eqref{zz} we defined $z,z'\in W_J$ by
\begin{equation*}
	r_\theta \lfloor w\rfloor  = \lfloor r_\theta w \rfloor z\,,\qquad r_\theta \lfloor w r_\gamma\rfloor = \lfloor r_\theta \lfloor w r_\gamma \rfloor \rfloor z'
	= \lfloor r_\theta w r_\gamma \rfloor z'\,.
\end{equation*} 

We need an analogue of Lemma~\ref{lem.covers}. 

\begin{lem} \label{lem.covers.PQBG}
Let $w\in W$, and let 
$\gamma\in\Phi^+\setminus\Phi^+_J$. Define $\beta\in\Phiafp$ by
\begin{equation}\label{ahbh}
	\beta:=\casetwo{w\gamma}{w\gamma\in\Phi^+}
	{\delta+w\gamma}{w\gamma\in\Phi^-}
\end{equation}
There exists an affine simple root $\alpha$ (in fact, $\alpha\ne\alpha_0$ if $w\gamma\in\Phi^+$)
such that $\langle \alpha^\vee, \beta \rangle>0$, 
and we have the edge in the PQBG indicated either in case {\rm (1)} or {\rm (2)} below:
\[
 	\text{\rm{(1)}} \;\; \casetwoc{\begin{diagram}\node{\lfloor w\rfloor} \arrow{e,t}{\lfloor w\rfloor^{-1}\alpha}\node{r_\alpha \lfloor w\rfloor}\end{diagram}}{\alpha\ne\alpha_0}
	{\begin{diagram}\dgARROWLENGTH=3em\node{\lfloor w\rfloor} \arrow{e,t,..}{-\lfloor w\rfloor^{-1}\theta}\node{\lfloor r_\theta w \rfloor}\end{diagram}}{\alpha=\alpha_0} \;\;\; 
 	\text{\rm{(2)}} \;\; \casetwo{\begin{diagram}\dgARROWLENGTH=4.5em\node{r_\alpha \lfloor w r_\gamma\rfloor} \arrow{e,t}{-\lfloor w r_\gamma\rfloor^{-1}\alpha}
	\node{\lfloor w r_\gamma\rfloor}\end{diagram}}{\alpha\ne\alpha_0}{\begin{diagram}\dgARROWLENGTH=5em\node{\lfloor r_\theta w r_\gamma\rfloor} \arrow{e,t,..}
	{z'\lfloor w r_\gamma\rfloor^{-1}\theta}\node{\lfloor w r_\gamma\rfloor}\end{diagram}}{\alpha=\alpha_0}
\]
\end{lem}

\begin{remark} 
If 
$\begin{diagram}\dgARROWLENGTH=1.5em\node{w} \arrow{e,t,..}{\gamma}\node{\lfloor wr_\gamma\rfloor}\end{diagram}$,
then $w\gamma\in \Phi^{-}$ for the following reason.
Observe that $\ell(\lfloor w r_{\gamma} \rfloor) \leq \ell(w) - 1$
since $\gamma \in \Phi^{+} \setminus \Phi_{J}^{+}$. Now suppose, by contradiction, that $w \gamma \in \Phi^{+}$.
Then, we have $w r_{\gamma} > w$ in the usual Bruhat order on $W$. Therefore, by~\cite[Proposition 2.5.1]{BB}, 
we obtain $\lfloor w r_{\gamma} \rfloor \geq \lfloor w \rfloor = w$, which implies that
$\ell(\lfloor w r_{\gamma} \rfloor) \geq \ell(w)$. This is a contradiction.
This proves that $w \gamma \in \Phi^{-}$.
\end{remark}

\begin{proof} 
Let $\mu:=w\lambda$, where $\lambda\in\Xafz$ is the fixed dominant element in the finite weight lattice whose stabilizer 
is $W_J$. We claim that $\mu<r_{\beta}\mu$ in 
$\Xafz(\lambda)$, which means that $\langle \beta^\vee, \mu\rangle>0$. Indeed, since $\gamma\in\Phi^+\setminus \Phi_J^+$, 
it follows from \eqref{ahbh} that in both cases we have
\[
	\langle \beta^\vee, \mu\rangle=\langle w\gamma^\vee, w\lambda \rangle=\langle \gamma^\vee, \lambda \rangle>0\,.
\]

We now apply Lemma \ref{lem.covers} to deduce the existence of a simple root $\alpha$ in $\Phiaf$ (in fact, $\alpha\ne\alpha_0$ if 
$\begin{diagram}\dgARROWLENGTH=1.5em\node{w} \arrow{e,t}{\gamma}\node{wr_\gamma}\end{diagram}$)
such that $\langle \alpha^\vee, \beta \rangle>0$, and either
\[
 	\text{\rm{(1)}} \quad \mu \lessdot r_{\alpha}\mu \qquad \text{or} \qquad 
 	\text{\rm{(2)}} \quad r_{\alpha} r_{\beta}\mu \lessdot r_{\beta}\mu
\]
in $\Xafz(\lambda)$. By Lemma \ref{lemma.simple.root} and Lemma \ref{lem.proj}, cases (1) and (2) can be rephrased as cases (1) and (2) in the lemma to 
be proved, respectively. 
\end{proof}

Note that we do not need all the cases of the diamond Lemma~\ref{lemma.diamond.PQBG} for the PQBG,
for instance the one where all four edges are quantum edges. By stating that we have a certain edge in the PQBG, 
we implicitly assume that both its vertices are in $W^J$.

\subsection{Main argument}
\label{section.main argument}

We address separately the direct $(\Rightarrow)$ and the converse $(\Leftarrow)$ statements. Recall that the height 
of a root is the sum of the coefficients in its expansion in the basis of simple roots.

\begin{proof}[Proof of $(\Rightarrow)$ in Theorem {\rm \ref{theorem.level 0 weight poset}}] 
Consider the cover $\mu\lessdot\nu=r_\beta\mu$  in $\Xafz(\lambda)$ labeled by $\beta$, and let $w:=\clproj(\mu)$. 
We proceed by induction on the height of $\beta$. If $\beta$ is a simple root, the conclusion follows directly from 
Lemma~\ref{lemma.simple.root}. If $\beta$ is not a simple root, we apply Lemma~\ref{lem.covers}; this gives an affine 
simple root $\alpha\ne\beta$ with $\langle \alpha^\vee, \beta \rangle>0$, which 
also satisfies condition (1) or (2) in the mentioned lemma.  
Depending on these two cases, by Lemma~\ref{lemma.diamond.weight}, we have one of the two diamonds in~\eqref{dw} 
(in $\Xafz(\lambda)$). 
Let $\beta':=r_\alpha\beta$. We will need the fact that $\beta$ and $\beta'$ are in $\Phi^+$ or $\delta-\Phi^+$ 
(not necessarily both in the same set), by Lemma~\ref{lemma.cover}~(1). 

Assume that we have the left diamond in \eqref{dw}, as the reasoning is completely similar for the right diamond (we simply 
interchange the statements of the form ``bottom implies top'' and ``top implies bottom'' provided by 
Lemmas~\ref{lemma.diamond.weight} and \ref{lemma.diamond.PQBG}). Lemma~\ref{lemma.simple.root} tells us that, by projecting 
its  edges pointing northwest (labeled by the simple root $\alpha$) via the map $\clproj$, we obtain two Bruhat edges or 
two quantum edges in the PQBG (depending on $\alpha\ne\alpha_0$ or $\alpha=\alpha_0$, respectively). Moreover, by 
Lemma~\ref{lem.proj}, the 
four vertices of the projected diamond and its top left edge are labeled as in left diamond in \eqref{dq1} (or \eqref{dq2}, which has the same labels), and \eqref{dq4}, respectively, where $\gamma$ is defined as in Lemma \ref{lem.proj}; indeed, if 
$\gamma^{\prime}$ is defined with respect to $\beta^{\prime}$ and $r_{\alpha} w$ as $\gamma$ is defined with respect to $\beta$ and $w$ in Lemma \ref{lem.proj}, then $\gamma^{\prime}=\gamma$ in the first case, and $\gamma^{\prime}=z(\gamma)$ in the second case. Since $\langle \alpha^\vee, \beta\rangle>0$, the height of $\beta'$ is strictly smaller than the height of $\beta$; 
so by induction we know that the top left edge of the projected diamond is a Bruhat or quantum edge in the PQBG,
depending on $\beta'\in\Phi^+$ or $\beta'\in\delta-\Phi^+$, respectively. 

By Lemma \ref{lem.covers}, we have one of the following three cases:
\begin{equation}\label{c31}
	(\beta\in\Phi^+,\alpha\ne\alpha_0)\,,\qquad (\beta\in\delta-\Phi^+,\alpha\ne\alpha_0)\,,\qquad (\beta\in\delta-\Phi^+,\alpha=\alpha_0)\,.
\end{equation}
By calculating $\beta'=r_{\alpha}\beta$, we deduce that, in the mentioned three cases, we have 
\begin{equation}\label{c32}
	\beta'\in\Phi^+\,,\qquad\beta'\in\delta-\Phi^+\,,\qquad\beta'\in\Phi^+\,,
\end{equation}
respectively. For the last computation, let $\beta=\delta-\overline{\beta}$ and write
\begin{equation}\label{ra0db}
	\beta'=r_{\alpha_0}(\delta-\overline{\beta})=r_\theta t_{-\theta^\vee}(\delta-\overline{\beta})
	=-r_\theta \overline{\beta}+(1-\langle\theta^\vee, \overline{\beta}\rangle)\delta\,;
\end{equation}
here the coefficient of $\delta$ needs to be $0$ or $1$, as noted above, but the second case cannot happen since 
\begin{equation}\label{nzero}
	\langle \theta^\vee, \overline{\beta} \rangle=\langle \alpha^\vee, \beta \rangle\ne 0\,.
\end{equation} 
Hence, in the three cases in \eqref{c31} and \eqref{c32}, the top two edges of the projected diamond (and their vertices) are 
as in the left diamonds in \eqref{dq1}, \eqref{dq2}, and \eqref{dq4}, respectively. By Remark \ref{lr-relabel}, these three diamonds coincide, up to relabeling, with the right diamonds in \eqref{dq1}, \eqref{dq2}, and \eqref{dq3}, respectively. Therefore, we 
can apply the statements in Lemma \ref{lemma.diamond.PQBG} associated with the latter diamonds (stating 
that their top two edges imply their bottom two edges) to deduce that the projection 
of the edge $\mu\lessdot\nu$ is as claimed, namely a Bruhat edge in the first case, and a quantum edge in the last two 
cases (in the PQBG). Note that the condition $\gamma\ne |w^{-1}\alpha|$ needed in the 
first case is satisfied since  $\beta=|w\gamma|$ in this case and $\beta\ne\alpha$; here $|\alpha| = \pm \alpha$
depending on whether $\alpha$ is positive or negative. In addition, the condition 
$\langle \gamma^\vee, w^{-1}\theta \rangle=\langle w\gamma^\vee, \theta \rangle\ne 0$ 
needed in the third case is precisely \eqref{nzero}. This concludes the induction step.
\end{proof}

Now let us turn to the converse statement.

\begin{proof}[Proof of $(\Leftarrow)$ in Theorem {\rm \ref{theorem.level 0 weight poset}}] 
Assume that $\clproj(\mu)=w$ and we have the edge in the PQBG
$\begin{diagram}\dgARROWLENGTH=1.5em\node{w} \arrow{e,t}{\gamma}\node{wr_\gamma=w'}\end{diagram}$ or 
$\begin{diagram}\dgARROWLENGTH=1.5em\node{w} \arrow{e,t,..}{\gamma}\node{\lfloor wr_\gamma\rfloor=w'}\end{diagram}$. 
Defining $\beta$ as in \eqref{ahbh}, we claim that $\nu:=r_\beta\mu$ satisfies the conditions in the theorem. Indeed, note first that 
$\clproj(\nu)=w'$, by Lemma \ref{lem.proj}~(1). We now proceed by induction on the height of $\beta$. If $\beta$ is an affine simple root, 
the conclusion follows directly from Lemma \ref{lemma.simple.root}. If $\beta$ is not a simple root, we apply 
Lemma~\ref{lem.covers.PQBG}; 
this gives an affine simple root $\alpha\ne\beta$ satisfying $\langle \alpha^\vee, \beta \rangle>0$ and either 
condition (1) or (2) in the mentioned 
lemma. Assume that condition (1) holds, as the reasoning is completely similar if condition (2) holds (we simply interchange the statements of the form ``bottom implies top'' and ``top implies bottom'' provided by Lemmas~\ref{lemma.diamond.PQBG} 
and~\ref{lemma.diamond.weight}).

By Lemma~\ref{lem.covers.PQBG}, we have one of the following three cases:
\begin{equation}\label{c31h}
	(\beta\in\Phi^+,\alpha\ne\alpha_0)\,,\qquad (\beta\in\delta-\Phi^+,\alpha\ne\alpha_0)\,,\qquad (\beta\in\delta-\Phi^+,\alpha=\alpha_0)\,.
\end{equation}
By Lemma \ref{lemma.diamond.PQBG}, we have the left diamonds in \eqref{dq1}, \eqref{dq2}, and \eqref{dq4}, respectively. Note that the conditions $\gamma\ne w^{-1}\alpha$ and $\gamma\ne -w^{-1}\theta$ needed in the first and third cases, respectively, are satisfied since $\beta\ne\alpha$, where we recall the definition of $\beta$ in (\ref{ahbh}); in addition, the condition 
$\langle \gamma^\vee, w^{-1}\theta\rangle\ne 0$ needed in the third case follows from $\langle \alpha^\vee, \beta \rangle>0$, 
cf. \eqref{nzero} above. Let $\beta'$ be defined as in \eqref{ahbh} for the top left edge of these diamonds. It is not hard to check 
that in all cases $\beta'=r_{\alpha}\beta$. For instance, letting $\beta=\delta-\overline{\beta}$ in the third case (where 
$\overline{\beta}=-w\gamma\in\Phi^+$), we have
\[
	\beta'=\lfloor r_\theta w\rfloor z(\gamma)=r_\theta w\gamma=-r_\theta\overline{\beta}=r_{\alpha_0}(\delta-\overline{\beta})\,;
\]
here the last equality follows from~\eqref{ra0db} and~\eqref{nzero} above, as well as the well-known fact that 
$\langle \theta^\vee, \overline{\beta}\rangle$ 
can only be $0$ or $1$ if $\overline{\beta}\ne\theta$ (which is clearly true). 

Since $\langle \alpha^\vee, \beta \rangle>0$, the height of $\beta'=r_{\alpha}\beta$ is strictly smaller than the height of $\beta$. 
Therefore, we can use induction (together with the calculation of $\clproj(r_{\alpha}\mu)$ from Lemma \ref{lem.proj}~(2)) to deduce that we have a cover
\[
	\begin{diagram}\dgARROWLENGTH=1.5em\node{r_{\alpha}\mu} \arrow{e,t}{\beta'}\node{r_{\beta'}r_{\alpha}\mu=r_{\alpha}r_{\beta}\mu}
	\end{diagram}
\]
in $\Xafz(\lambda)$. On the other hand, by Lemma~\ref{lem.proj}, we can see that $r_{\beta}\mu$ and $r_{\alpha}r_{\beta}\mu$ project to 
the vertices of the top right edge of the left diamonds in \eqref{dq1}, \eqref{dq2}, and \eqref{dq4}, depending on the case. Therefore, 
by Lemma~\ref{lemma.simple.root}, we also have the cover
\[
	\begin{diagram}\dgARROWLENGTH=1.5em\node{r_{\beta}\mu} \arrow{e,t}{\alpha}\node{r_{\alpha}r_{\beta}\mu}\end{diagram}
\]
in $\Xafz(\lambda)$. We now proved that we have the top two edges in the left diamond in~\eqref{dw}. As $\beta\ne\alpha$, we can now apply the statement of 
Lemma~\ref{lemma.diamond.weight} corresponding to the right diamond in~\eqref{dw} (which is just a relabeling of the left one) 
to deduce that we have the cover $\mu\lessdot r_{\beta}\mu$ labeled by $\beta$ in $\Xafz(\lambda)$. This concludes the induction step.
\end{proof}

\subsection{Connectivity of the parabolic quantum Bruhat graph and quantum length}
\label{subsection.quantum length}

In this subsection we show that the PQBG is strongly connected
when using only simple reflections. 
For the QBG, this result is~\cite[Theorem 4.2]{HST}.

We use the following notation:
\begin{equation*}
\ti{\alpha}_{i}:=
\begin{cases}
\alpha_{i} & \text{if $i \ne 0$}, \\[1.5mm]
-\theta & \text{if $i = 0$},
\end{cases}
\qquad
s_{i}:=
\begin{cases}
r_{i} & \text{if $i \ne 0$}, \\[1.5mm]
r_{\theta} & \text{if $i = 0$}.
\end{cases}
\end{equation*}
Also, in this subsection we do not draw quantum edges in the PQBG  by dotted lines.

%
%
\begin{lem} \label{lem:ql2}
For each $u,\,v \in W^{J}$, 
there exist a sequence $u=x_{0},\,x_{1},\,\dots,\,x_{n}=v$ of 
elements of $W^{J}$ and a sequence $i_{1},\,i_{2},\,\dots,\,i_{n} \in I_{\af}$ 
such that $x_{k+1}=\mcr{ s_{i_{k+1}}x_{k} }$ 
with $x_{k}^{-1}\ti{\alpha}_{i_{k+1}} \in \Phi^{+} \setminus \Phi^{+}_{J}$ 
for each $0 \le k \le n-1$. 
\end{lem}

\begin{remark} \label{remark.strongly connected}
Keep the notation in the lemma above. 
We see from Lemma \ref{lemma.simple.root} and Lemma \ref{lemma.cover}~(2) that
\begin{equation*}
u=x_{0} 
 \xrightarrow{x_{0}^{-1}\ti{\alpha}_{i_{1}}}
x_{1} 
 \xrightarrow{x_{1}^{-1}\ti{\alpha}_{i_{2}}} \cdots \cdots
 \xrightarrow{x_{n-2}^{-1}\ti{\alpha}_{i_{n-1}}}
x_{n-1}
 \xrightarrow{x_{n-1}^{-1}\ti{\alpha}_{i_{n}}}
x_{n}=v
\end{equation*}
in the PQBG. In particular, 
the PQBG is strongly connected when using only
simple reflections (i.e.,  for each $u,\,v \in W^{J}$, there exists a directed path from $u$ to $v$ 
in the PQBG, where the edges correspond to 
multiplying on the left by simple reflections). Note that a similar result for the QBG is stated in~\cite[Lemma 1 (1)]{Po}.
\end{remark}

We are now ready to define the notion of {\em quantum length} of an element in $W^J$. This will be
used in the proofs of the tilted Bruhat Theorem~\ref{thm:tilted} and the generalization of 
Postnikov's lemma (Proposition~\ref{prop:weight}).
\begin{definition}
\label{definition.quantum length}
Let $u \in W^J$. We see from Lemma~{\rm \ref{lem:ql2}} (with $v=e$, where $e$ is the identity
in $W_J$) and Remark~{\rm \ref{remark.strongly connected}} that there exist a sequence 
$u=x_{0},\,x_{1},\,\dots,\,x_{n}=v$ of elements of $W^J$ and a sequence 
$i_{1},\,i_{2},\,\dots,\,i_{n} \in I_{\af}$ 
such that 
\begin{equation*}
u=x_{0} 
 \xrightarrow{x_{0}^{-1}\ti{\alpha}_{i_{1}}} 
x_{1} 
 \xrightarrow{x_{1}^{-1}\ti{\alpha}_{i_{2}}} \cdots \cdots
 \xrightarrow{x_{n-2}^{-1}\ti{\alpha}_{i_{n-1}}}
x_{n-1}
 \xrightarrow{x_{n-1}^{-1}\ti{\alpha}_{i_{n}}}
x_{n}=e
\end{equation*}
in the PQBG. 
We define the quantum length $\qlj(u)$ of $u$ to be 
the minimal of the length $n$ of such sequences.

When $J=\emptyset$, we denote the quantum length (in the QBG) by $\ql(u)$. 
\end{definition}

\begin{proof}[Proof of Lemma~{\rm \ref{lem:ql2}}]
Let $\lambda$ be a dominant weight such that 
$\bigl\{j \in I \mid \pair{\alpha_j^\vee}{\lambda}=0\bigr\}=J$; 
note that the stabilizer of $\lambda$ in $W$ is identical to $W_{J}$, 
and hence $W\lambda \cong W/W_{J} = W^{J}$. 
Set $\mu:=u\lambda$ and $\nu:=v\lambda$. 
We see from \cite[Lemma 1.4]{AK} 
that there exists $i_{1},\,i_{2},\,\dots,\,i_{n} \in I_{\af}$ 
such that 
\begin{equation*}
\begin{cases}
s_{i_n} \cdots s_{i_2}s_{i_1}\mu=\nu, & \\[1.5mm]
\pair{\alpha^\vee_{i_{k+1}}}{s_{i_{k}} \cdots s_{i_{2}}s_{i_{1}}\mu} > 0 
& \text{for all $0 \le k \le n-1$}.
\end{cases}
\end{equation*}
For each $0 \le k \le n$, we define $x_{k} \in W^{J}$ to be 
the minimal coset representative for the coset containing 
$s_{i_k} \cdots s_{i_2}s_{i_1}$; 
note that $x_{0}=u$ and $x_{n}=v$. 
It is obvious that $x_{k+1}=\mcr{s_{i_{k+1}}x_{k}}$ for every $0 \le k \le n-1$. 
Also, because 
\begin{equation*}
\pair{\alpha^\vee_{i_{k+1}}}{x_{k}\lambda}=
\pair{\alpha^\vee_{i_{k+1}}}{s_{i_{k}} \cdots s_{i_{2}}s_{i_{1}}\mu}> 0,
\end{equation*}
it follows immediately that $x_{k}^{-1}\ti{\alpha}_{i_{k+1}} 
\in \Phi^{+} \setminus \Phi^{+}_{J}$. 
Thus we have proved the lemma.
\end{proof}

\section{Tilted Bruhat theorem}
\label{section.tilted}

\subsection{Tilted Bruhat order}\label{section.tbo}
Given $u\in W$ the \emph{$u$-tilted Bruhat order} on $W$ \cite{BFP} is defined by $w_1 \preccurlyeq_u w_2$
if there is a shortest path in the quantum Bruhat graph $\QB(\W)$ from $u$ to $w_2$ that passes through $w_1$.
More precisely, if we denote by $\sdp{w_{1}}{w_{2}}$ the length of a shortest directed path from 
$w_{1}$ to $w_{2}$ in the quantum Bruhat graph $\QB(\W)$, then for $u,\,w_{1},\,w_{2} \in W$, 
\begin{equation*}
w_{1} \preccurlyeq_{u} w_{2} \quad \Longleftrightarrow \quad 
 \sdp{u}{w_{2}}=\sdp{u}{w_{1}}+\sdp{w_{1}}{w_{2}}.
\end{equation*}
It was shown in~\cite{BFP} that this is a partial order.
In~\cite[Theorem 4.8]{LS} it was reproved by showing that $(\W,\preccurlyeq)$
is (dual to) an induced subposet of the affine Bruhat order.

Here we prove a property of the $u$-tilted Bruhat order with respect to any parabolic subgroup
$\WP \subset W$ of the finite Weyl group.
\begin{thm}[Tilted Bruhat Theorem] \label{thm:tilted}
For every $u, z \in \W$ and any parabolic subgroup $\WP\subset\W$, 
the coset $z\WP$ contains a unique $\preccurlyeq_u$-minimal element.
\end{thm}

The tilted Bruhat theorem is a QBG analogue of the Deodhar lift~\cite{D}
(see also~\cite[Proposition 3.1]{LeSh}), which states that if $\tau \in \W/\WP$ and $v\in \W$ such that
$v\WP \le \tau$ in $\W/\WP$, then the set
\[
	\{ w\in \W \mid v\le w \text{ and } w\WP = \tau\}
\]
has a Bruhat-minimum.

We start by stating a weaker version of Theorem~\ref{thm:tilted}, which is easily proved. 

\begin{prop}\label{prop:onemin} 
Fix $u,z \in W$. There exists a unique element $x \in z W_{J}$ such that
the distance $\sdp{u}{x}$ attains its minimum value.
\end{prop}

Proposition \ref{prop:onemin} suffices for our main application in \cite{LNSSS},
namely for bijecting the models for KR crystals based on projected LS-path and quantum Bruhat chains. However, an explicit 
construction of this bijection depends on an algorithm for determining $x=x_0\in z W_J$ 
minimizing $\sdp{u}{x}$; such an algorithm 
is given in the proof of Theorem \ref{thm:tilted}. The proof of Proposition \ref{prop:onemin} relies on the {\em shellability} of the 
QBG with respect to a reflection ordering on the positive roots \cite{Dy}, which we now recall. 

\begin{thm}\cite{BFP}\label{thm:shell} 
Fix a reflection ordering on $\Phi^+$.
\begin{enumerate}
\item For any pair of elements $v,w\in W$, there is a unique path from $v$ to $w$ in the quantum 
Bruhat graph $\QB(W)$ such that its sequence of edge labels is strictly increasing (resp., decreasing) 
with respect to the reflection ordering.
\item The path in {\rm (1)} has the smallest possible length $\sdp{v}{w}$ and is lexicographically minimal 
(resp., maximal) among all shortest paths from $v$ to $w$.
\end{enumerate}
\end{thm}

The proof of Proposition \ref{prop:onemin} is immediate once we have the following two easy lemmas. 
These are in terms of a reflection ordering whose top (also called an initial section) consists of the roots in $\Phi^+\setminus\Phi_J^+$, while its bottom is a reflection ordering on $\Phi_J^+$. Such an order was constructed 
in~\cite[Section 4.3]{LeSh} in terms of a dominant weight $\lambda$ whose stabilizer is $W_J$. The roots in $\Phi^+\setminus\Phi_J^+$ are ordered according to the lexicographic order on their images in ${\mathbb Q}^r$ via the injective map
\[\alpha\mapsto\frac{1}{\langle \alpha^\vee, \lambda \rangle}(c_1,\ldots,c_r)\,,\]
where $\alpha^\vee=c_1\alpha_1^\vee+\cdots+c_r\alpha_r^\vee$ expresses $\alpha^\vee$ in the basis of simple coroots (on which we fix an order). 
For the roots in $\Phi_J^+$, we choose any reflection ordering. 

\begin{lem}\label{lem1}
Assume that $\sdp{u}{x}$, as a function of $x\in z W_J$, has a minimum at $x=x_0$. Then the path from $u$ to $x_0$ 
with increasing edge labels has all its labels in $\Phi^+\setminus\Phi_J^+$.
\end{lem}

\begin{proof} The mentioned path has length $\sdp{u}{x_0}$, by Theorem \ref{thm:shell}~(2). 
Assume that it has at least one label in $\Phi_J^+$. By the structure of our particular reflection ordering, 
all of these labels must be at the end of the path. This means that the tail of the path starting with some $x_1\ne x_0$ consists entirely of 
elements in $z W_J$. Since $\sdp{u}{x_1} < \sdp{u}{x_0}$, we reached a contradiction.
\end{proof} 

\begin{lem}\label{lem2}
Assume that the paths with increasing edge labels from $u$ to two elements $x_0,x_1$ in $zW_J$ 
have all labels in $\Phi^+\setminus\Phi_J^+$. Then $x_0=x_1$.
\end{lem}

\begin{proof} Assume $x_0\ne x_1$. The induced subgraph of $\QB(W)$ on $z W_J$, 
to be denoted $\QB(z W_J)$, is isomorphic to $\QB(W_J)$ under the map $w\mapsto \lfloor z\rfloor w$ for $w\in W_J$ (this is immediate from definitions and the length-additive factorization of the elements in $z W_J$). Thus, by 
Theorem~\ref{thm:shell}~(1), we can consider the path from $x_0$ to $x_1$ in $\QB(z W_J)$ with 
increasing edge labels (in $\Phi_J^+$). By concatenating this path with the one from $u$ to $x_0$ 
in the hypothesis (whose labels are in $\Phi^+\setminus\Phi_J^+$), we obtain a path with increasing 
edge labels from $u$ to $x_1$. But this path is clearly different from the one in the hypothesis between 
the same vertices. This  contradicts the uniqueness statement in Theorem \ref{thm:shell}~(1). 
\end{proof}

\begin{proof}[Proof of Proposition {\rm \ref{prop:onemin}}] 
This is immediate by combining Lemmas \ref{lem1} and \ref{lem2}.
\end{proof}

Next we prepare for the proof of the tilted Bruhat Theorem~\ref{thm:tilted}.

\subsection{Preliminaries}

We use the same notation for $\ti{\alpha}_{i}$ and $s_i$ as in Section~\ref{subsection.quantum length}.
In addition, we denote the identity of $\W$ by $e$.
%
%
\begin{remark} \label{rem:arrow}
Let $w \in W$, and $i \in I_{\af}$. 
If $w^{-1}\ti{\alpha}_{i}$ is positive, then 
we have 
\begin{equation*}
\begin{diagram}
\node{w} \arrow{e,t}{w^{-1}\ti{\alpha}_{i}} \node{s_{i}w}
\end{diagram}
\end{equation*}
in the QBG by Theorem~\ref{theorem.level 0 weight poset}. 
Here, this arrow is an Bruhat arrow (resp., quantum arrow) if $i \ne 0$ (resp., $i=0$).
\end{remark}
%
%
The following lemma will be needed in the proof of the tilted Bruhat Theorem~\ref{thm:tilted},  
in generalizing Postnikov's lemma (in Section \ref{section.poslem}), as well as in our second paper. Only certain parts of the lemma are needed in each
of the mentioned proofs; for instance, in this section we only need weaker versions 
of parts (1) and (3), and no reference to the weights of the considered paths.
In the sequel, the symbol $\equiv$ means equivalence modulo $Q_{J}^{\vee}$.

\begin{lem} \label{lem:diamond1}
Let $w_{1},\,w_{2} \in W^{J}$, and let $j \in I_{\af}$. 
Let
%
%
\begin{equation} \label{eq:bp0}
\bp : 
w_1=
x_{0} \stackrel{\gamma_{1}}{\longrightarrow} 
x_{1} \stackrel{\gamma_{2}}{\longrightarrow}
\cdots \stackrel{\gamma_{n}}{\longrightarrow} 
x_{n}=w_2
\end{equation}
be a directed path from $w_{1}$ to $w_{2}$ of length $n$ in the PQBG. 
In addition, $\lambda$ is a dominant weight with stabilizer $W_J$.

{\rm (1)} 
If $\pair{\ti{\alpha}_{j}^{\vee}}{w_{2}\lambda} < 0$, and 
there exists $0 \le k \le n$ such that 
$\pair{\ti{\alpha}_{j}^{\vee}}{x_{k}\lambda} \ge 0$, 
then there exists a directed path $\bp'$ from 
$w_{1}$ to $\mcr{s_{j}w_{2}}$ of length $n-1$ in the PQBG 
such that
\begin{equation*}
\wt(\bp') \equiv \wt(\bp) + 
 \delta_{j,\,0}w_{2}^{-1}\ti{\alpha}_{j}^{\vee}.
\end{equation*}

%

{\rm (2)} If $\pair{\ti{\alpha}_{j}^{\vee}}{w_{2}\lambda} < 0$ and 
$\pair{\ti{\alpha}_{j}^{\vee}}{w_{1}\lambda} < 0$, 
then there exists a directed path $\bp'$ from 
$\mcr{s_{j}w_{1}}$ to $\mcr{s_{j}w_{2}}$ of length $n$ in the PQBG 
such that
\begin{equation*}
\wt(\bp') \equiv \wt(\bp) 
-\delta_{j,\,0}w_{1}^{-1}\ti{\alpha}_{j}^{\vee}
+\delta_{j,\,0}w_{2}^{-1}\ti{\alpha}_{j}^{\vee}. 
\end{equation*}

{\rm (3)} If $\pair{\ti{\alpha}_{j}^{\vee}}{w_{1}\lambda} > 0$, and 
there exists $0 \le k \le n$ such that 
$\pair{\ti{\alpha}_{j}^{\vee}}{x_{k}\lambda} \le 0$, 
then there exists a directed path $\bp'$ from 
$\mcr{s_{j}w_{1}}$ to $w_{2}$ of length $n-1$ in the PQBG 
such that
\begin{equation*}
\wt(\bp') \equiv \wt(\bp) -
\delta_{j,\,0}w_{1}^{-1}\ti{\alpha}_{j}^{\vee}. 
\end{equation*}

{\rm (4)} If $\pair{\ti{\alpha}_{j}^{\vee}}{w_{1}\lambda} > 0$, and 
$\pair{\ti{\alpha}_{j}^{\vee}}{w_{2}\lambda} > 0$, 
then there exists a directed path $\bp'$ from 
$\mcr{s_{j}w_{1}}$ to $\mcr{s_{j}w_{2}}$ of length $n$ 
in the PQBG such that
\begin{equation*}
\wt(\bp') \equiv \wt(\bp) -
\delta_{j,\,0}w_{1}^{-1}\ti{\alpha}_{j}^{\vee} + 
\delta_{j,\,0}w_{2}^{-1}\ti{\alpha}_{j}^{\vee}. 
\end{equation*}

{\rm (5)} In each of parts above, 
if the directed path $\bp$ is shortest in the PQBG, then 
the directed path $\bp'$ is also shortest in the PQBG. 
\end{lem}
\begin{proof}
We will omit the proofs of parts (3) and (4),
since they are similar to those of parts (1) and (2), 
respectively; alternatively, we can reduce the former to the latter by using Proposition \ref{P:circ}.

(1) Since $\pair{\ti{\alpha}_{j}^{\vee}}{w_{2}\lambda} < 0$, 
we have $w_{2}^{-1}\ti{\alpha}_{j} \in \Phi^{-} \setminus \Phi^{-}_{J}$. 
Thus it follows from Propositions~\ref{P:leftrPQBG}\,(1) and 
\ref{P:thetaPQBG}\,(3) that
\begin{equation*}
\begin{diagram}
\node{}
\node{\mcr{s_{j}w_2}} \arrow{s} \\
\node{x_{n-1}} \arrow{e,t}{\gamma_{n}} 
\node{w_2}
\end{diagram}
\end{equation*}
If $\pair{\ti{\alpha}_{j}^{\vee}}{x_{n-1}\lambda} < 0$, then
we can apply the assertion for the right diagram in 
Lemma~\ref{lemma.diamond.PQBG} (the diamond lemma) to this diagram; 
choose a suitable right diagram in Lemma~\ref{lemma.diamond.PQBG}, 
depending on the types (Bruhat or quantum) of 
the edges $x_{n-1} \stackrel{\gamma_{n}}{\longrightarrow} w_{2}$ and 
$\mcr{s_{j}w_{2}} \rightarrow w_{2}$, and also on the value of 
$\pair{\gamma_{n}^{\vee}}{x_{n-1}^{-1}\theta}$ if $j=0$. 
Thus we obtain
\begin{equation*}
\begin{diagram}
\node{}
\node{\mcr{s_{j}x_{n-1}}} \arrow{s} \arrow{e,t}{z_{n}\gamma_{n}} 
\node{\mcr{s_{j}w_2}} \arrow{s} \\
\node{x_{n-2}} \arrow{e,t}{\gamma_{n-1}} 
\node{x_{n-1}} \arrow{e,t}{\gamma_{n}} 
\node{w_2}
\end{diagram}
\end{equation*}
for some $z_{n} \in W_{J}$; 
by Lemma~\ref{lemma.diamond.PQBG}, 
the weights of the directed paths 
from $\mcr{s_{j}x_{n-1}}$ to $w_2$ 
appearing in the diagram above are 
all congruent modulo $Q_{J}^{\vee}$. 
Next, if $\pair{\ti{\alpha}_{j}^{\vee}}{x_{n-2}\lambda} < 0$, 
then by the same reasoning as above, we obtain
\begin{equation*}
\begin{diagram}
\node{}
\node{\mcr{s_{j}x_{n-2}}} \arrow{s} \arrow{e,t}{z_{n-1}\gamma_{n-1}}
\node{\mcr{s_{j}x_{n-1}}} \arrow{s} \arrow{e,t}{z_{n}\gamma_{n}} 
\node{\mcr{s_{j}w_2}} \arrow{s} \\
\node{x_{n-3}} \arrow{e,t}{\gamma_{n-2}} 
\node{x_{n-2}} \arrow{e,t}{\gamma_{n-1}} 
\node{x_{n-1}} \arrow{e,t}{\gamma_{n}} 
\node{w_2}
\end{diagram}
\end{equation*}
for some $z_{n-1} \in W_{J}$; 
by Lemma~\ref{lemma.diamond.PQBG}, 
the weights of the directed paths 
from $\mcr{s_{j}x_{n-2}}$ to $w_{2}$ 
appearing in the diagram above are 
all congruent modulo $Q_{J}^{\vee}$. 
Continue this procedure until 
$\pair{\ti{\alpha}_{j}^{\vee}}{x_{k}\lambda} \ge 0$ 
for the first time. Then we have
\begin{equation*}
\begin{diagram}
\node{}
\node{\mcr{s_{j}x_{k+1}}} \arrow{s} \arrow{e,t}{z_{k+2}\gamma_{k+2}}
\node{\cdots\cdots} \arrow{e,t}{z_{n}\gamma_{n}} 
\node{\mcr{s_{j}w_2}} \arrow{s} \\
\node{x_{k}} \arrow{e,t}{\gamma_{k+1}} 
\node{x_{k+1}} \arrow{e,t}{\gamma_{k+2}} 
\node{\cdots\cdots} \arrow{e,t}{\gamma_{n}} 
\node{w_2}
\end{diagram}
\end{equation*}
where $z_{k+2},\,z_{k+3},\,\dots,\,z_{n} \in W_{J}$; 
by Lemma~\ref{lemma.diamond.PQBG}, the weights of the directed paths 
from $\mcr{s_{j}x_{k+1}}$ to $w_{2}$ 
appearing in the diagram above are all congruent 
modulo $Q_{J}^{\vee}$. 
Because $\pair{\ti{\alpha}_{j}^{\vee}}{x_{k+1}\lambda} < 0$ and 
$\pair{\ti{\alpha}_{j}^{\vee}}{x_{k}\lambda} \ge 0$, 
we deduce 
from Lemma~\ref{lemma.littelmann}\,(1) and 
Theorem~\ref{theorem.level 0 weight poset}
that $\mcr{s_{j}x_{k+1}}=x_{k}$; 
in this case, the type (Bruhat or quantum) of the edges 
$x_{k} \stackrel{\gamma_{k+1}}{\longrightarrow} x_{k+1}$ and 
$\mcr{s_{j}x_{k+1}} \rightarrow x_{k+1}$ are the same, 
and their weights are congruent modulo $Q_{J}^{\vee}$. 
Concatenating with the remaining edges in $\bp$, we obtain
%
%
\begin{equation} \label{eq:d1}
\begin{diagram}
\node{}
\node{}
\node{}
\node{\mcr{s_{j}x_{k+1}}} \arrow{s} \arrow{e,t}{z_{k+2}\gamma_{k+2}}
\node{\cdots\cdots} \arrow{e,t}{z_{n}\gamma_{n}} 
\node{\mcr{s_{j}w_2}} \arrow{s} \\
\node{w_1}
\arrow{e,t}{\gamma_{1}}
\node{\cdots}
\arrow{e,t}{\gamma_{k}}
\node{x_{k}} \arrow{ne,t,-}{=} \arrow{e,t}{\gamma_{k+1}} 
\node{x_{k+1}} \arrow{e,t}{\gamma_{k+2}} 
\node{\cdots\cdots} \arrow{e,t}{\gamma_{n}} 
\node{w_2}
\end{diagram}
\end{equation}
Set
%
%
\begin{equation} \label{eq:bp1}
\bp' : w_1=x_{0} 
\stackrel{\gamma_{1}}{\longrightarrow} \cdots 
\stackrel{\gamma_{k}}{\longrightarrow} x_{k}=\mcr{s_{j}x_{k+1}} 
\stackrel{z_{k+2}\gamma_{k+2}}{\longrightarrow} \cdots 
\stackrel{z_{n}\gamma_{n}}{\longrightarrow}  \mcr{s_{j}x_{n}}=\mcr{s_{j}w_2};
\end{equation}
the length of $\bp'$ is equal to $n-1$. 
Also, in the diagram \eqref{eq:d1}, we set
\begin{align*}
 & \bq_{0} : w_1=x_{0} 
\stackrel{\gamma_{1}}{\longrightarrow} \cdots 
\stackrel{\gamma_{k}}{\longrightarrow} x_{k} \\
 & \bq_1 : 
\begin{diagram}
\node{\mcr{s_{j}x_{k+1}}} \arrow{e,t}{z_{k+2}\gamma_{k+2}}
\node{\cdots\cdots} \arrow{e,t}{z_{n}\gamma_{n}} 
\node{\mcr{s_{j}w_2}}
\end{diagram} \\
 & \bq_2 : 
\begin{diagram}
\node{x_{k+1}} \arrow{e,t}{\gamma_{k+2}} 
\node{\cdots\cdots} \arrow{e,t}{\gamma_{n}} \node{w_2}
\end{diagram}
\end{align*}
Then we have
\begin{equation*}
\wt(\bq_{1})+ \wt(\mcr{s_{j}w_{2}} \rightarrow w_{2}) \equiv 
\wt(\mcr{s_{j}x_{k+1}} \rightarrow x_{k+1})+\wt(\bq_{2}). 
\end{equation*}
Here, recall that 
$\wt(\mcr{s_{j}x_{k+1}} \rightarrow x_{k+1}) \equiv 
 \wt(x_{k} \stackrel{\gamma_{k+1}}{\longrightarrow} x_{k+1})$. 
Therefore we obtain
\begin{align*}
\wt(\bp') & = \wt (\bq_{0})+\wt(\bq_{1}) \\ 
& 
\equiv \wt(\bq_{0})+
\wt(\mcr{s_{j}x_{k+1}} \rightarrow x_{k+1})+\wt(\bq_{2})-
\wt(\mcr{s_{j}w_{2}} \rightarrow w_{2}) \\
& 
\equiv 
\underbrace{
\wt(\bq_{0})+\wt(x_{k} \stackrel{\gamma_{k+1}}{\longrightarrow} x_{k+1})+
\wt(\bq_{2})}_{=\wt (\bp)} -
\wt(\mcr{s_{j}w_{2}} \rightarrow w_{2}) \\
& = 
\wt(\bp)-
\wt(\mcr{s_{j}w_{2}} \rightarrow w_{2}) 
\equiv
\wt(\bp)+\delta_{j,0}w_{2}^{-1}\ti{\alpha}_{j}^{\vee},
\end{align*}
as desired. 

\medskip

(2) Assume first that $\pair{\ti{\alpha}_{j}^{\vee}}{x_{k}\lambda} < 0$ for all 
$0 \le k \le n$. 
Continuing the procedure in the proof of part (1) above, 
we finally obtain 
\begin{equation*}
\begin{diagram}
\node{\mcr{s_{j}w_1}} \arrow{s} \arrow{e,t}{z_1\gamma_{1}} 
\node{\mcr{s_{j}x_1}} \arrow{s} \arrow{e,t}{z_2\gamma_{2}} 
\node{\mcr{s_{j}x_2}} \arrow{s} \arrow{e,t}{z_3\gamma_{3}} 
\node{\cdots\cdots} \arrow{e,t}{z_{n}\gamma_{n}} 
\node{\mcr{s_{j}w_2}} \arrow{s} \\
\node{w_1} \arrow{e,b}{\gamma_{1}} 
\node{x_1} \arrow{e,b}{\gamma_{2}} 
\node{x_2} \arrow{e,b}{\gamma_{3}} 
\node{\cdots \cdots} \arrow{e,b}{\gamma_{n}}
\node{w_2}
\end{diagram}
\end{equation*}
for some $z_{1},\,z_{2},\,\dots,\,z_{n} \in W_{J}$; 
by Lemma~\ref{lemma.diamond.PQBG}, the weights of the directed paths 
from $\mcr{s_{j}w_1}$ to $w_{2}$ 
appearing in the diagram above are all congruent 
modulo $Q_{J}^{\vee}$. Set
%
%
\begin{equation} \label{eq:bp2}
\bp' : \mcr{s_{j}w_1}=\mcr{s_{j}x_{0}}
\stackrel{z_{1}\gamma_{1}}{\longrightarrow} 
\mcr{s_{j}x_{1}}
\stackrel{z_{2}\gamma_{2}}{\longrightarrow} \cdots 
\stackrel{z_{n}\gamma_{n}}{\longrightarrow} \mcr{s_{j}x_n}=\mcr{s_{j}w_2};
\end{equation}
the length of $\bp'$ is equal to $n$. 
Furthermore, we obtain
\begin{equation*}
\wt(\bp')+\wt(\mcr{s_{j}w_{2}} \rightarrow w_{2}) \equiv
\wt(\mcr{s_{j}w_{1}} \rightarrow w_{1})+\wt(\bp), 
\end{equation*}
and hence 
\begin{align*}
\wt(\bp') & \equiv 
\wt(\bp)+
\wt(\mcr{s_{j}w_{1}} \rightarrow w_{1})-
\wt(\mcr{s_{j}w_{2}} \rightarrow w_{2}) \\
& \equiv 
\wt(\bp)-\delta_{j,0}w_{1}^{-1}\ti{\alpha}_{j}^{\vee}+
\delta_{j,0}w_{2}^{-1}\ti{\alpha}_{j}^{\vee}, 
\end{align*}
as desired. 

\medskip

Now assume that there exists $0 < k < n$ such that 
$\pair{\ti{\alpha}_{j}^{\vee}}{x_{k}\lambda} \ge 0$.
By part (1), there exists a directed path $\bp''$ from 
$w_{1}$ to $\mcr{s_{j}w_{2}}$ of length $n-1$ in the PQBG 
such that
\begin{equation*}
\wt(\bp'') \equiv \wt(\bp) + 
\delta_{j,0}w_{2}^{-1}\ti{\alpha}_{j}^{\vee}. 
\end{equation*}
By concatenating this directed path $\bp''$ and 
the edge $\mcr{s_{j}w_{1}} \longrightarrow w_{1}$, 
we obtain a directed path $\bp'$ from 
$\mcr{s_{j}w_{1}}$ to $\mcr{s_{j}w_{2}}$ of length $n-1+1=n$ 
in the PQBG such that
\begin{equation*}
\wt(\bp') \equiv 
\wt(\bp'')+
 \underbrace{\wt (\mcr{s_{j}w_{1}} \rightarrow w_{1})}_{%
 =-\delta_{j,0}w_{1}^{-1}\ti{\alpha}_{j}^{\vee} }
\equiv 
\wt(\bp) 
-\delta_{j,0}w_{1}^{-1}\ti{\alpha}_{j}^{\vee}+
\delta_{j,0}w_{2}^{-1}\ti{\alpha}_{j}^{\vee},
\end{equation*}
as desired. 

\medskip

(5) We give the proofs only for parts (1) and (2); 
the proofs for the other cases are similar. 
Suppose that in part (1), $\bp$ is shortest in the PQBG, but 
$\bp'$ is not shortest in the PQBG. 
Concatenating a shortest directed path from 
$w_{1}$ to $\mcr{s_{j}w_{2}}$ in the PQBG 
(note that its length is less than $n-1$) and 
$\mcr{s_{j}w_{2}} \rightarrow w_{2}$, 
we obtain a directed path from $w_{1}$ to $w_{2}$ 
whose length is less than $n$. This contradicts 
the assumption that $\bp$ is shortest. 

Suppose that in part (2), 
$\bp$ is shortest in the PQBG, but 
$\bp'$ is not shortest in the PQBG. 
Concatenating a shortest directed path from 
$\mcr{s_{j}w_{1}}$ to $\mcr{s_{j}w_{2}}$ in the PQBG 
(note that its length is less than $n$) and 
$\mcr{s_{j}w_{2}} \rightarrow w_{2}$, 
we obtain a directed path from $\mcr{s_{j}w_{1}}$ to $w_{2}$ 
whose length is less than $n+1$. 
By the assumption of part (2), 
$\pair{\ti{\alpha}_{j}^{\vee}}{s_{j}w_{1}\lambda} > 0$ and 
$\pair{\ti{\alpha}_{j}^{\vee}}{w_{2}\lambda} < 0$.
Therefore it follows from part (3) that 
there exists a directed path from $w_{1}$ to $w_{2}$ 
whose length is less than $n$, which contradicts the 
assumption that $\bp$ is shortest in the PQBG. 

This completes the proof of the lemma. 
\end{proof}

\subsection{Proof of the tilted Bruhat Theorem~{\rm \ref{thm:tilted}}}

\begin{proof}[{}Proof of Theorem~{\rm \ref{thm:tilted}}]
The proof proceeds by induction on $\ql(u)$.
If $\ql(u)=0$, then $u=e$. 
We know from~\cite[p. 435]{BFP} that the $e$-tilted Bruhat order 
$\preccurlyeq_{e}$ on $W$ is just the Bruhat order on $W$. 
Hence, for each $z \in W$, the minimal coset representative 
in $zW_{J}$ is the unique $\preccurlyeq_{e}$-minimal element. 
Therefore the assertion holds.

Assume that $\ql(u) > 0$. 
Let $u=x_{0},\,x_{1},\,\dots,\,x_{n}=e$ be a sequence of 
elements in $W$ satisfying the condition in 
Lemma~\ref{lem:ql2}, with $n=\ql(u)$. Put $v:=x_{1}$; 
note that $\ql(v)=\ql(u)-1$. 
Thus the inductive assumption is: 
\begin{center}
Theorem~\ref{thm:tilted} is true for this $v$ 
(and arbitrary $z \in W$). 
\end{center}

Assume that $v=s_{i}u$ for some $i \in I_{\af}$.
Since $u^{-1}\ti{\alpha}_{i}$ is positive, it follows from 
Remark~\ref{rem:arrow} that 
\begin{equation} \label{eq:u}
\begin{diagram}
\node{u} \arrow{e,t}{u^{-1}\ti{\alpha}_{i}} \node{v=s_{i}u,}
\end{diagram}
\end{equation}
where this arrow is an Bruhat arrow (resp., a quantum arrow) if $i \ne 0$ 
(resp., $i=0$). 

\paragraph{\bf Case 1.} Assume that 
$z^{-1}\ti{\alpha}_{i} \in \Delta^{-} \setminus \Delta_{J}^{-}$; 
note that $(zy)^{-1}\ti{\alpha}_{i}$ is negative for all $y \in W_{J}$. 

By the inductive assumption, there exists a unique minimal element 
in the coset $zW_{J}$ with respect to $\preccurlyeq_{v}$, 
which we denote by $\min (zW_{J},\,\preccurlyeq_{v})$. 
Let $x \in W_{J}$ be such that 
\begin{equation*}
\min (zW_{J},\,\preccurlyeq_{v})=zx.
\end{equation*}
Let us show that $zx \in zW_{J}$ is a unique minimal element 
in the coset $zW_{J}$ with respect to $\preccurlyeq_{u}$, that is, 
\begin{equation*}
\min (zW_{J},\,\preccurlyeq_{u})=zx.
\end{equation*}
Let $y \in W_{J}$ be an arbitrary element in $W_{J}$. 
There exists a shortest directed path from $v$ to $zy$ 
that passes through $zx$: 
\begin{equation*}
v \rightarrow \cdots \rightarrow zx 
 \rightarrow \cdots \rightarrow zy.
\end{equation*}
Concatenating $u \rightarrow v$ of \eqref{eq:u}
and this directed path, we obtain a directed path 
\begin{equation} \label{eq:case1}
u \rightarrow v \rightarrow \cdots \rightarrow zx 
 \rightarrow \cdots \rightarrow zy
\end{equation}
of length $\sdp{v}{zy}+1$. 
Let us show that 
this directed path is shortest. Suppose that 
$\sdp{u}{zy} < \sdp{v}{zy}+1$. 
Recall that $u^{-1}\ti{\alpha}_{i}$ is positive, and 
$(zy)^{-1}\ti{\alpha}_{i}$ is negative. 
By Lemma~\ref{lem:diamond1}\,(3), we obtain a directed path 
from $s_{i}u=v$ to $zy$ whose length is equal to 
$\sdp{u}{zy}-1$. Hence,
\begin{equation*}
\sdp{v}{zy} \le 
\sdp{u}{zy} -1 < \sdp{v}{zy}+1-1=\sdp{v}{zy},
\end{equation*}
which is a contradiction. 
Therefore, the directed path \eqref{eq:case1} is shortest. 

\paragraph{\bf Case 2.} Assume that 
$z^{-1}\ti{\alpha}_{i} \in \Delta^{+} \setminus \Delta_{J}^{+}$;
note that $(zy)^{-1}\ti{\alpha}_{i}$ is positive for all $y \in W_{J}$, 
which implies that $zy \rightarrow s_{i}zy$ by Remark~\ref{rem:arrow}. 

By the inductive assumption, there exists a unique minimal element 
in the coset $s_{i}zW_{J}$ with respect to $\preccurlyeq_{v}$, 
which we denote by $\min (s_{i}zW_{J},\,\preccurlyeq_{v})$. 
Let $x \in W_{J}$ be such that 
\begin{equation*}
\min (s_{i}zW_{J},\,\preccurlyeq_{v})=s_{i}zx.
\end{equation*}
Let us show that $zx \in zW_{J}$ is a unique minimal element 
in the coset $zW_{J}$ with respect to $\preccurlyeq_{u}$;
\begin{equation*}
\min (zW_{J},\,\preccurlyeq_{u})=zx.
\end{equation*}
Let $y \in W_{J}$ be an arbitrary element in $W_{J}$. 
We construct a directed path 
from $u$ to $zy$ that passes through $zx$ as follows: 
First, we construct a directed path from $u$ to $zx$. 
Concatenating $u \rightarrow v$ of \eqref{eq:u}
and a shortest directed path from $v$ to $s_{i}zx$, 
we obtain a directed path from $u$ to $s_{i}zx$ 
of length $\sdp{v}{s_{i}zx}+1$:
\begin{equation*}
u \rightarrow v \rightarrow \cdots \rightarrow s_{i}zx
\end{equation*}
Because $u^{-1}\ti{\alpha}_{i}$ is positive and 
$(s_{i}zx)^{-1}\ti{\alpha}_{i}$ is negative, it follows from 
Lemma~\ref{lem:diamond1}\,(1) that there exists a directed path from 
$u$ to $zx$ of length $\sdp{v}{s_{i}zx}+1-1=\sdp{v}{s_{i}zx}$:
\begin{equation*}
\begin{diagram}
\node[3]{zx}\arrow{s} \\
\node{u}\arrow{ene,l,..}
{\begin{array}{ll}
 {\mbox{\scriptsize ${}^{\exists}$directed path}} \\
 {\mbox{\scriptsize of length $\sdp{v}{s_{i}zx}$}}
 \end{array}
} \arrow{e} \node{\cdots} \arrow{e} \node{s_{i}zx}
\end{diagram}
\end{equation*}
Next, we construct a directed path from $zx$ to $zy$. 
Concatenating $zx \rightarrow s_{i}zx$ and 
a shortest directed path from $s_{i}zx$ to $s_{i}zy$, 
we obtain a directed path from $zx$ to $s_{i}zy$ of 
length $\sdp{s_{i}zx}{s_{i}zy}+1$:
\begin{equation*}
\begin{diagram}
\node[3]{zx}\arrow{s}\arrow{ese,t,..}{\mbox{\scriptsize Concatenation of directed paths}} \\
\node{u}\arrow{ene,l,..}
{\begin{array}{ll}
 \mbox{\scriptsize directed path} \\
 \mbox{\scriptsize of length $\sdp{v}{s_{i}zx}$}
 \end{array}
} \arrow{e} \node{\cdots} \arrow{e} \node{s_{i}zx} 
\arrow{e} \node{\cdots} \arrow{e} \node{s_{i}zy}
\end{diagram}
\end{equation*}
Because $(zx)^{-1}\ti{\alpha}_{i}$ is positive and 
$(s_{i}zy)^{-1}\ti{\alpha}_{i}$ is negative, it follows from 
Lemma~\ref{lem:diamond1}\,(1) that there exists a directed path from 
$zx$ to $zy$ of length $\sdp{s_{i}zx}{s_{i}zy}+1-1=\sdp{s_{i}zx}{s_{i}zy}$. 
\vspace*{5mm}
\begin{equation*}
\begin{diagram}
\node[3]{zx} \arrow{s} \arrow{ese,..}
\arrow[2]{e,t,..}
{\begin{array}{ll}
 \mbox{\scriptsize ${}^{\exists}$directed path} \\
 \mbox{\scriptsize of length $\sdp{s_{i}zx}{s_{i}zy}$}
 \end{array}
} \node[2]{zy} \arrow{s}\\
\node{u}\arrow{ene,l,..}
{\begin{array}{ll}
 \mbox{\scriptsize directed path} \\
 \mbox{\scriptsize of length $\sdp{v}{s_{i}zx}$}
 \end{array}
} \arrow{e} \node{\cdots} \arrow{e} \node{s_{i}zx} 
\arrow{e} \node{\cdots} \arrow{e} \node{s_{i}zy}
\end{diagram}
\end{equation*}
Concatenating the directed paths above, 
we obtain a directed path from $u$ to $zy$ of length 
$\sdp{v}{s_{i}zx}+\sdp{s_{i}zx}{s_{i}zy}=\sdp{v}{s_{i}zy}$ 
(recall that $s_{i}zx \preccurlyeq_{v} s_{i}zy$ by the definition of $x \in W_{J}$)
that passes through $zx$. 

Let us show that 
this directed path is shortest. Suppose that 
$\sdp{u}{zy} < \sdp{v}{s_{i}zy}$. 
Concatenating a shortest directed path from $u$ to $zy$ and 
the directed path $zy \rightarrow s_{i}zy$, 
we obtain a directed path from $u$ to $s_{i}zy$ of the form:
\begin{equation*}
\underbrace{u \rightarrow \cdots \rightarrow zy}_{\text{shortest}} \rightarrow s_{i}zy;
\end{equation*}
note that its length is $\sdp{u}{zy}+1$. 
Because $u^{-1}\ti{\alpha}_{i}$ is positive, and 
$(s_{i}zy)^{-1}\ti{\alpha}_{i}$ is negative, 
it follows from Lemma~\ref{lem:diamond1}\,(3) that
there exists a directed path from $s_{i}u=v$ to $s_{i}zy$ 
of length $\sdp{u}{zy}+1-1=\sdp{u}{zy}$. 
Since $\sdp{u}{zy} < \sdp{v}{s_{i}zy}$, 
this is a contradiction. 

\paragraph{\bf Case 3.} 
Assume that $z^{-1}\ti{\alpha}_{i} \in \Delta_{J}$; 
note that $s_{i}zW_{J}=zW_{J}$. 

By the inductive assumption, there exists a unique minimal element 
in the coset $zW_{J}$ with respect to $\preccurlyeq_{v}$, 
which we denote by $\min (zW_{J},\,\preccurlyeq_{v})$. 
Let $x \in W_{J}$ be such that 
\begin{equation*}
\min (zW_{J},\,\preccurlyeq_{v})=zx.
\end{equation*}

\paragraph{\bf Subcase 3.1} Assume that 
$(zx)^{-1}\ti{\alpha}_{i} \in \Delta_{J}^{+}$. 
Let us show that 
\begin{equation*}
\min (zW_{J},\,\preccurlyeq_{u})=zx.
\end{equation*}
Take an arbitrary $y \in W_{J}$. 

\paragraph{\bf 3.1.1.}
%
Assume first that $(zy)^{-1}\ti{\alpha}_{i} \in \Delta_{J}^{-}$. 
Then we can check in exactly the same way as in Case 1 that 
concatenating $u \rightarrow v$ of \eqref{eq:u} and 
a shortest directed path from $v$ to $zy$ that passes through $zx$ 
gives a shortest directed path from $u$ to $zy$:
\begin{equation*}
\underbrace{u \rightarrow}_{\text{\eqref{eq:u}}}
v 
\underbrace{\rightarrow \cdots \rightarrow zx
 \rightarrow \cdots \rightarrow zy}_{\text{shortest}}.
\end{equation*}

\paragraph{\bf 3.1.2.}
%
Assume next that $(zy)^{-1}\ti{\alpha}_{i} \in \Delta_{J}^{+}$. 
Concatenating $u \rightarrow v$ of \eqref{eq:u} and 
a shortest directed path from $v$ to $zx$, we obtain a directed path 
from $u$ to $zx$ of length $\sdp{v}{zx}+1$:
\begin{equation*}
\underbrace{u \rightarrow}_{\text{\eqref{eq:u}}}
v 
\underbrace{\rightarrow \cdots \rightarrow zx}_{\text{shortest}}.
\end{equation*}
Because $(zx)^{-1}\ti{\alpha}_{i}$ is positive, 
and $(s_{i}zy)^{-1}\ti{\alpha}_{i}$ is negative, 
we see by applying Lemma~\ref{lem:diamond1}\,(1)
to a shortest directed path from $zx$ to $s_{i}zy$ 
that there exists a directed path from $zx$ to $zy$ of length 
$\sdp{zx}{s_{i}zy}-1$:
\begin{equation*}
\begin{diagram}
\node[6]{zy} \arrow{s} \\
\node{u} \arrow{e} \node{v} \arrow{e} \node{\cdots} \arrow{e} \node{zx} 
\arrow{ene,t,..}{
 \begin{array}{ll}
 \mbox{\scriptsize ${}^{\exists}$directed path} \\
 \mbox{\scriptsize of length $\sdp{zx}{s_{i}zy}-1$}
 \end{array}
}
\arrow{e} \node{\cdots} \arrow{e} \node{s_{i}zy}
\end{diagram}
\end{equation*}
Concatenating these directed paths, 
we obtain a directed path from $u$ to $zy$ that passes through $zx$; 
its length is equal to 
\begin{align*}
(\sdp{v}{zx}+1)+(\sdp{zx}{s_{i}zy}-1) & =
\sdp{v}{zx}+\sdp{zx}{s_{i}zy} \\ 
& =\sdp{v}{s_{i}zy};
\end{align*}
recall that $zx \preccurlyeq_{v} s_{i}zy$. 
We can show in exactly the same way as in Case 2 that 
this directed path is shortest. 

\paragraph{\bf Subcase 3.2.} Assume that 
$(zx)^{-1}\ti{\alpha}_{i} \in \Delta_{J}^{-}$. 
Let us show that 
\begin{equation*}
\min (zW_{J},\,\preccurlyeq_{u})=s_{i}zx.
\end{equation*}
Take an arbitrary $y \in W_{J}$. 

\paragraph{\bf 3.2.1.}
%
Assume that $(zy)^{-1}\ti{\alpha}_{i} \in \Delta_{J}^{-}$. 
Concatenating $u \rightarrow v$ of \eqref{eq:u} and 
a shortest directed path from $v$ to $zx$, we obtain a directed path 
from $u$ to $zx$ of length $\sdp{v}{zx}+1$:
\begin{equation*}
\underbrace{u \rightarrow}_{\text{\eqref{eq:u}}}
v 
\underbrace{\rightarrow \cdots \rightarrow zx}_{\text{shortest}}.
\end{equation*}
Because $u^{-1}\ti{\alpha}_{i}$ is positive, 
and $(zx)^{-1}\ti{\alpha}_{i}$ is negative, 
it follows from Lemma~\ref{lem:diamond1}\,(1) 
that there exists a directed path from $u$ to $s_{i}zx$ of length 
$\sdp{v}{zx}+1-1=\sdp{v}{zx}$:
\begin{equation*}
\begin{diagram}
\node[3]{s_{i}zx}\arrow{s} \\
\node{u}\arrow{ene,l,..}
{\begin{array}{ll}
 {\mbox{\scriptsize ${}^{\exists}$directed path}} \\
 {\mbox{\scriptsize of length $\sdp{v}{zx}$}}
 \end{array}
} \arrow{e} \node{\cdots} \arrow{e} \node{zx}
\end{diagram}
\end{equation*}
Concatenating this directed path, 
$s_{i}zx \rightarrow zx$, and a shortest directed path from 
$zx$ to $zy$, we obtain a directed path from $u$ to $zy$ 
that passes through $s_{i}zx$:
\begin{equation*}
\begin{diagram}
\node[3]{s_{i}zx}\arrow{s}\arrow{ese,t,..}{\mbox{\scriptsize Concatenation of directed paths}} \\
\node{u}\arrow{ene,l,..}
{\begin{array}{ll}
 \mbox{\scriptsize directed path} \\
 \mbox{\scriptsize of length $\sdp{v}{zx}$}
 \end{array}
} \arrow{e} \node{\cdots} \arrow{e} \node{zx} 
\arrow{e} \node{\cdots} \arrow{e} \node{zy}
\end{diagram}
\end{equation*}
The length of this directed path is equal to 
\begin{equation*}
\sdp{v}{zx}+1+\sdp{zx}{zy}=\sdp{v}{zy}+1;
\end{equation*}
recall that $zx \preccurlyeq_{v} zy$.
We can show in exactly the same way as the argument in Case 1 that 
this directed path is shortest. 

\paragraph{\bf 3.2.2.}
%
Assume that $(zy)^{-1}\ti{\alpha}_{i} \in \Delta_{J}^{+}$. 
By the same argument as in 3.2.1, we have 
\begin{equation*}
\begin{diagram}
\node[3]{s_{i}zx}\arrow{s} \\
\node{u}\arrow{ene,l,..}
{\begin{array}{ll}
 {\mbox{\scriptsize ${}^{\exists}$directed path}} \\
 {\mbox{\scriptsize of length $\sdp{v}{zx}$}}
 \end{array}
} \arrow{e} \node{\cdots} \arrow{e} \node{zx}
\end{diagram}
\end{equation*}
Concatenating $s_{i}zx \rightarrow zx$ and a shortest directed path 
from $zx$ to $s_{i}zy$, we obtain a directed path from 
$s_{i}zx$ to $s_{i}zy$ of length $\sdp{zx}{s_{i}zy}+1$:
\begin{equation*}
\begin{diagram}
\node[3]{s_{i}zx}\arrow{s}
\arrow{ese,t,..}{\mbox{\scriptsize Concatenation of directed paths}} \\
\node{u}\arrow{ene,l,..}
{\begin{array}{ll}
 \mbox{\scriptsize directed path} \\
 \mbox{\scriptsize of length $\sdp{v}{zx}$}
 \end{array}
} \arrow{e} \node{\cdots} \arrow{e} \node{zx} 
\arrow{e} \node{\cdots} \arrow{e} \node{s_{i}zy}
\end{diagram}
\end{equation*}
Since$(s_{i}zx)^{-1}\ti{\alpha}_{i}$ is positive, and 
$(s_{i}zy)^{-1}\ti{\alpha}_{i}$ is negative, 
it follows from Lemma~\ref{lem:diamond1}\,(1) that 
there exists a directed path from 
$s_{i}zx$ to $zy$ of length $\sdp{zx}{s_{i}zy}+1-1=\sdp{zx}{s_{i}zy}$:
\vspace*{5mm}
\begin{equation*}
\begin{diagram}
\node[3]{s_{i}zx} \arrow{s} \arrow{ese,..}
\arrow[2]{e,t,..}
{\begin{array}{ll}
 \mbox{\scriptsize ${}^{\exists}$directed path} \\
 \mbox{\scriptsize of length $\sdp{zx}{s_{i}zy}$}
 \end{array}
} \node[2]{zy} \arrow{s}\\
\node{u}\arrow{ene,l,..}
{\begin{array}{ll}
 \mbox{\scriptsize directed path} \\
 \mbox{\scriptsize of length $\sdp{v}{zx}$}
 \end{array}
} \arrow{e} \node{\cdots} \arrow{e} \node{zx} 
\arrow{e} \node{\cdots} \arrow{e} \node{s_{i}zy}
\end{diagram}
\end{equation*}
Concatenating these directed paths, 
we obtain a directed path from $u$ to $zy$ 
that passes through $s_{i}zx$; its length is equal to 
\begin{equation*}
\sdp{v}{zx}+\sdp{zx}{s_{i}zy}=
 \sdp{v}{s_{i}zy} \qquad (\because \ zx \preccurlyeq_{v} s_{i}zy).
\end{equation*}
We can show in exactly the same way as the argument in Case 2 that 
this directed path is shortest. 
Thus we have proved Theorem~\ref{thm:tilted}.
\end{proof}

\section{Postnikov's lemma}\label{section.poslem}
\numberwithin{lem}{section}

We now prove a generalization to the PQBG of a lemma due to Postnikov~\cite{Po}.

\begin{prop}[{cf.~\cite[Lemma 1\, (2), (3)]{Po}}] 
\label{prop:weight}
Let $x,\,y \in W^{J}$. Let $\bp$ and $\bq$ be a shortest and an arbitrary directed path 
from $x$ to $y$ in $\QB(W^J)$, respectively. Then there exists 
$h \in Q^{\vee}_{+}$ such that 
\begin{equation*}
\wt (\bq) - \wt (\bp) \equiv h \mod Q_{J}^{\vee}.
\end{equation*}
In particular, if $\bq$ is also shortest, 
then $\wt (\bq) \equiv \wt (\bp) \mod Q_{J}^{\vee}$. 
\end{prop}

\begin{proof}
We prove the first assertion of the proposition 
(for general $x \in W^{J}$) by induction on $\qlj(x)$, 
where $\qlj(x)$ is the quantum length of 
Definition~\ref{definition.quantum length}. 
If $\qlj(x)=0$, then $x=e$. It is easy to see that the $e$-tilted parabolic Bruhat order (cf. Section \ref{section.tbo}) is just the usual parabolic Bruhat order on $W^J$, cf. \cite{BFP}[Section 6]. (Indeed, a path from $e$ to $y$ contains at least $\ell(y)$ edges, because an edge either increases length by $1$ or decreases length; so the path consisting of covers in the parabolic Bruhat order is shortest.) Thus $\bp$ contains no quantum edges in this case, so $\wt(\bp)=0$, which makes the statement to prove trivial.

Now assume that $\qlj(x) > 0$, and 
take $j \in I_{\af}$ such that 
$\qlj(\mcr{s_{j}x})=\qlj(x)-1$; 
note that $x^{-1}\ti{\alpha}_{j} \in 
\Phi^{+} \setminus \Phi_{J}^{+}$, or equivalently, 
$\pair{\ti{\alpha}_{j}^{\vee}}{x\lambda} > 0$, where $\lambda$ is a dominant
weight with stabilizer $W_J$.

\paragraph{\bf{Case 1.}}
%
Assume that $\pair{\ti{\alpha}_{j}^{\vee}}{y\lambda} > 0$.
Applying Lemma~\ref{lem:diamond1}\,(4) and (5) to 
$\bp$ (resp., $\bq$), we obtain a shortest directed path $\bp'$ 
(resp., a directed path $\bq'$) from $\mcr{s_{j}x}$ to $\mcr{s_{j}y}$ 
such that
\begin{align*}
\wt(\bp') & \equiv 
 \wt(\bp)-
\delta_{j,0}x^{-1}\ti{\alpha}_{j}^{\vee}+
\delta_{j,0}y^{-1}\ti{\alpha}_{j}^{\vee}, \\
\wt(\bq') & \equiv 
 \wt(\bq)-
\delta_{j,0}x^{-1}\ti{\alpha}_{j}^{\vee}+
\delta_{j,0}y^{-1}\ti{\alpha}_{j}^{\vee}.
\end{align*}
By the induction hypothesis, $\wt(\bp')-\wt(\bq') \equiv h'$ 
for some $h' \in Q^{\vee}_{+}$. Then we have 
$\wt(\bp) - \wt(\bq) \equiv h' \in Q^{\vee}_{+}$, as desired. 

\paragraph{\bf{Case 2.}}
%
Assume that $\pair{\ti{\alpha}_{j}^{\vee}}{y\lambda} \le 0$.
Applying Lemma~\ref{lem:diamond1}\,(3) and (5) to 
$\bp$ (resp., $\bq$), we obtain a shortest directed path $\bp'$ 
(resp., a directed path $\bq'$) from $\mcr{s_{j}x}$ to $y$ such that
\begin{equation*}
\wt(\bp') \equiv 
 \wt(\bp) - \delta_{j,0}x^{-1}\ti{\alpha}_{j}^{\vee}, \qquad
\wt(\bq') \equiv 
 \wt(\bq) - \delta_{j,0}x^{-1}\ti{\alpha}_{j}^{\vee}.
\end{equation*}
By the induction hypothesis, $\wt(\bp')-\wt(\bq') \equiv h'$ 
for some $h' \in Q^{\vee}_{+}$. Then we have 
$\wt(\bp) - \wt(\bq) \equiv h' \in Q^{\vee}_{+}$, as desired. 
Thus we have proved the first assertion.

Next, assume that $\bq$ is also shortest. 
By interchanging $\bp$ and $\bq$ in the first assertion, 
we see that there exists $h'' \in Q^{\vee}_{+}$ such that 
$\wt (\bp) - \wt (\bq) \equiv h'' \mod Q_{J}^{\vee}$, which implies that 
$\wt (\bq) \equiv \wt (\bp) \mod Q_{J}^{\vee}$. 
Thus we have proved the proposition.
\end{proof}


\end{document}